\documentclass[11pt]{article}
\usepackage{amsmath,amsthm,amssymb}
\usepackage[usenames,dvipsnames]{xcolor}
\usepackage{enumerate}
\usepackage{graphicx}
\usepackage{cite}
\usepackage{comment}
\usepackage{oands}
\usepackage{tikz}
\usepackage{changepage}
\usepackage{bbm}
\usepackage{mathtools}
\usepackage[margin=1in]{geometry}
\usepackage[pagewise,mathlines]{lineno}
\usepackage{appendix}
\usepackage{stmaryrd} 
\usepackage{multicol}
\usepackage{microtype}
\usepackage[colorinlistoftodos]{todonotes}

\usepackage[normalem]{ulem}

\newcommand{\hf}{{\mathsf h}}

\usepackage[pdftitle={Weak Liouville quantum gravity metrics with matter central charge up to 25},
  pdfauthor={Joshua Pfeffer},
colorlinks=true,linkcolor=NavyBlue,urlcolor=RoyalBlue,citecolor=PineGreen,bookmarks=true,bookmarksopen=true,bookmarksopenlevel=2,unicode=true,linktocpage]{hyperref}

\theoremstyle{plain}
\newtheorem{thm}{Theorem}[section]

\newtheorem{lem}[thm]{Lemma}
\newtheorem{prop}[thm]{Proposition}

\def\@rst #1 #2other{#1}
\newcommand\MR[1]{\relax\ifhmode\unskip\spacefactor3000 \space\fi
\MRhref{\expandafter\@rst #1 other}{#1}}
\newcommand{\MRhref}[2]{\href{http://www.ams.org/mathscinet-getitem?mr=#1}{MR#2}}

\newcommand{\arxiv}[1]{\href{http://arxiv.org/abs/#1}{#1}}

\theoremstyle{definition}
\newtheorem{defn}[thm]{Definition}
\newtheorem{remark}[thm]{Remark}

\numberwithin{equation}{section}

\newcommand{\dsb}{\begin{adjustwidth}{2.5em}{0pt}
\begin{footnotesize}}
\newcommand{\dse}{\end{footnotesize}
\end{adjustwidth}}

\newcommand{\ssb}{\begin{adjustwidth}{2.5em}{0pt}}
\newcommand{\sse}{\end{adjustwidth}}

\newcommand{\aryb}{\begin{eqnarray*}}
\newcommand{\arye}{\end{eqnarray*}}
\def\alb#1\ale{\begin{align*}#1\end{align*}}
\def\allb#1\alle{\begin{align}#1\end{align}}
\newcommand{\eqb}{\begin{equation}}
\newcommand{\eqe}{\end{equation}}
\newcommand{\eqbn}{\begin{equation*}}
\newcommand{\eqen}{\end{equation*}}

\newcommand{\BB}{\mathbbm}
\newcommand{\ol}{\overline}

\newcommand{\op}{\operatorname}

\newcommand{\frk}{\mathfrak}
\newcommand{\eqD}{\overset{d}{=}}
\newcommand{\ep}{\varepsilon}
\newcommand{\rta}{\rightarrow}

\newcommand{\wt}{\widetilde}
\newcommand{\wh}{\widehat} 
\newcommand{\mcl}{\mathcal}

\newcommand{\rng}{\mathring}

\newcommand{\Var}{\operatorname{Var}}
\newcommand{\len}{\operatorname{len}}

\newcommand{\diam}{{\operatorname{diam}}}

\newcommand{\cc}{\op{\mathbf{c}}}
\newcommand{\xicrit}{\op{\xi_{\text{crit}}}}

\newcommand{\changes}[1]{{\color{black}{#1}}}

\newcommand{\Ninf}{\ol{\BB {N}}}

\let\originalleft\left
\let\originalright\right
\renewcommand{\left}{\mathopen{}\mathclose\bgroup\originalleft}
\renewcommand{\right}{\aftergroup\egroup\originalright}

\title{Weak Liouville quantum gravity metrics with matter central charge $\cc  \in (-\infty, 25)$
}
\author{ \begin{tabular}{c}{Joshua Pfeffer}\\[-4pt]\small Columbia\end{tabular}
}

\date{   }

\begin{document}

\maketitle 

\begin{abstract}
We define a random metric associated to Liouville quantum gravity (LQG) for all values of matter central charge $\mathbf{c} < 25$ by extending the axioms for a weak LQG metric from the $\mathbf{c} < 1$ setting.  We show that the axioms are satisfied by subsequential limits of Liouville first passage percolation; Ding and Gwynne (2020) showed these limits exist in a suitably chosen topology.  We show that, in contrast to the $\mathbf{c} < 1$ phase, the metrics for $\mathbf{c} \in (1,25)$ do not induce the Euclidean topology since they a.s.\ have a dense (measure zero) set of singular points, points at infinite distance from all other points.  We use this fact to prove that a.s.\ the metric ball is not compact and its boundary has infinite Hausdorff dimension.  On the other hand, we extend many fundamental properties of LQG metrics for $\mathbf{c} < 1$ to all $\mathbf{c} \in (-\infty,25)$, such as a version of the (geometric) Knizhnik-Polyakov-Zamolodchikov (KPZ) formula.
\end{abstract}

\tableofcontents

\section{Introduction}
\label{sec-intro}

This paper studies a universal family of random fractal surfaces called \textit{Liouville quantum gravity} (LQG).  The theory of LQG is a subject of active study in the probability community for its applications to conformal field theory, string theory, and other areas of mathematical physics, as well as for its links to the geometry of many classes of random planar maps.  LQG was first described heuristically by Polyakov and other physicists in the 1980s; the following is an early formulation of LQG.  

\begin{defn}[Heuristic formulation of LQG]
\label{defn-heuristic}
An \emph{LQG surface with matter central charge $\cc \in (-\infty,1]$}\footnote{The matter central charge is often denoted $\op{\mathbf{c}_M}$ to distinguish it from the \emph{Liouville} central charge $\op{\mathbf{c}_L} = 26 - \op{\mathbf{c}_M}$. Since we do not use the Liouville central charge in this paper, we denote the matter central charge simply by $\cc$.} and a specified topology (such as the sphere, disk, or torus) is a random surface whose law is given by ``the uniform measure on the space of surfaces with this topology'' weighted by the $(-\cc/2)$-th power of the determinant of the Laplace-Beltrami operator of the surface. \end{defn}

In this work, we restrict to LQG with simply connected topology; see~\cite{drv-torus, remy-annulus, grv-higher-genus} for discussion of LQG surfaces that are not simply connected.

The physicists who introduced the theory of LQG believed that a notion of LQG should exist for all real values of $\cc$, and in studying LQG, they were particularly motivated by the regime $\cc \in (1,25)$. 
Yet, their formulations of LQG do not extend to values of $\cc$ greater than $1$.   Physicists have observed~\cite{cates-branched-polymer,david-c>1-barrier,adjt-c-ge1,ckr-c-ge1,bh-c-ge1-matrix,dfj-critical-behavior,bj-potts-sim,adf-critical-dimensions} through numerical simulations and heuristics that, when $\cc \in (1,25)$, Definition~\ref{defn-heuristic}  describes the geometry of a branched polymer.  In mathematical terms, if we model the geometry of Definition~\ref{defn-heuristic}  by a random planar map with   a fixed large number $n$ of edges, sampled with probability proportional to the $(-\cc/2)$-th power of the determinant of its discrete Laplacian, we expect that  this random planar map model converges as $n \rta \infty$ to an object similar to the continuum random tree (CRT) defined by~\cite{aldous-crt1,aldous-crt2,aldous-crt3}.  This geometry does not depend on $\cc$ and is degenerate from a conformal field theory perspective since it is a tree and not a surface.  Physicists have tried to associate a nontrivial geometry to LQG for $\cc \in (1,25)$ despite these apparent obstacles, but have not been successful.  See~\cite{ambjorn-remarks} for a survey of these works.

In this paper, we study LQG for $\cc \in (1,25)$ from a rigorous perspective.   In recent years, probabilists have translated the heuristic of Definition~\ref{defn-heuristic} into a rigorous theory of LQG for $\cc \in (-\infty,1)$, constructing both an LQG random measure and an LQG random metric. Their constructions are based on the \emph{DDK ansatz} of David~\cite{david-conformal-gauge} and Distler-Kawai~\cite{dk-qg}, which heuristically describes an LQG surface as  a random two-dimensional Riemannian manifold whose Riemannian metric tensor is obtained by exponentiating a multiple of a random distribution $h$ called the \emph{Gaussian free field} (GFF).  (See~\cite{shef-gff},~\cite{berestycki-gmt-elementary},~\cite{pw-gff-notes} for a detailed introduction to the GFF.)  
The LQG measure is constructed  by exponentiating a regularized version of $\gamma h$, where $\gamma = \gamma(\cc)$ is the \emph{coupling constant} (see Figure~\ref{fig-c-Q-gamma-table}).\footnote{The construction of this measure in~\cite{shef-kpz} is a special case of a general theory of regularized random measures known as \emph{Gaussian multiplicative chaos} (GMC), first introduced in~\cite{kahane}.  See~\cite{rhodes-vargas-review,berestycki-gmt-elementary,aru-gmc-survey} for reviews of this theory.} When $\cc \in (1,25)$, this parameter $\gamma$ is complex, which suggests that the LQG random measure does not extend to a real measure for $\cc \in (1,25)$.  
In contrast, the definition of the LQG metric suggests that it could possibly be defined for $\cc \in (-\infty,25)$.  The definition of the metric for \changes{$\cc \in (-\infty, 1)$} involves  the \emph{background charge} $Q$ (see Figure~\ref{fig-c-Q-gamma-table}), which is real and nonzero for $\cc \in (-\infty,25)$.

The idea of defining a metric associated to LQG for $\cc \in (1,25)$ was first explored in~\cite{ghpr-central-charge}, which constructs a heuristic model for such a metric by considering a natural discretization of LQG in terms of dyadic squares that makes sense for all $\changes{\cc \in (-\infty,25)}$. 
This paper, building on the previous work~\cite{dg-supercritical-lfpp}, \emph{rigorously} proves the existence of a metric that satisfies a collection of axioms that an LQG metric for $\cc \in (-\infty,25)$ should satisfy.  We then apply these axioms to illustrate the key distinguishing features of LQG in this regime and provide a framework for proving results about LQG for $\cc \in (-\infty,25)$. \changes{The results of this paper are used to prove in~\cite{dg-uniqueness} that the axioms we state in this paper \emph{uniquely} characterize a metric associated to LQG for $\cc \in (-\infty,25)$.}

\begin{figure}[t!]
 \begin{center}
\includegraphics[scale=1.1]{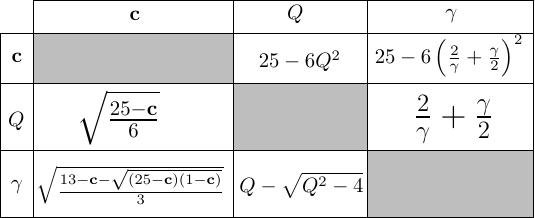}
\vspace{-0.01\textheight}
\caption{Table of relationships between the values of the matter central charge $\cc$, the background charge $Q$, and the coupling constant $\gamma$. When $\changes{\cc \in (-\infty,1]}$, all three parameters are real.  When $\cc \in (1,25)$, $\gamma$ is complex, but $Q$ is real and nonzero.
}\label{fig-c-Q-gamma-table}
\end{center}
\vspace{-1em}
\end{figure}

Before stating our results, we review the construction of the LQG random metric in the classical regime $\changes{\cc \in (-\infty, 1)}$. The DDK ansatz suggests constructing a random metric associated to LQG as a 
limit of regularized versions of the heuristic metric
\[
(z,w) \mapsto  \inf_{P : z\rta w} \int_0^1 e^{\xi h(P(t))} |P'(t)| \,dt
\]
for some $\xi = \xi(\cc)$, where the infimum is over all paths from $z$ to $w$.  
To make this idea precise, we first mollify  the field $h$ by the heat kernel; i.e., we set
\eqb \label{eqn-gff-convolve}
h_\ep^*(z) := (h *p_{\ep^2/2})(z) = \int_{\changes{\BB{C}}} h(w) p_{\ep^2/2}(z,w) \, dw,
\eqe
where $p_s(z,w) =   \frac{1}{2\pi s} \exp\left( - \frac{|z-w|^2}{2s} \right) $ is the heat kernel on $\BB C$ and the integral~\eqref{eqn-gff-convolve} is interpreted in the sense of distributional pairing.  We then define the following regularized metric associated to LQG.

\begin{defn}
Let $\xi > 0$. We define the \emph{$\ep$-Liouville first passage percolation (LFPP) metric} with parameter $\xi$ associated to the field $h$ as
\eqb \label{eqn-lfpp}
D_h^\ep(z,w) := \inf_{P : z\rta w} \int_0^1 e^{\xi h_\ep^*(P(t))} |P'(t)| \,dt 
\eqe
where the infimum is over all piecewise continuously differentiable paths from $z$ to $w$.  We define the \emph{rescaled $\ep$-LFPP metric} as 
\eqb
\frk a_\ep^{-1} D_h^{\ep},
\label{eqn-defn-rescaled-LFPP}
\eqe
where $\frk a_\ep$ denote the median of the $D_h^\ep$-distance between the left and right boundaries of the unit square along paths which stay in the unit square.  
\end{defn}

It is shown in~\cite[Theorem 1.5]{dg-lqg-dim} and~\cite[Proposition 1.1]{dg-supercritical-lfpp} that, for every $\xi > 0$, there exists $Q = Q(\xi) > 0$ with 
\eqb
\frk a_\ep = \ep^{1-  \xi Q + o_\ep(1)} \qquad \text{as $\ep \rta 0$} \label{eqn-defn-Q(xi)}
\eqe
Furthermore, $\xi \rta Q(\xi)$ is continuous and nonincreasing on $(0,\infty)$, and tends to $0$ as $\xi \rta \infty$.  With $\xicrit$ chosen so that $Q(\xicrit) = 2$, we call LFPP for $\xi \in (0,\xicrit)$ the \emph{subcritical} phase and  LFPP for $\xi \in (\xicrit, \infty)$ the \emph{supercritical} phase.  Equivalently, if we associate LFPP with parameter $\xi$ to the value of matter central charge associated to background charge $Q(\xi)$ (see Table~\ref{fig-c-Q-gamma-table}), then the subcritical phase of LFPP corresponds to $\changes{\cc \in (-\infty, 1)}$, and the supercritical phase of LFPP corresponds to $\cc \in (1,25)$.

Probabilists used the LFPP approximation scheme in the subcritical phase to prove the {existence} and uniqueness of a random metric associated to LQG for $\changes{\cc \in (-\infty,1)}$.  They proved this in three stages:

\begin{itemize}
    \item \emph{Stage 1: constructing a candidate metric associated to LQG.}  First,~\cite{dddf-lfpp} constructed a candidate metric associated to LQG for $\changes{\cc \in (-\infty,1)}$ by showing that, in the subcritical phase, the laws of the rescaled $\ep$-LFPP metrics are tight in the local uniform topology.
    \item \emph{Stage 2: Showing that this candidate metric satisfies a collection of axioms that characterize the LQG metric.} Next,~\cite{lqg-metric-estimates} showed that every weak subsequential limit of  rescaled $\ep$-LFPP metrics in the subcritical phase can indeed be viewed as a random metric that describes an LQG surface.  Specifically,~\cite{lqg-metric-estimates} proved that, for each fixed $\xi \in (0,\xicrit)$, if $\cc \in (-\infty,1)$ is chosen so that $\xi = \xi(\cc)$, every subsequential limit of  rescaled $\ep$-LFPP metrics satisfies a collection of axioms that a $\cc$-LQG metric should satisfy.  These axioms also provide a framework for proving important results about these metrics, such as H\"older continuity and estimates for distances.  By applying these estimates,~\cite{gp-kpz} showed that, for every metric $D_h$ satisfying these axioms, and every set $X \subset \BB{C}$ independent of the field $h$, the Hausdorff dimensions of $X$ with respect to the Euclidean metric and the metric $D_h$ are related by the famous \emph{KPZ formula}~\eqref{eqn-old-kpz}.  (We define Hausdorff dimension in Definition~\ref{defn-dim}.) Also, the Hausdorff dimension of a metric satisfying the axioms in~\cite{lqg-metric-estimates} is equal to the ratio $\xi/\gamma$.
    \item 
    \emph{Stage 3: Proving that the axioms formulated in the previous stage uniquely determine the LQG metric.} Finally, the work~\cite{gm-uniqueness} subsequently showed that the axioms stated in~\cite{lqg-metric-estimates} characterize a \emph{unique} metric (up to global rescalings).  A key input in this uniqueness proof is the work~\cite{gm-confluence}, which shows that a metric satisfying the axioms in~\cite{lqg-metric-estimates} satisfies a \emph{confluence of geodesics} property.
\end{itemize}

In the more general setting $\cc \in (-\infty,25)$, a natural approach for constructing a candidate random metric is to consider the rescaled $\ep$-LFPP metrics for all $\xi > 0$. 
The recent work~\cite{dg-supercritical-lfpp} proved that, for all $\xi > 0$, the laws of the rescaled $\ep$-LFPP metrics are tight with respect to the topology on lower semicontinuous functions introduced in~\cite{beer-usc}.  This topology is defined as follows.

\begin{defn}
We define the \emph{lower semicontinuous topology} as the topology on the space of lower semicontinuous functions from $\BB{C} \times \BB{C} \rta \BB{R} \cup \{\pm \infty\}$, in which a sequence $f_n$ converges to a function $f$ if and only if
\begin{enumerate}[(a)]
    \item For every sequence of points $(z_n,w_n) \in \BB{C} \times \BB{C}$ converging to a point $(z,w) \in \BB{C} \times \BB{C}$, we have $f(z,w) \leq \liminf_{n \rta \infty} f_n(z_n,w_n)$.
    \item For every point $(z,w) \in \BB{C} \times \BB{C}$, there exists a sequence of points $(z_n,w_n) \in \BB{C} \times \BB{C}$ converging to $(z,w)$ for which $f(z,w) = \lim_{n \rta \infty} f_n(z_n,w_n)$.
\end{enumerate}
\end{defn}

The reason we do not have tightness of the rescaled LFPP metrics in the local uniform topology is that, in the supercritical phase, the subsequential limiting metrics are not finite-valued or continuous on $\BB{C} \times \BB{C}$.  Rather, each subsequential limiting metric a.s.\ has an uncountable dense set of \emph{singular points}:

\begin{defn}
Let $D$ be a metric defined on a set $U \subset \BB{C}$.
We say that $z \in U$ is a  \emph{singular point} for $D$ if $D(z,w) = \infty$ for every $w \in U \backslash \{z\}$.
\label{defn-singular}
\end{defn}

The fact that the subsequential limiting metrics have an uncountable dense set of singular points in the supercritical phase is consistent with the work~\cite{ghpr-central-charge}, which observed an analogous phenomenon in their heuristic model of an LQG metric for $\cc \in (1,25)$.

\subsection{An axiomatic description of an LQG metric for \texorpdfstring{$\cc \in (-\infty,25)$}{matter central charge less than 25}}
    
Our first main result, which we prove in Section~\ref{sec-lfpp-local}, builds on the work of~\cite{dg-supercritical-lfpp} by showing that  subsequential limits of LFPP metrics satisfy a list of axioms that a $\cc$-LQG metric for $\cc \in (-\infty,25)$ would be expected to satisfy.  These two works together prove the existence of a metric associated to LQG for $\cc \in (-\infty,25)$.
To state these axioms, we will need the following elementary metric space definitions.

\begin{defn}
Let $(X,D)$ be a metric space. Note that, throughout this paper, we allow the distance between pairs of points with respect to a metric $D$ to be infinite.
\begin{itemize}
\item
For a curve $P : [a,b] \rta X$, the \emph{$D$-length} of $P$ is defined by 
\eqbn
\op{len}\left( P ; D  \right) := \sup_{T}
\sum_{i=1}^{\# T} D(P(t_i) , P(t_{i-1})) 
\eqen
where the supremum is over all partitions $T : a= t_0 < \dots < t_{\# T} = b$ of $[a,b]$. Note that the $D$-length of a curve may be infinite.
\item
\changes{We say $(X,D)$  is a \emph{length space} if the $D$-distance between any two points in $X$ is given by the infimum of the $D$-lengths of all paths in $X$ between the two points.}
\item
For $Y\subset X$, the \emph{internal metric of $D$ on $Y$} is defined by
\eqb \label{eqn-internal-def}
D(x,y ; Y)  := \inf_{P \subset Y} \op{len}\left(P ; D \right) ,\quad \forall x,y\in Y 
\eqe 
where the infimum is over all paths $P$ in $Y$ from $x$ to $y$. 
Note that $D(\cdot,\cdot ; Y)$ is a metric on $Y$, except that it is allowed to take infinite values.  
\item
If $\BB{A}$ is \changes{a Euclidean annulus}, we define the $D$-distance \emph{across} $\BB{A}$ as the distance between the inner and outer boundaries of $\BB{A}$, and the $D$-distance \emph{around} $\BB{A}$ as the infimum of the $D$-distances of closed paths that separate the inner and outer boundaries of $\BB{A}$.
\item \changes{
If $f:\BB{R} \rta \BB{C}$ is a continuous function, then we define the metric $e^f \cdot D$ as
\eqb
\label{eqn-metric-f}
(e^{f} \cdot D) (z,w) := \inf_{P : z\rta w} \int_0^{\op{len}(P ; D)} e^{f(P(t))} \,dt , \quad \forall z,w\in \BB C ,
\eqe
where the infimum in~\eqref{eqn-metric-f} is over all  paths $P$ from $z$ to $w$ parametrized by $D$-length.
}
\end{itemize}
\end{defn}

\changes{Also, we let $h_r(0)$ be the average of $h$ over the circle $\partial B_r(0)$ (see~\cite[Section 3.1]{shef-kpz} for a more detailed explanation of the \emph{circle average process} associated to $h$).}

\begin{defn}
Let $\mcl D'(\BB C)$ be the space of distributions (generalized functions) on $\BB C$, equipped with the usual weak topology.
For each $\xi > 0$, we define a \emph{weak LQG metric with parameter $\xi $} as a measurable function $h\mapsto D_h$ from $\mcl D'(\BB C)$ to the space of lower semicontinuous metrics on $\BB C$ such that the following is true whenever $h$ is a whole-plane GFF plus a continuous function.
\begin{enumerate}[I.] 
\item \textbf{Length space.} Almost surely, $(\BB C , D_h)$ is a length space. \label{item-metric-length}
\item \textbf{Locality.} For each deterministic open set  $U\subset\BB C$, the $D_h$-internal metric $D_h(\cdot,\cdot ; U)$ is determined a.s.\ by $h|_U$.  \label{item-metric-local}  
\item \textbf{Weyl scaling.} If $f : \BB C\rta \BB R$ is a continuous function, then a.s.\ $ D_{h+f} = e^{\xi f} \cdot D_h$. 
\label{item-metric-f}
\item \textbf{Translation invariance.} For each deterministic point $z \in \BB C$, a.s.\ $D_{h(\cdot + z)} = D_h(\cdot+ z , \cdot+z)$.  \label{item-metric-translate}
\item \textbf{Tightness across scales.} For each  $r>0$, there is a deterministic constant $\frk c_r$ such that, if $\BB{A}$ is a fixed annulus, the laws of the distances $\frk c_r^{-1} e^{-\xi h_r(0)} D_h ( \changes{\text{across $r \BB{A}$}})$ and $\frk c_r^{-1} e^{-\xi h_r(0)} D_h ( \changes{\text{around $r \BB{A}$}})$ and their inverses for  $r > 0$ \changes{are} tight. 
Moreover, there exists  \label{item-metric-coord} 
$\Lambda > 1$ such that for each $\delta \in (0,1)$, 
\eqb \label{eqn-scaling-constant}
\Lambda^{-1} \delta^\Lambda \leq \frac{\frk c_{\delta r}}{\frk c_r} \leq \Lambda \delta^{-\Lambda} ,\quad\forall r  \changes{\in \BB{Q} \cap (0,\infty)}.
\eqe
\end{enumerate}
\label{defn-weak-metric}
\end{defn}

In analogy with~\cite{lqg-metric-estimates} in the $\xi < \xicrit$ setting, we call the metric in Definition~\ref{defn-weak-metric} a \emph{weak} LQG metric because we expect the LQG metric to satisfy one additional property that we have not stipulated as an axiom.  Namely, if $h$ is a field on $U \subset \BB{C}$, $\phi: \wt U \rta U$ is a conformal mapping and $\wt h:= h \circ \phi + Q \log |\phi'|$, then the LQG metric $D_{\wt h}$ associated to $\wt h$ on $\wt U$ should equal the pullback by $\phi$ of the LQG metric $D_h$ associated to $h$ on $U$.  We call this property \emph{conformal covariance} of the metric.  The reason we have not included the conformal covariance property as an axiom in Definition~\ref{defn-weak-metric}  is that it is far from obvious that the property is satisfied by subsequential limits of LFPP.  The problem is that, if we fix a subsequence $\ep_n \rta 0$ for which the rescaled $\ep_n$-LFPP metrics converge, then applying a conformal mapping $\phi$ rescales space and therefore changes the LFPP parameters $\ep_n$ that index the subsequence. As in the $\xi < \xicrit$ phase, we expect that every weak LQG metric satisfies the conformal covariance property, \changes{and that one can deduce this property from the recent work~\cite{dg-uniqueness} proving the uniqueness of the metric defined in Definition~\ref{defn-weak-metric}.}
Axiom~\ref{item-metric-coord} is a weaker substitute for an  exact conformal \changes{covariance} property. (Axiom~\ref{item-metric-coord}  may be deduced from the conformal \changes{covariance} property combined with Axiom~\ref{item-metric-f}.)  Nonetheless, Axiom~\ref{item-metric-coord} can be applied in place of exact conformal covariance in many settings.  Essentially, Axiom~\ref{item-metric-coord}  allows us to compare distance quantities at the same Euclidean scale; see, e.g., Lemma~\ref{lem-tight} below.

Our first main result asserts that subsequential limits of LFPP metrics yield weak LQG metrics in the sense of Definition~\ref{defn-weak-metric}.

\begin{thm} \label{thm-lfpp-axioms}
For every sequence of values of $\ep$ tending to zero, there is a weak $\xi$-LQG metric $D$ (with the constants $\frk c_r$ given by~\eqref{eqn-defn-cr}) and a subsequence $\{\ep_n\}_{n\in\BB N}$ for which the following is true. Suppose that $h$ is a whole-plane GFF plus a bounded continuous function\changes{, and as in Definition~\ref{defn-weak-metric}, let $D_h = D(h)$.}
\begin{itemize}
    \item The re-scaled $\ep_n$-LFPP metrics $\frk a_{\ep_n}^{-1} D_h^{\ep_n}$ defined in~\eqref{eqn-defn-rescaled-LFPP} converge a.s.\ as $n \rta \infty$ to $D_h$ in the lower semicontinuous topology.
    \item 
Let $O,O'$ be \emph{rational circles}, i.e., circles with centers in $\BB{Q}^2$ and rational radii.  (We allow the radius to equal zero, in which case the circle is a rational point.) Then a.s.\ $\frk a_{\ep_n}^{-1} D_h^{\ep_n}(O,O') \rta D_h(O,O')$.
\end{itemize}
Moreover, the constant $\frk c_r$ in Axiom~\ref{item-metric-coord} corresponding to this weak $\xi$-LQG metric $D$ \changes{can be chosen to satisfy}
\eqb
\frk c_r := \lim_{\ep \rta 0} \frac{ r \frk a_{\ep/r}}{\frk a_{\ep}} \qquad \forall r \in \BB{Q}.
\label{eqn-defn-cr}
\eqe
\end{thm}

We observe that, in particular, Theorem~\ref{thm-lfpp-axioms} proves the existence of a metric satisfying the axioms in Definition~\ref{defn-weak-metric} in the critical case $\cc = 1$; this special value of $\cc$ was not treated in the previous works~\cite{lqg-metric-estimates, gm-confluence, gm-uniqueness, gm-coord-change} on the LQG metric. 

\changes{We note that the reason we consider distances between rational circles in Theorem~\ref{thm-lfpp-axioms}---and not just distances between rational points---is that if a length metric $D$ is lower semicontinuous, then $D$ is \emph{not} determined by  the $D$-distances between all pairs of rational points (as it would be in the continuous setting), but $D$ \emph{is}  determined by  the $D$-distances between all pairs of rational circles.}\footnote{\label{footnote-rational}\changes{Given a pair of points $z,w$, let $O^n_z$ and $O^n_w$ are sequences of rational circles surrounding $z$ and $w$ and shrinking to $z$ and $w$, respectively. Then $D(z,w) \geq D(O^n_z,O^n_w)$ since $D$ is a length metric, and $
D(z,w) \leq \liminf_{n \rta \infty} D(O^n_z,O^n_w)$ since $D$ is lower semicontinuous; hence, $
D(z,w) = \liminf_{n \rta \infty} D(O^n_z,O^n_w)$.}}

We prove Theorem~\ref{thm-lfpp-axioms} in Section~\ref{sec-lfpp-local} of this paper.  The proof of Theorem~\ref{thm-lfpp-axioms} uses many of the techniques in the works~\cite{df-lqg-metric, lqg-metric-estimates, local-metrics}, but it also requires some nontrivial ideas due to the unique challenges that arise in the $\cc \in (1,25)$ regime. We explain this further in Section~\ref{sec-thm-outline}, which outlines the proof of Theorem~\ref{thm-lfpp-axioms}  in detail.  One consequence of our analysis that is worth highlighting is Lemma~\ref{lem-avoids-thick}, which asserts that geodesics  a.s.\ do not contain any thick points with thickness $>\zeta$ for some deterministic (but nonexplicit) $\zeta < Q$. Roughly speaking, this means that geodesics do not spend too much time traversing small Euclidean neighborhoods.

\subsection{Properties of weak LQG metrics for \texorpdfstring{$\cc \in (-\infty,25)$}{matter central charge less than 25}}
\label{sec-intro-properties}

In Section~\ref{sec-properties}, we demonstrate the power of the axioms stated in Definition~\ref{defn-weak-metric} by proving properties of weak LQG metrics.  We show that many of the properties that characterize LQG metrics for $\changes{\cc \in (-\infty,1)}$ extend to the range $\cc \in (-\infty,25)$.  We also demonstrate that weak LQG metrics for $\cc \in (1,25)$ have several distinguishing features that contrast sharply with the $\cc \in (-\infty,1)$ regime.
In all the results that we state here and in Section~\ref{sec-intro-kpz}, we let $h$ be a whole-plane GFF plus a continuous function (unless specified otherwise), and we let $D_h$ be the metric associated to $h$ by a weak LQG metric $h \mapsto D_h$.  

First, we prove the following estimate for distances between sets, which extends~\cite[Proposition 3.1]{lqg-metric-estimates} to general $\xi > 0$.

\begin{prop}[Distance between sets] \label{prop-two-set-dist}
Let $h$ be a whole-plane GFF normalized so that $h_1(0) = 0$.
Let $U \subset \BB C$ be an open set (possibly all of $\BB C$) and let $K_1,K_2\subset U$ be connected, disjoint compact sets which are not singletons. 
For each $\BB r  >0$, it holds with superpolynomially high probability as $A\rta \infty$, at a rate which is uniform in the choice of $\BB r$, that 
\eqb \label{eqn-two-set-dist}
 A^{-1}\frk c_{\BB r} e^{\xi h_{\BB r}(0)} \leq D_h(\BB r K_1,\BB r K_2 ; \BB r U) \leq A \frk c_{\BB r} e^{\xi h_{\BB r}(0)} .  
\eqe
\end{prop}

Second, we strengthen the bounds~\eqref{eqn-scaling-constant} for the scaling constant $\frk c_r$ associated to a weak LQG metric.

\begin{prop} \label{prop-metric-scaling}
Let $Q = Q(\xi)$ be as in~\eqref{eqn-defn-Q(xi)}. 
Then for $r>0$, the scaling constants $\frk c_r$ in Axiom~\ref{item-metric-coord} satisfy
\eqb \label{eqn-metric-scaling}
\frac{\frk c_{\delta  r } }{ \frk c_{r} } = \delta^{\xi Q + o_\delta(1)} \quad \text{as $\delta \rta 0$},
\eqe
at a rate which is uniform over all $r>0$. 
\end{prop}

In the special case of subsequential limits of rescaled LFPP metrics,  Proposition~\ref{prop-metric-scaling} follows immediately from~\eqref{eqn-defn-Q(xi)} and~\eqref{eqn-defn-cr}.  Note that Proposition~\ref{prop-metric-scaling} associates \emph{every} weak LQG metric with parameter $\xi$ to the value of background charge $Q(\xi)$ defined in~\eqref{eqn-defn-Q(xi)}.  We already described this correspondence in the setting of LFPP metrics.  In particular, the phases $\xi \in (0,\xicrit)$ and $\xi > \xicrit$ correspond to the ranges $\changes{\cc \in (-\infty,1)}$ and $\cc \in (1,25)$, respectively.

Next, it was proved in~\cite[Theorem 1.7]{lqg-metric-estimates} that weak LQG metrics for $\xi < \xicrit$ satisfy a H\"older continuity property.  It was shown in~\cite[Proposition 5.20]{dg-supercritical-lfpp} that  subsequential limits of the rescaled $\ep$-LFPP metrics~\eqref{eqn-defn-rescaled-LFPP} satisfy a weaker version of this property. We can generalize the latter result to weak LQG metrics for all $\xi > 0$.

\begin{prop}[H\"older continuity] \label{prop-holder}
Almost surely, the identity map from $\BB{C}$, equipped with the metric $D_h$, to $\BB{C}$, equipped with the Euclidean metric, is locally H\"older continuous with any exponent less than $[\xi(Q+2)]^{-1}$.
\end{prop}

In the $\changes{\cc \in (-\infty,1)}$ regime,~\cite[Theorem 1.7]{lqg-metric-estimates} asserts that the inverse map of the map in Proposition~\ref{prop-holder} is a.s.\ locally H\"older continuous with any exponent smaller than $\xi(Q-2)$.  For $\cc \in (1,25)$, however, this inverse map is not continuous, since a.s.\ the metric $D_h$ has \emph{singular points} (Definition~\ref{defn-singular}).
The following proposition characterizes the set of singular points of the metric $D_h$ when $Q<2$.

\begin{prop}
\label{prop-singular}
Almost surely, the following is true.
\begin{enumerate}[(a)]
    \item 
Every point $z \in \BB{C}$ for which 
\eqb
\limsup_{r \rta 0} \frac{h_r(z)}{\log{r^{-1}}}
\label{eqn-limsup-Q}
\eqe
is greater than $Q$ is a singular point.  In particular, the set of singular points is dense when $Q<2$ (i.e., $\xi > \xicrit$).
\item Conversely,   every point $z \in \BB{C}$ for which~\eqref{eqn-limsup-Q} is less than $Q$ is not a singular point. 
\end{enumerate}
\end{prop}

\changes{We note that we do not determine whether points $z$ for which~\eqref{eqn-limsup-Q} is equal to $Q$ are singular.  In particular,} our estimates are not precise enough to determine whether a weak LQG metric with $\cc = 1$ has singular points.

Next, we show that, a.s.\ on the complement of the set of singular points, the metric $D_h$ is complete and finite-valued, every pair of points can be joined by a geodesic, and the set of rational points are $D_h$-dense.

\begin{prop}\label{prop-complete}
Almost surely, the metric $D_h$ is complete and finite-valued on $\BB{C} \backslash \{\text{singular points}\}$, and every pair of points in  $\BB{C} \backslash \{\text{singular points}\}$ can be joined by a geodesic.
\end{prop}

\begin{prop}
\label{prop-rational-dense}
Almost surely, the set $\BB{Q}^2$ is $D_h$-dense in $\BB{C} \backslash \{\text{singular points}\}$. Moreover, it is a.s.\ the case that for every $\ep > 0$ and every nonsingular point $z \in \BB{C}$, we can choose a loop $\mcl L = \mcl L(z,\ep)$ surrounding $z$ whose $D_h$-length and Euclidean diameter are less than $\ep$, and such that  $D_h(z,\mcl L) < \ep$.
\end{prop}

Finally, we prove several properties of $D_h$-balls  for $\cc \in (1,25)$.  (See Section~\ref{sec-notation} for definitions of the notation used in the proposition that follows.) When $\changes{\cc \in (-\infty,1)}$, the LQG metric ball is compact and its boundary has Hausdorff dimension $d_{\cc} - 1$ with respect to the metric $D_h$, where $d_{\cc}$ is the dimension of $\BB{C}$ in the metric $D_h$~\cite{gwynne-ball-bdy, lqg-zero-one}.  The situation is starkly different for $\cc \in (1,25)$:

\begin{prop}
\label{prop-compact}
Let $\cc \in (1,25)$, i.e., $\xi > \xicrit$. Almost surely, for every $r>0$ and every Borel set $X \subset \BB{C}$ \changes{that does not contain any singular points}, the following is true. \changes{Let $B_r(X;D)$ (resp. $B_r[X;D]$) denote the set of points in $\BB{C}$ whose $D$-distance from $X$ is less than $r$ (resp. at most $r$).  }
\begin{enumerate}
\item The set $B_{r}[X;D_h]$ has infinitely many complementary connected components \changes{(in the Euclidean sense)}. \label{item-compact-prop-a}
\item  The set $B_r[X;D_h]$ is not compact in $(\BB{C},D_h)$.  In fact, for every $s<r$,  $B_r[X;D_h]$  cannot be covered by finitely many $D_h$-balls of radius $s$.
\label{item-compact-prop-b}
\item The $D_h$-boundary of the set $B_{r}(X;D_h)$ has infinite Hausdorff dimension in the metric $D_h$. \label{item-compact-prop-c}
\end{enumerate}
\end{prop}

\changes{
The properties of weak LQG metrics that we prove in this paper have been used in several subsequent works~\cite{dg-confluence,dg-uniqueness,dg-polylog,dg-critical-lqg} to establish uniqueness of the supercritical LQG metric and many  important results about this metric (such as the confluence of geodesics property).}

\subsection{A KPZ formula for \texorpdfstring{$\cc \in (-\infty,25)$}{matter central charge less than 25}}
\label{sec-intro-kpz}

Finally, in Section~\ref{sec-kpz}, we extend to the range $\cc \in (-\infty,25)$ a version of the  (geometric) \emph{Knizhnik-Polyakov-Zamolodchikov (KPZ) formula}, a fundamental result in the theory of LQG that relates the Euclidean and  quantum dimensions of a fractal set sampled independently from $h$.  There are many mathematical formulations of this relation, stemming from different rigorous formulations of the notion of the ``quantum dimension'' of a fractal set.\footnote{The first rigorous versions of the KPZ formula were proven by Duplantier and Sheffield~\cite{shef-kpz} and Rhodes and Vargas~\cite{rhodes-vargas-log-kpz}.  Other versions of the KPZ formula have since been established in the works~\cite{shef-kpz,aru-kpz,ghs-dist-exponent,ghm-kpz,grv-kpz,gwynne-miller-char,bjrv-gmt-duality,benjamini-schramm-cascades,wedges,shef-renormalization,shef-kpz}.}  After the metric $D_h$ associated to LQG was rigorously constructed for all $\changes{\cc \in (-\infty,1)}$, the work~\cite{gp-kpz} was able to state the most natural formulation of the KPZ relation for $\changes{\cc \in (-\infty,1)}$: namely, a relation between the Hausdorff dimension of a fractal set in the $D_h$ metric and its Hausdorff dimension in the Euclidean metric.  
Their formula~\cite[Theorem 1.4]{gp-kpz} states that, if $X \subset \BB{C}$ is a deterministic Borel set or a random Borel set independent from $h$, then a.s.\ (using the notation in Definition~\ref{defn-dim}) we have
\eqb
\label{eqn-old-kpz}
\dim_{\mcl H}  X =  \xi Q \dim_{\mcl H}(X;D_h)  - \frac{\xi^2}{2}  (\dim_{\mcl H}  (X;D_h))^2.
\eqe
Equivalently,
\eqb
\dim_{\mcl H}(X;D_h) = \xi^{-1} (Q - \sqrt{Q^2 - 2 \dim_{\mcl H} X}).
\label{eqn-old-kpz-2}
\eqe
We prove a version of the KPZ formula for $\cc \in (-\infty,25)$ in terms of two dual notions of dimension, namely, Hausdorff dimension and packing dimension.   (See Definition~\ref{defn-dim} for their definitions, and for the notation we use in the theorem statement.)

\begin{thm}[KPZ formula]
\label{thm-kpz}
Let $X \subset \BB{C}$ be a deterministic Borel set or a random Borel set independent from $h$.  Then a.s.\
\changes{\eqb
\label{kpz-independent-1}
 \dim_{\mcl H}(X;D_h) \geq 
 \begin{cases*}
   \xi^{-1}(Q - \sqrt{Q^2 - 2\dim_{\mcl H} X}) & \text{if $\dim_{\mcl H} X \leq Q^2/2$} \\
   \\ \infty & \text{if $\dim_{\mcl H} X > Q^2/2$}
    \end{cases*} 
\eqe
and
\eqb
\label{kpz-independent-2}
\dim_{\mcl H}(X;D_h) \leq 
 \begin{cases*}
   \xi^{-1}(Q - \sqrt{Q^2 - 2\dim_{\mcl P} X}) & \text{if $\dim_{\mcl P} X < Q^2/2$} \\
   \\ \infty & \text{if $\dim_{\mcl P} X \geq Q^2/2$}
    \end{cases*} 
\eqe}
\end{thm}

Observe that, for $\cc \in (1,25)$,  the function $f: [0,2] \rta \BB{R} \cup \{\infty\}$ is not one-to-one. This is why the KPZ formula in Theorem~\ref{thm-kpz} takes the form of~\eqref{eqn-old-kpz-2} rather than~\eqref{eqn-old-kpz}.  

Unlike the KPZ formula~\cite[Theorem 1.4]{gp-kpz} for $\changes{\cc \in (-\infty,1)}$, our KPZ formula is \emph{not} an equality: the upper bound is expressed in terms of packing dimension, while the lower bound is in terms of Hausdorff dimension.  We explain the reason for this discrepancy in the beginning of Section~\ref{sec-kpz}.
It is not clear to us whether the stronger version of the KPZ formula, with packing dimension replaced by Hausdorff dimension, is even true.  Nonetheless, we note that the two notions of dimensions are equal for many of the sets associated to LQG (and independent from the metric) that we are interested in, such as SLE-type sets~\cite{schramm-sle, beffara-dim}.  %We also remark that the reason that the upper bound in~\eqref{kpz-independent} is the right limit $f(\dim_{\mcl P} X +)$ instead of just $f(\dim_{\mcl P} X)$ is that, in proving the upper bound, we must assume that $\dim_{\mcl P} X$ is \emph{strictly} greater than $Q^2/2$.

We also prove a KPZ-type result for the dimension of the intersection of a fractal $X$  with the set $T_h^\alpha$ of \emph{$\alpha$-thick points} of $h$, defined  in~\cite{hmp-thick-pts}  as
\eqb \label{eqn-thick-pts}
\mcl T_h^\alpha := \left\{ z\in U : \lim_{\ep\rta 0} \frac{h_\ep(z)}{\log\ep^{-1}} = \alpha \right\},
\eqe 
where $h_\ep(z)$ denotes the average of $h$ on the circle of radius $\ep$ centered at $z$.
Specifically, we derive an exact relation, which holds almost surely for all Borel sets $X \subset \BB{C}$, between the Euclidean Hausdorff dimension of $X \cap \mcl T_h^\alpha$ and the Hausdorff dimension of $X \cap \mcl T_h^\alpha$ with respect to the metric $D_h$.

\begin{thm}
\label{thm-kpz-thick}
If $\alpha \in [-2,Q)$, then almost surely, for every Borel set $X \subset \BB{C}$,
\[
\dim_{\mcl H} \left( X \cap \mcl T_h^\alpha ; D_h \right) = \frac{1}{\xi(Q - \alpha)} \dim_{\mcl H}\left( X \cap \mcl T_h^\alpha  \right).
\]
If $\alpha \in (Q,2]$, then almost surely, for every Borel set $X \subset \BB{C}$,
\[
\dim_{\mcl H} \left( X \cap \mcl T_h^\alpha ; D_h \right) = 
\begin{cases*}
    0 &  \text{if  $X \cap  T_h^\alpha$ is countable}  \\
    \infty & \text{otherwise}
    \end{cases*}
\]
\end{thm}

We observe that Theorem~\ref{thm-kpz-thick} allows us to consider sets $X$ that depend on the field $h$, such as $D_h$-geodesics and $D_h$-metric balls.  This strengthens the analogous result~\cite[Theorem 1.5]{gp-kpz} for $\changes{\cc \in (-\infty,1)}$, since they required the Borel set $X$ to be deterministic or independent of $h$.  

In the special case in which $X$ is a deterministic Borel set or a random Borel set independent from $h$, it is known\footnote{The result~\cite[Theorem 4.1]{ghm-kpz} is stated for  a zero-boundary GFF; the statement for a whole-plane GFF follows from local absolute continuity.}~\cite[Theorem 4.1]{ghm-kpz} that the Euclidean Hausdorff dimension of $X\cap \mcl T_h^\alpha$ is given by
\eqb
\label{eqn-thick-dim}
\dim_{\mcl H}\left( X\cap \mcl T_h^\alpha \right) = \max\left\{ \dim_{\mcl H} X - \frac{\alpha^2}{2}  , 0 \right\} \qquad a.s.
\eqe 
Thus, for such sets $X$, Theorem~\ref{thm-kpz-thick} gives the explicit value of $\dim_{\mcl H} \left( X \cap \mcl T_h^\alpha ; D_h \right)$ for each $\alpha \in [-2,Q)$ in terms of $\dim_{\mcl H} X$.
\medskip

\noindent
\textbf{Acknowledgments.}
We thank E.\ Gwynne and J.\ Ding for their very helpful comments and insights. 
The author was partially supported by an NSF Postdoctoral Research Fellowship under Grant No. 2002159.

\subsection{Basic notation}
\label{sec-notation}

\noindent
We write $\BB N = \{1,2,3,\dots\}$ and $\ol{\BB N} =  \BB{N}\cup \{\infty\}$.
\medskip

\noindent
For $a < b$, we define the discrete interval $[a,b]_{\BB Z}:= [a,b]\cap\BB Z$. 
\medskip

\noindent
If $f  :(0,\infty) \rta \BB R$ and $g : (0,\infty) \rta (0,\infty)$, we say that $f(\ep) = O_\ep(g(\ep))$ (resp.\ $f(\ep) = o_\ep(g(\ep))$) as $\ep\rta 0$ if $f(\ep)/g(\ep)$ remains bounded (resp.\ tends to zero) as $\ep\rta 0$. We similarly define $O(\cdot)$ and $o(\cdot)$ errors as a parameter goes to infinity. 
\medskip

\noindent
Let $\{E^\ep\}_{\ep>0}$ be a one-parameter family of events. We say that $E^\ep$ occurs with
\begin{itemize}
\item \emph{polynomially high probability} as $\ep\rta 0$ if there is a $p > 0$ (independent from $\ep$ and possibly from other parameters of interest) such that  $\BB P[E^\ep] \geq 1 - O_\ep(\ep^p)$. 
\item \emph{superpolynomially high probability} as $\ep\rta 0$ if $\BB P[E^\ep] \geq 1 - O_\ep(\ep^p)$ for every $p>0$.  
\item \emph{exponentially high probability} as $\ep\rta 0$ if there exists $\lambda >0$ (independent from $\ep$ and possibly from other parameters of interest) $\BB P[E^\ep] \geq 1 - O_\ep(e^{-\lambda/\ep})$. 
\item \emph{\changes{superexponentially} high probability} as $\ep\rta 0$ if $\BB P[E^\ep] \geq 1 - O_\ep(e^{-\lambda/\ep})$ for every $\lambda>0$.  
\end{itemize}
We similarly define events which occur with polynomially, superpolynomially, exponentially, and superexponentially high probability as a parameter tends to $\infty$. 
\medskip

\noindent
We will often specify any requirements on the dependencies on rates of convergence in $O(\cdot)$ and $o(\cdot)$ errors, implicit constants in $\preceq$, etc., in the statements of lemmas/propositions/theorems, in which case we implicitly require that errors, implicit constants, etc., appearing in the proof satisfy the same dependencies. 
\medskip

\noindent
For a metric $D$ defined on a set $U \subset \BB{C}$, a subset $X\subset U$, and $r>0$, we write $B_r(X;D)$ (resp. $B_r[X;D]$) for the set of points in $U$ whose $D$-distance from $X$ is less than $r$ (resp. at most $r$).    When $X = \{x\}$ is a singleton, we write $B_r(X;D)$ and $B_r[X;D]$ as $B_r(x;D)$ and $B_r[x;D]$, respectively, and we call these sets the open and closed $D$-ball centered at $x$ with radius $r$. When $D$ is the Euclidean metric, we abbreviate $B_r(X) = B_r(X;D)$ and $B_r[X] = B_r[X;D]$.
\medskip

\noindent
For a set $X \subset U$, we write $\diam(X;D)$ and $\dim_{\mcl H}(X;D)$ for its diameter in the metric $D$. When $D$ is the Euclidean metric, we abbreviate $\diam(X) = \diam(X;D)$.
\medskip

\noindent
When we write $\ol{X}$ and $\partial X$ for a set $X \subset U$, we mean the closure and boundary of $X$, respectively, in the Euclidean topology on $U$.  In general, when describing a topological property (e.g. open, dense, limit point) without specifying the underlying topology, we are implicitly considering this property relative to the Euclidean topology.
\medskip

\noindent
We define the open annulus
\eqb \label{eqn-annulus-def}
\BB A_{r_1,r_2}(z) := B_{r_2}(z) \setminus \ol{B_{r_1}(z)} ,\quad\forall 0 < r_r < r_2 < \infty.
\eqe 
\medskip

\noindent
We write $\BB S = (0,1)^2$ for the open Euclidean unit square. 
\medskip

\section{Subsequential limits of LFPP are weak LQG metrics}
\label{sec-lfpp-local}

This section is devoted to proving Theorem~\ref{thm-lfpp-axioms}.  

\subsection{Outline of the proof}
\label{sec-thm-outline}

In this section, we outline the main steps of the proof of Theorem~\ref{thm-lfpp-axioms}.

In the preceding work~\cite{dg-supercritical-lfpp} on tightness of LFPP metrics in the supercritical phase, the authors proved the following assertion in the special case in which $h$ is a whole-plane GFF.\footnote{The result that we state here combines~\cite[Theorem 1.2, Assertion 1]{dg-supercritical-lfpp} and~\cite[Equation 5.3]{dg-supercritical-lfpp}.}   For every sequence of values of $\ep$ tending to zero, we can choose a subsequence $\ep_n$ and a lower semicontinuous function $D_h: \BB C \times \BB C \rta [0,\infty)$ such that, as $n \rta \infty$, we have  \eqb
    (h,\frk a_{\ep_n}^{-1} D^{\ep_n}_h) \rta (h,D_h) \qquad \text{ in law} \label{eqn-limit-law} \eqe with respect to the distributional topology in the first coordinate and the lower semicontinuous topology in the second coordinate.

    The convergence statement~\eqref{eqn-limit-law} is our starting point for proving Theorem~\ref{thm-lfpp-axioms}. To deduce the theorem from~\eqref{eqn-limit-law}, we must prove the following assertions:
\begin{enumerate}
    \item \label{item-issue-1} The convergence~\eqref{eqn-limit-law} of LFPP metrics holds not only for $h$ a whole-plane GFF, but also in the more general setting in which  $h$ is a whole-plane GFF plus a bounded continuous function.
\item \label{item-issue-2}
The limiting metric $D_h$ satisfies the five axioms of Definition~\ref{defn-weak-metric}.
\item \label{item-issue-3}
The convergence~\eqref{eqn-limit-law} occurs in probability and not just in law.  This means that the rescaled $\ep$-LFPP metrics converge a.s.\ after passing to a further subsequence.
\item \label{item-issue-4}
The LFPP distances between pairs of rational circles  converge almost surely along any subsequence in which the metrics converge almost surely.
\end{enumerate}
\medskip

\noindent
\textbf{Section~\ref{sec-weyl-scaling}: Weyl scaling.}
To prove  assertion~\ref{item-issue-1} above, we prove the following extension of~\cite[Lemma 2.12]{lqg-metric-estimates} from the continuous metric setting.

\begin{prop} \label{prop-weyl-scaling}
Let $h$ be a whole-plane GFF plus a bounded continuous function, and let $\ep_n$ be a sequence of positive real numbers tending to zero as $n \rta \infty$. Suppose that the metrics $ \frk a_{\ep_n}^{-1} D_h^{\ep_n}$ are coupled so that they converge a.s.\ as $n \rta \infty$ to some metric $ D_h$ w.r.t.\ the lower semicontinuous topology.
Then almost surely, for every sequence of bounded continuous functions $f^n : \BB C\rta \BB R$ such that $f^n$ converges to a bounded continuous function $f$ uniformly on compact subsets of $\BB C$, the metrics $\frk a_{\ep_n}^{-1}  D_{h+f^n}^{\ep_n}$ converge to $e^{\xi f}\cdot D_h$ in the lower semicontinuous topology, where  $D_{h+f^n}^\ep$ is defined as in~\eqref{eqn-lfpp} with $h+f^n$ in place of $h$ and $e^{\xi f}\cdot D_h$ is defined as in~\eqref{eqn-metric-f}. 

Furthermore, if we have also chosen the sequence $\ep_n$ and a coupling of the metrics such that $\frk a_{\ep_n}^{-1} D_h^{\ep_n}(O,O') \rta D_h(O,O')$ a.s.\ for all rational circles $O,O'$, then we have $\frk a_{\ep_n}^{-1}  D_{h+f^n}^{\ep_n}(O,O') \rta e^{\xi f}\cdot D_h(O,O')$ a.s.\ for all rational circles $O,O'$.
\end{prop} 

From this proposition, we may deduce that, if both $h$ and $h'$ are whole-plane GFFs plus  bounded continuous functions, and $\ep_n \rta 0$ is a sequence along which $ \frk a_{\ep_n}^{-1} D_h^{\ep_n} \rta D_h$ in law, then we also have $\frk a_{\ep_n}^{-1} D_{h'}^{\ep_n} \rta D_{h'}$ in law for some limiting metric $D_{h'}$. In fact, we can \emph{couple} $(h , h' , D_h , D_{h'})$ so that $h'- h$ is a bounded continuous function and $D_{h'} = e^{\xi(h' - h)} \cdot D_h$. 

In particular, by taking $h$ to be a whole-plane GFF, we see that~\eqref{eqn-limit-law} implies convergence in law whenever the underlying field is a whole-plane GFF plus a bounded continuous function.  Moreover, Proposition~\ref{prop-weyl-scaling} implies that the limiting metrics associated to fields that differ by a bounded continuous function satisfy Axiom~\ref{item-metric-f}.
\medskip

\noindent
\textbf{Section~\ref{sec-three-axioms-check}: Checking Axioms~\ref{item-metric-length},~\ref{item-metric-translate} and~\ref{item-metric-coord}, and a weaker version of Axiom~\ref{item-metric-local}.}
Next, we address assertion~\ref{item-issue-2} above; namely, showing that the limiting metric $D_h$ satisfies each of the five axioms listed in Definition~\ref{defn-weak-metric}.  As we have just noted, we may deduce Axiom~\ref{item-metric-f} from Proposition~\ref{prop-weyl-scaling} above; so we just need to check the other four axioms.  

Verifying   Axioms~\ref{item-metric-length},~\ref{item-metric-translate} and~\ref{item-metric-coord} is straightforward, but proving Axiom~\ref{item-metric-local} is a more difficult task.  Following~\cite{lqg-metric-estimates}, instead of proving Axiom~\ref{item-metric-local} directly, we begin by restricting to $h$ a whole-plane GFF and proving a weaker locality property of the subsequential limiting metric $D_h$: 

\begin{defn}
\label{defn-local-metric}
Let $U \subset \BB{C}$ be a connected open set, and let \changes{$(h,D_1,\ldots,D_n)$  be a coupling of  $h$ with random length metrics $D_1,\ldots,D_n$ on $U$}. We say that \changes{$D_1,\ldots,D_n$ are \textit{jointly local metrics}} for $h$ if, for any open set $V \subset U$, the \changes{collection of  internal metrics $\{D_j(\cdot,\cdot;V)\}_{j=1}^n$} is conditionally independent from the  pair\footnote{We  can equivalently formulate this condition by replacing this pair with the pair \changes{$(h|_{U \setminus V}  ,\{D_j(\cdot,\cdot; U\setminus \ol V)\}_{j=1}^n)$}.  The fact that the resulting conditions are equivalent is proven in the case in which $D$ is continuous  in~\cite[Lemma 2.3]{local-metrics}; the proof of that lemma extends directly to our more general setting.} \changes{\[ (h  , \{D_j(\cdot,\cdot; U\setminus \ol V)\}_{j=1}^n)\]} given $h|_{\ol V}$.  
\end{defn}
\begin{defn}
\label{defn-xi-additive}
\changes{
Let $U \subset \BB{C}$ be a connected open set, and let $(h,D_1,\ldots,D_n)$  be a coupling of  $h$ with random length metrics $D_1,\ldots,D_n$ on $U$ which are jointly local for $h$. We say that $\{D_j\}_{j=1}^n$ are \textit{$\xi$-additive}  for $h$ if, for each $z \in U$ and each $r > 0$ such that $B_r(z) \subset U$, the metrics $(e^{-\xi h_r(z)} D_1,\ldots,e^{-\xi h_r(z)} D_n)$ are jointly local for $h - h_r(z)$.
}\end{defn}

\begin{comment}
\begin{defn}
\label{defn-local-metric}
Let $U \subset \BB{C}$ be a connected open set, and let $(h,D)$  be a coupling of  $h$ with a random length metric $D$ on $U$. We say that $D$ is a \textit{local metric} for $h$ if, for any open set $V \subset U$, the internal metric $D(\cdot,\cdot;V)$ is conditionally independent from the  pair\footnote{We  can equivalently formulate this condition by replacing this pair with the pair $(h|_{U \setminus V}  ,D(\cdot,\cdot; U\setminus \ol V))$.  The fact that the resulting conditions are equivalent is proven in the case in which $D$ is continuous  in~\cite[Lemma 2.3]{local-metrics}; the proof of that lemma extends directly to our more general setting.} \[ (h  , D(\cdot,\cdot; U\setminus \ol V))\] given $h|_{\ol V}$.  
\end{defn}
\begin{defn}
\label{defn-xi-additive}
If $D$ is a local metric for $h$, then we say that $D$ is \textit{$\xi$-additive}  for $h$ if, for each $z \in U$ and each $r > 0$ such that $B_r(z) \subset U$, the metric $e^{-\xi h_r(z)}$ is local for $h - h_r(z)$.
\end{defn}
\end{comment}

Note that these definitions are~\cite[Definitions 1.2 and 1.5]{local-metrics} with the assumption that the metrics are continuous removed.

\begin{prop}
\label{prop-check-3.5-axioms}
If $h$ is any whole-plane GFF plus a bounded continuous function, then every subsequential limit~\eqref{eqn-limit-law} $(h,D_h)$ of the rescaled $\ep$-LFPP metrics~\eqref{eqn-defn-rescaled-LFPP} satisfies Axioms~\ref{item-metric-length},~\ref{item-metric-translate} and~\ref{item-metric-coord} with $\frk c_r = r^{\xi Q + o_r(1)}$ as $\BB{Q} \ni r \rta 0$ for some $Q>0$.  Moreover, if $h$ is a whole-plane GFF, then $(h,D_h)$ is a $\xi$-additive local metric and $D_h$ is a complete \changes{and geodesic} metric on the set $U \backslash \{\text{singular points}\}$.
\end{prop}
\medskip

\noindent
\textbf{Sections~\ref{sec-geo-behavior}-~\ref{sec-measurability}: Proving Axiom~\ref{item-metric-local}.} 
We now restrict to the case in which $h$ is a whole-plane GFF, and we show that $D_h$ satisfies Axiom~\ref{item-metric-local} by proving the following general assertion:

\begin{prop}
\label{prop-meas-general}
 Let $U \subset \BB C$, and let $h$ be a whole-plane GFF normalized so that $h_1(0) = 0$.  Let $(h,D)$ be a coupling of $h$ with a random lower semicontinuous metric $D$ on $U$.  Suppose that $D$ is a local $\xi$-additive metric for $h$ that  satisfies Axioms~\ref{item-metric-length},~\ref{item-metric-translate}, and~\ref{item-metric-coord} with $\frk c_r = r^{\xi Q + o_r(1)}$ as $\BB{Q} \ni r \rta 0$ for some $Q>0$. Moreover, suppose that 
$D$ is a complete \changes{and geodesic} metric on the set $U \backslash \{\text{singular points}\}$.

Then $D$ satisfies Axiom~\ref{item-metric-local}.  In particular, $D$ is almost surely determined by $h$.
\end{prop}

The analogous result in the continuous setting is the main result of~\cite{local-metrics}. That paper proved the result by dividing a $D$-geodesic path into small Euclidean neighborhoods and applying an Efron-Stein argument.  However, in trying to adapt this method to the supercritical setting, we encounter a fundamental obstacle: when $Q<2$, the path may spend a constant order amount of time in arbitrarily small Euclidean neighborhoods.  These neighborhoods roughly correspond to points where the field is $\alpha$-thick for $\alpha > Q$.  (See Proposition~\ref{prop-singular}.)  To overcome this issue, we show in Lemma~\ref{lem-geo-doesnt-dawdle} that almost $D$-geodesic paths \changes{(and in particular $D$-geodesic paths)} do not spend a macroscopic amount of time in a small Euclidean neighborhood. Our analysis also implies (Lemma~\ref{lem-avoids-thick}) that $D$-geodesics  a.s.\ do not contain any thick points with thickness $>\zeta$ for some deterministic (but nonexplicit) $\zeta < Q$.

After proving Proposition~\ref{prop-meas-general}, we may combine it with Proposition~\ref{prop-check-3.5-axioms} to get Axiom~\ref{item-metric-local} when $h$ is a whole-plane GFF.  We then deduce the case of general $h$ fairly easily.

\begin{lem}
\label{cor-axiom-2-check}
If $h$ is a whole-plane GFF
plus a bounded continuous function, then every subsequential limit~\eqref{eqn-limit-law} of the rescaled $\ep$-LFPP metrics~\eqref{eqn-defn-rescaled-LFPP} satisfies Axiom~\ref{item-metric-local}.
\end{lem}

\medskip

\noindent
\textbf{Section~\ref{sec-complete}: Almost sure convergence of metrics and distances between rational circles.}
To finish the proof of Theorem~\ref{thm-lfpp-axioms}, we first show that the convergence~\eqref{eqn-limit-law} to $(h,D_h)$ occurs not just in law, but also in probability (and therefore a.s.\ along a further subsequence).  Indeed, since $D_h$ is a.s.\ determined by $h$ by Axiom~\ref{item-metric-local}, we immediately get that the convergence occurs in probability, by the following elementary probabilistic lemma:
    
    \begin{lem}[{\cite[Lemma 4.5]{ss-contour}}]
    \label{lem-prob}
    Suppose that $(\Omega_1,d_1)$ and $(\Omega_2,d_2)$  are complete separable metric spaces, $\mu$ is a Borel probability measure on $\Omega_1$, and  $\phi_n: \Omega_1 \rta \Omega_2$ is a sequence of Borel measurable functions. Moreover, suppose that, with $X \sim \mu$, we have $(X,\phi_n(X)) \rta (X,\phi(X))$ in law as $n \rta \infty$, for some Borel measurable  function $\phi: \Omega_1 \rta \Omega_2$.  Then the functions $\phi_n$, viewed as random variables on the probability space $\Omega_1$, converge to $\phi$ in probability.
    \end{lem}
    
Finally, in Lemma~\ref{lem-as-circles}, we show that the LFPP distances between pairs of rational circles converge almost surely along subsequences. 

We conclude the proof of Theorem~\ref{thm-lfpp-axioms} by showing that we  can define a weak LQG metric $D: h \mapsto D_h$ that agrees with the limiting metrics~\eqref{eqn-limit-law} whenever $h$ is a whole-plane GFF plus a bounded continuous function.

\subsection{Weyl scaling}
\label{sec-weyl-scaling}

To prove Proposition~\ref{prop-weyl-scaling}, we first consider a general sequence of  metrics that converges to a  metric in the lower semicontinuous topology that satisfies a continuity assumption, and a general sequence of functions that converges uniformly on compact subsets to some continuous function. We show that the desired Weyl scaling property for convergence of the metrics holds under a mild boundedness condition on geodesics.

\begin{lem}
\label{lem-general-metric-weyl}
Let $\{ D_n\}_{n \in \BB{N}}$ be a sequence of metrics in $\BB C$ that converges to a metric $D_\infty$ in the lower semicontinuous topology.  Assume that the metric space $(\BB{C},D_n)$ is a length space for each $n \in \ol{\BB{N}}$, and that the identity map from $(\BB{C}, D_\infty)$ to $(\BB{C}, |\cdot|)$ is continuous, where $|\cdot|$ denotes the Euclidean metric. 
Let $f_n$ be a sequence of continuous functions  converging uniformly on compact subsets to some continuous function $f_\infty$. 
\changes{Finally, suppose the following two properties hold for some fixed compact sets $K \subset K' \subset K''$.
\begin{enumerate}
\item 
\emph{(Boundedness of almost geodesic paths.)} 
 \label{item-assumption-1}
  Let $n \in \ol{\BB{N}}$, and suppose that $P$ is a path in $K$ (resp. $K'$) between some pair of points $z,w$, such that the $D_n$-length of $P$ is less than $D_n(z,w)+1$, or the $e^{f_n} \cdot D_n$-length of $P$ is less than $e^{f_n} \cdot D_n(z,w)+1$. Then $P$ is contained in $K'$ (resp. $K''$).  
\item
\emph{(Short distances around small annuli.)} 
    \label{item-assumption-2}
    For each nonsingular point $z \in K'$, there is a sequence of rational circles $O^m$ surrounding $z$ whose radii shrink to zero, such that if $\wh{O}^m$ denotes the rational circle with the same center as $O^m$ and twice the radius, and $A^m$ the annulus bounded by $O^m$ and $\wh{O}^m$, then
    \eqb
    \lim_{m \rta \infty} \lim_{n \rta \infty} D_n(\text{around $A^m$}) = 0
\label{eqn-second-assumption}
    \eqe
   (We can equivalently state~\eqref{eqn-second-assumption} with  $e^{f_n} \cdot D_n$  instead of $D_n$, since the metrics $e^{f_n} \cdot D_n$ and $D_n$ are bi-Lipschitz equivalent on $K''$ for all $n$ with bi-Lipschitz constant uniform in $n$.)
   \end{enumerate}} Then the metrics $e^{f_n} \cdot D_n$ restricted to $K$ converge to $e^{f_\infty} \cdot D_\infty$ restricted to $K$ in the lower semicontinuous topology.
\end{lem}

Our proof follows the argument of~\cite[Lemma 7.1]{df-lqg-metric}, with several modifications to account for the different topology of convergence and for the fact that our metrics do not induce the Euclidean topology.   Roughly speaking, to prove Lemma~\ref{lem-general-metric-weyl}, we first decompose a path of near-minimal $(e^{f_n} \cdot D_n)$-length into a collection of short segments.  By bounding the total variation of $f_n$ along a segment $S$, we can compare the $(e^{f_n} \cdot D_n)$-length of $S$ to its $D_n$-length times the value of $e^{f_n}$ at some point of $S$.  This allows us to deduce convergence of $e^{f_n} \cdot D_n$ in the lower semicontinuous topology from convergence of $f_n$ and the metrics $D_n$. 

To bound the variation of $f_n$ along these short segments, we apply the following lemma:

\begin{lem}
\label{lem-diam}
Let $\{ D_n\}_{n \in \BB{N}}$ be a sequence of metrics in $\BB C$ that converges to a metric $D_\infty$ in the lower semicontinuous topology, such that the identity map from $(\BB{C}, D_\infty)$ to $(\BB{C}, |\cdot|)$ is continuous. \changes{Let $G \subset \BB{C}$ be compact, and let $f_n$ be a sequence of continuous functions  converging uniformly on compact subsets to some continuous function $f_\infty$.}   Then we can choose a continuous function $\omega: [0,\infty) \rta [0,\infty)$ with $\omega(0)=0$ such that the following is true.  If $\{ A_n\}_{n \in \ol{\BB{N}}}$ are compact subsets of $G$ with 
\[\liminf_{n \rta \infty} \diam(A_n;D_n) \leq \delta \qquad \text{and} \qquad \diam(A_\infty;D_\infty) \leq \delta,\] then
\[
\liminf_{n \rta \infty} V_{A_n}(f_n) < \omega(\delta) \qquad \text{and} \qquad V_{A_\infty}(f_\infty) < \omega(\delta),\]
where $V_A(f)$ denotes the total variation of $f$ on the set $A$.
\end{lem}

\begin{proof}
First, we can choose an \changes{increasing} continuous function $\varphi: [0,\infty) \rta [0, \infty)$ with $\varphi(0) = 0$ such that $|x-y| < \varphi(D_\infty(x,y))$ for all $x,y \in G$. Moreover, by the Arzel\`{a}-Ascoli theorem, we can choose \changes{an increasing} continuous function $\psi: [0,\infty) \rta [0,\infty)$ with $\psi(0) = 0$ such that $|f_n(x) - f_n(y)| < \psi(|x-y|)$ for all $n \in \ol{ \BB{N}}$ and all $x,y \in G$.  Set $\omega = \psi \circ \varphi$.  

Now, for each $n \in \ol{\BB{N}}$, we choose $z_n,w_n \in A_n$ with $|z_n-w_n|$ maximal.  Then
\[
\changes{\diam(A_\infty)} = |z_\infty - w_\infty| < \varphi(D_\infty(z_\infty,w_\infty)) \leq \varphi(\delta).
\]
and so $V_{A_\infty}(f_\infty) < \psi(\varphi(\delta)) = \omega(\delta)$.

Next, we choose an increasing sequence $\{n_k\}_{k \in \BB{N}} \subset \BB{N}$ for which the corresponding subsequences $\{z_{n_k}\}_{k \in \BB{N}}$ and $\{w_{n_k}\}_{k \in \BB{N}}$ converge to some $z,w \in G$, respectively. \changes{ Since $D_n \rta D_\infty$ in the lower semicontinuous topology,}
\[
|z - w| < \varphi(D_\infty(z,w)) \leq \varphi\left(\liminf_{k \rta \infty} D_{n_k}(z_{n_k},w_{n_k})\right) \leq \varphi(\delta).
\]
Thus, for all $k \in \BB{N}$ sufficiently large, $|z_{n_k} - w_{n_k}| < \varphi(\delta)$ and so $V_{A_{n_k}}(f_{n_k}) < \psi(\varphi(\delta)) = \omega(\delta)$.
\end{proof}

\changes{We also need the following general result for metrics defined as in~\eqref{eqn-metric-f}.

\begin{lem}
\label{lem-weyl-length-lem}
For each $\delta>0$ and $\ell \in \BB{N}$, we can choose points $z=z_0,z_1,\ldots,z_{\ell-1},z_\ell=w$ with $D(z_{j-1},z_j) < \delta$ for each $j$, and such that
\eqb
e^{f} \cdot D(z,w) + \delta \geq  \sum_{j=1}^{\ell} e^{f} \cdot  D(z_{j-1},z_{j}).
\label{eqn-weyl-length-lem}
\eqe
\end{lem}

\begin{proof}
By definition of $e^f \cdot D$, we can choose a path $P:[0,L] \rta \BB{C}$ from $z$ to $w$ parametrized by $D$-length with
\eqb
\int_0^{L} e^{f(P(t))} dt < e^{f} \cdot D(z,w) + \delta.
\label{eqn-weyl-length-pf-1}
\eqe
Choose $\ell > L/\delta$, let $z_j = P( jL/\ell)$, and let $P_j: [0,L/\ell] \rta \BB{C}$ be the path from $z_{j-1}$ to $z_j$ along $P$; i.e., $P_j(t) = P((j-1)L/\ell + t)$.  We have  $D(z_{j-1},z_j) \leq \len(P_i;D) < \delta$ for each $j$.  Moreover,
\eqb
\int_0^{L} e^{f(P(t))} dt 
=
\sum_{j=1}^\ell \int_0^{L/\ell} e^{f(P_j(t))} dt 
\geq 
\sum_{j=1}^\ell e^f \cdot D(z_{j-1},z_j).
\label{eqn-weyl-length-pf-2}
\eqe
Combining~\eqref{eqn-weyl-length-pf-1} and~\eqref{eqn-weyl-length-pf-2} yields~\eqref{eqn-weyl-length-lem}.
\end{proof}
}

\begin{proof}[Proof of Lemma~\ref{lem-general-metric-weyl}]
To show that the metrics $e^{f_n} \cdot D_n$ converge to $e^{f_\infty} \cdot D_\infty$ in the lower semicontinuous topology, we must check the following two conditions:
\begin{enumerate}
\item \label{beer-a}
If $(z_n,w_n)$ is a sequence  in $K \times K$ that converges to a point $(z,w)$, then $\liminf_{n \rta \infty} e^{f_n} \cdot D_n(z_n,w_n) \geq e^{f_\infty} \cdot D_\infty(z,w)$.
\item \label{beer-b}
If $(z,w) \in K \times K$, then there  exists a sequence $(z_n,w_n)$   in $K \times K$ that converges to  $(z,w)$ such that $\liminf_{n \rta \infty} e^{f_n} \cdot D_n(z_n,w_n) = e^{f_\infty} \cdot D_\infty(z,w)$.
\end{enumerate}

 \medskip

\noindent\textit{Proof of Condition~\ref{beer-a}.}
\medskip

\noindent First, if $\liminf_{n \rta \infty} e^{f_n} \cdot D_n(z_n,w_n)$ is infinite, then the result trivially holds; so assume that $\liminf_{n \rta \infty} e^{f_n} \cdot D_n(z_n,w_n)$ is finite. To prove condition~\ref{beer-a}, we show that, for any infinite increasing sequence $\mcl N'$ of natural numbers, we can find an infinite subsequence $\mcl N$ such that $\liminf_{\mcl N \ni n \rta \infty} e^{f_n} \cdot D_n(z_n,w_n) \geq e^{f_\infty} \cdot D_\infty(z,w)$.   

Let  $\mcl N'$ be such a sequence, and fix $\delta > 0$. \changes{
Let $\mcl N$ be an infinite subsequence of $\mcl N'$ for which the distances $d_n := e^{f_n} \cdot D_n(z_n,w_n) + \delta$ are uniformly bounded
by $\ell \delta$ for some $\ell \in \BB{N}$.  For each $n \in \mcl N$, we can consider a path from $z_n$ to $w_n$ with $D_n$-length at most $d_n$.  Since $d_n \leq \ell \delta$, we can choose points $z_n=z_n^0,z_n^1,\ldots,z_n^{\ell-1}, z_n^{\ell}=w_n \in K'$ along this path, such that $D_n(z_n^j,z_n^{j+1}) \leq \delta$ for each $j$ and
\eqb
e^{f_n} \cdot D_n(z_n,w_n) + \delta \geq \sum_{j=0}^{\ell-1} e^{f_n} \cdot  D_n(z_n^j,z_n^{j+1}).
\label{eqn-geodesic-decomp}
\eqe}
Moreover, by replacing $\mcl N$ by a subsequence if necessary, we can assume that, for each $j$, the points $z_{n}^j$ converge to some point $ z_\infty^j \in K'$ as $n \rta \infty$ along $\mcl N$.  Since $ D_n \rta D_\infty$ in the lower semicontinuous topology, we also have $D_\infty(z^j_\infty,z_\infty^{j+1}) \leq \delta$ for each $j$.

\changes{By property~\ref{item-assumption-1} in the lemma statement, for each $n \in \ol{\BB{N}}$ and each $j$, each path between $z_n^j$ and $z_n^{j+1}$  with $D_n$-length sufficiently close to $D_n(z_n^j,z_n^{j+1})$ is a subset of $K''$  with $D_n$-diameter $\leq 2\delta$.  Since the functions $\{f_n\}_{n \in \ol{\BB{N}}}$ are uniformly bounded on $K''$ by some $M>0$, we deduce that each path between $z_n^j$ and $z_n^{j+1}$  with $(e^{f_n} \cdot D_n)$-length sufficiently close to $(e^{f_n} \cdot D_n)(z_n^j,z_n^{j+1})$ is a subset of $K''$ with $D_n$-diameter at most $2 e^{M} \delta$.  By Lemma~\ref{lem-diam}, we can choose a continuous function $\omega: [0,\infty) \rta [0,\infty)$ with $\omega(0)=0$ such that the total variation of $f_n$ on all of these paths is at most $\omega(\delta)$.  It follows that, for each $n \in \ol{\mathbb{N}}$ and $j$,
\eqb
\label{eqn-f-along-segment}
e^{-\omega(\delta)}  \leq 
\frac{(e^{f_n} \cdot D_n)(z_n^j,z_n^{j+1})}{e^{f_n(z_n^j)} \cdot D_n(z_n^j,z_n^{j+1}) } \leq e^{\omega(\delta)}
\eqe
}
Hence,
\[
    e^{f_n} \cdot D_n(z_n,w_n) + \delta
\geq      \sum_{j=0}^{\ell-1} e^{f_n} \cdot  D_n(z_n^j,z_n^{j+1})
\geq     e^{-\omega(\delta)} \sum_{j=0}^{\ell-1} e^{f_n(z_n^j)} \cdot  D_n(z_n^j,z_n^{j+1}) 
\]
for each $n \in \mcl N$, and so
\alb
\liminf_{\mcl N \ni n \rta \infty} e^{f_n} \cdot D_n(z_n,w_n) + \delta
&\geq     e^{-\omega(\delta)} \liminf_{\mcl N \ni n \rta \infty}  \sum_{j=0}^{\ell-1} e^{f_n(z_n^j)} \cdot  D_n(z_n^j,z_n^{j+1}) \\
&\geq e^{-\omega(\delta)} \sum_{j=0}^{\ell-1} e^{f_\infty(z_\infty^j)} \cdot  D_\infty(z_\infty^j,z_\infty^{j+1}) 
\\
&\geq e^{-2\omega(\delta)}  \sum_{j=0}^{\ell-1} e^{f_\infty} \cdot  D_\infty(z_\infty^j,z_\infty^{j+1})   \qquad \text{\changes{by~\eqref{eqn-f-along-segment}}}
\\
&\geq e^{-2\omega(\delta)}   e^{f_\infty} \cdot  D_\infty(z,w).
\ale
Condition~\ref{beer-a} follows from taking $\delta$ arbitrarily small and applying a diagonal argument.
\medskip

\noindent\textit{Proof of Condition~\ref{beer-b}.}
\medskip

\noindent
First, if $D_\infty(z,w) = \infty$, then since  $D_n$ converges to $D_\infty$ a.s.\ in the lower semicontinuous topology, 
\changes{we have $D_n(z_n,w_n) \rta \infty$ as $n \rta \infty$ for some sequences $\{z_n\}$ and $\{w_n\}$. Since the functions $\{f_n\}$ are uniformly bounded, this implies that $e^{f_n} \cdot D_n(z_n,w_n) \rta \infty$ as $n \rta \infty$.} Since $e^{f_\infty}\cdot D_\infty(z,w) = \infty$, the condition is satisfied. So assume that $D_\infty(z,w)$ is finite.

\changes{
By Lemma~\ref{lem-weyl-length-lem},} we can find points $z= z_\infty^0,z_\infty^1,\ldots,z_\infty^\ell=w \in K'$ with
\eqb
e^{f_\infty} \cdot D_\infty(z,w) + \delta \geq  \sum_{j=0}^{\ell-1} e^{f_\infty} \cdot  D_\infty(z_\infty^j,z_\infty^{j+1}) 
\label{eqn-geodesic-decomp-b}
\eqe
and $D_\infty(z_\infty^j,z_\infty^{j+1}) < \delta$ for each $j$.  Since $D_n \rta D_\infty$ in the lower semicontinuous topology, we can choose points $z_n^j$ and $\changes{\wt z_n^{j+1}}$ converging to $z_\infty^j$ and $\changes{z_\infty^{j+1}}$ for each $j \in [0,\ell-1]_{\BB Z}$, such that \eqb
\liminf_{n \rta \infty}  D_n(z_n^j,\changes{ \wt z_n^{j+1}}) \changes{=} D_\infty(z_\infty^j,z_\infty^{j+1}), \label{eqn-liminf-assume-b}
\eqe 
for each $j$.
In particular, we can choose an infinite increasing sequence $\mcl N$ of natural numbers such that $D_n(z_n^j,\changes{\wt z_n^{j+1}})< \delta$ for each $n \in \mcl N$.
\changes{As in~\eqref{eqn-f-along-segment}, we may apply Lemma~\ref{lem-diam} to deduce that, for each $n \in \mcl N$ and $j \in [0,\ell-1]_{\BB{Z}}$, 
\eqb
\label{eqn-f-along-segment-b}
e^{-\omega(\delta)}  \leq 
\frac{(e^{f_n} \cdot D_n)(z_n^j,\changes{ \wt z_n^{j+1}})}{e^{f_n(z_n^j)} \cdot D_n(z_n^j, \changes{ \wt z_n^{j+1}}) } \leq e^{\omega(\delta)}
\eqe
Moreover, by  property~\ref{item-assumption-2} in the statement of the lemma, for each $j \in [1,\ell-1]_{\BB Z}$, we can choose an annulus $A^j$ surrounding $z_\infty^j$ with diameter at most $\delta$, such that \eqb
\label{eqn-around-bound-2}
\liminf_{n \rta \infty} e^{f_n} \cdot D_n(\text{around $A^j$}) \leq \delta/\ell.
\eqe}
 Therefore, setting $z_n = z_n^0$ and $w_n = z_n^\ell$,
\alb
\liminf_{n \rta \infty} e^{f_n} \cdot D_n(z_n,w_n) &\leq \liminf_{n \rta \infty} \sum_{j=0}^{\ell-1} e^{f_n} \cdot  D_n(z_n^j, \changes{ \wt z_n^{j+1}}) \changes{ +  \liminf_{n \rta \infty}   \sum_{j=1}^{\ell-1} e^{f_n} \cdot  D_n(\text{around $A^j$})} \\
&\leq 
\liminf_{n \rta \infty}  \sum_{j=0}^{\ell-1} e^{f_n} \cdot  D_n(z_n^j, \changes{ \wt z_n^{j+1}}) \changes{+ \delta} \qquad \text{\changes{by~\eqref{eqn-around-bound-2}}}\\
&\leq 
\liminf_{n \rta \infty}  e^{\omega(\delta)} \sum_{j=0}^{\ell-1} e^{f_n(z_n^j)} \cdot  D_n(z_n^j, \changes{ \wt z_n^{j+1}}) \changes{+ \delta}\qquad \text{\changes{by~\eqref{eqn-f-along-segment-b}}}
\ale
for each $n \in \mcl N$,
and so
\alb
\liminf_{\mcl N \ni n \rta \infty} e^{f_n} \cdot D_n(z_n,w_n) 
&\leq     e^{\omega( \delta)} \liminf_{\mcl N \ni n  \rta \infty}  \sum_{j=0}^{\ell-1} e^{f_n(z_n^j)} \cdot  D_n(z_n^j,z_n^{j+1}) \changes{+ \delta}\\
&\leq e^{\omega(\delta )} \sum_{j=0}^{\ell-1} e^{f_\infty(z_\infty^j)} \cdot  D_\infty(z_\infty^j,
\changes{\wt z_\infty^{j+1}})  \changes{+ \delta} \qquad \text{by~\eqref{eqn-liminf-assume-b}}
\\
&\leq e^{2\omega( \delta )}  \sum_{j=0}^{\ell-1} e^{f_\infty} \cdot  D_\infty(z_\infty^j,z_\infty^{j+1})   \changes{+ \delta} \qquad \text{\changes{by~\eqref{eqn-f-along-segment-b}}}
\\
&\leq e^{2\omega(  \delta )}  ( e^{f_\infty} \cdot  D_\infty(z,w) + \delta)  \changes{+ \delta}. \qquad \text{by~\eqref{eqn-geodesic-decomp-b}}
\ale
Condition~\ref{beer-b} follows from taking $\delta$ arbitrarily small and applying a diagonal argument.
\end{proof}

Finally, we prove a lemma proving the convergence of distances between pairs of rational circles with respect to a general sequence of metrics that converges in the lower semicontinuous topology.

\begin{lem}
\label{lem-general-rational-weyl}
In the setting of Lemma~\ref{lem-general-metric-weyl}, if
    \eqb
    D_n(O,O') \rta D_\infty(O,O') \qquad \text{for all rational circles $O,O'$ in $K'$,}
    \label{eqn-first-assumption}
    \eqe 
then 
    \eqb
   e^{f_n} \cdot D_n(O,O') \rta e^{f_\infty} \cdot D_\infty(O,O') \qquad \text{for all rational circles $O,O'$ in $K$.}
    \eqe 
\end{lem}

\begin{proof}
The result is trivial if $D_\infty(O,O') = \infty$, so suppose that $D_\infty(O,O') < \infty$. It suffices to show that $\liminf_{n \rta \infty} e^{f_n} \cdot D_n(O,O') \leq e^{f_\infty}\cdot D_\infty(O,O')$. 

Fix $\delta  > 0$. 
By Lemma~\ref{lem-weyl-length-lem}, we can find points $ z^0,z^1,\ldots,z^\ell\in K'$, with $z^0 \in O$ and $z^\ell \in O'$, such that
\[
e^{f_\infty} \cdot D_\infty(O,O') + \delta \geq  \sum_{j=0}^{\ell-1} e^{f_\infty} \cdot  D_\infty(z^j,z^{j+1}) 
\]
and $D_\infty(z^j,z^{j+1}) < \delta$ for each $j$.  
%Now, since the metrics $e^{f_n} \cdot D_n$ and $D_n$ are bi-Lipschitz equivalent on $K''$ for all $n$ with bi-Lipschitz constant uniform in $n$, the second assumption in the statement of the lemma still holds with $D_n$ replaced by $e^{f_n} \cdot D_n$. 
\changes{Fix $j \in [1,\ell-1]_{\BB Z}$. Since the set of rational points is dense in $\BB{C}$, there exists a sequence of rational circles containing $z^j$ with radii shrinking to zero. By property~\ref{item-assumption-2} in the statement of Lemma~\ref{lem-general-metric-weyl},} this implies that we can choose a rational circle  $O^j \subset K'$ containing $z^j$ with diameter at most $\delta$, such that, if $\wh{O}^j$ denotes the circle with the same center as $O^j$ and twice the radius, and $A^j$ is the annulus bounded by $O^j$ and $\wh{O}^j$, then \[
\liminf_{n \rta \infty} e^{f_n} \cdot D_n(\text{around $A^j$}) \changes{\leq \delta/\ell}.
\]
\changes{Also, property~\ref{item-assumption-2}  implies that we may choose the radii of these annuli to be small enough that the annuli are disjoint for different $j$.}
Thus, denoting $O^0 = O$ and $O^\ell = O'$, we have \eqb
e^{f_\infty} \cdot D_\infty(O,O') + \delta \geq \sum_{j=0}^{\ell-1} e^{f_\infty} \cdot  D_\infty(O^j,O^{j+1}) 
\label{eqn-geodesic-decomp1}
\eqe
and
\allb\nonumber
\liminf_{n \rta \infty} e^{f_n} \cdot D_n(O,O') &\leq \liminf_{n \rta \infty} \sum_{j=0}^{\ell-1} e^{f_n} \cdot  D_n(O^j,O^{j+1}) +  \liminf_{n \rta \infty}   \sum_{j=1}^{\ell-1} e^{f_n} \cdot  D_n(\text{around $A^j$}) \\
&\leq 
\liminf_{n \rta \infty}  \sum_{j=0}^{\ell-1} e^{f_n} \cdot  D_n(O^j,O^{j+1}) + \delta.
\label{eqn-geodesic-decomp2}
\alle
By Lemma~\ref{lem-diam} and property~\ref{item-assumption-1} in the statement of Lemma~\ref{lem-general-metric-weyl}, we can choose a continuous function $\omega: [0,\infty) \rta [0,\infty)$ with $\omega(0)=0$ such that, for each $n \in \ol{\BB{N}}$ and each $j \in [0,\ell-1]_{\BB{Z}}$, the total variation of $f_n$ along $O^j$ satisfies
\eqb
\label{eqn-total-var-O}
V_{O^j}(f_n) \leq \omega(\delta).
\eqe
\changes{Moreover, as in~\eqref{eqn-f-along-segment}, we may apply Lemma~\ref{lem-diam} to deduce that, for each $n \in \mcl N$ and $j \in [0,\ell-1]_{\BB{Z}}$, 
\eqb
\label{eqn-f-along-segment-O}
e^{-\omega(\delta)}  \leq 
\frac{(e^{f_n} \cdot D_n)(O^j,O^{j+1})}{e^{f_n(z_n^j)} \cdot D_n(O^j,O^{j+1}) } \leq e^{\omega(\delta)}
\eqe for some  $z_n^j \in O^j$. Hence,}
\alb
\liminf_{n \rta \infty} e^{f_n} \cdot D_n(O,O') 
&\leq     e^{\omega(\delta)} \liminf_{n \rta \infty}  \sum_{j=0}^{\ell-1} e^{f_n(z_n^j)} \cdot  D_n(O^j,O^{j+1}) + \delta \qquad \text{\changes{by~\eqref{eqn-f-along-segment-O}}}\\
&\leq e^{\omega(\delta)} \sum_{j=0}^{\ell-1} \left(\liminf_{n \rta \infty} e^{f(z_n^j)}\right) \cdot  D_\infty(O^j,O^{j+1}) + \delta\qquad \text{by~\eqref{eqn-first-assumption}}
\\
&\leq e^{2\omega(\delta) + \omega(\delta) }  \sum_{j=0}^{\ell-1} e^{f_\infty} \cdot  D_\infty(O^j,O^{j+1})   + \delta \qquad \text{\changes{by~\eqref{eqn-total-var-O} and~\eqref{eqn-f-along-segment-O}}}
\\
&\leq e^{2\omega(\delta) + \omega(\delta)}   (e^{f_\infty} \cdot  D_\infty(z,w) + \delta) + \delta \qquad \text{by~\eqref{eqn-geodesic-decomp1}}
\ale
The result follows from taking $\delta$ arbitrarily small and applying a diagonal argument. 
\end{proof}

To apply Lemmas~\ref{lem-general-metric-weyl} and~\ref{lem-general-rational-weyl}  to our setting, we must check that the conditions in the lemma hold in our setting with high probability.

\begin{lem}
\label{lem-geo-annuli}
For every  $C>0$ \changes{and $r>0$},
\[
\lim_{R \rta \infty}  \BB P(D_h(\text{around $\BB A_{r,2r}(0)$}) \changes{ + C < (1/C) \cdot} D_h(\text{across $\BB A_{2r,R}(0)$})) = 1.
\]
\end{lem}
\begin{proof}
By~\cite[Proposition 4.1]{dg-supercritical-lfpp},
\eqb
\lim_{T \rta \infty} \lim_{\ep \rta 0}
\BB P(
\frk a_{\ep}^{-1}  D_h^{\ep}( \text{around $\BB{A}_{r,2r}(0)$}) \leq T) = 1.
\label{eqn-lem-geo-annuli-2a}
\eqe
On the other hand, for each fixed $T$,~\cite[Lemma 4.12]{dg-supercritical-lfpp} gives
\eqb
\lim_{R \rta \infty} \lim_{\ep \rta 0}
\BB P(
\frk a_{\ep}^{-1}  D_h^{\ep}( \partial B_R(0), \partial B_{2r}(0) ) \geq T) = 1.
\label{eqn-lem-geo-annuli-2b}
\eqe
Combining~\eqref{eqn-lem-geo-annuli-2a} and~\eqref{eqn-lem-geo-annuli-2b} gives
\eqb
\lim_{R \rta \infty} \lim_{\ep \rta 0} \BB P(\frk a_{\ep}^{-1}  D_h^{\ep}(\text{around $\BB A_{r,2r}(0)$}) \changes{ + C < (1/C) \cdot} \frk a_{\ep}^{-1}  D_h^{\ep}(\text{across $\BB A_{2r,R}(0)$})) = 1.
\label{eqn-lem-geo-annuli-2}
\eqe
The result follows from~\eqref{eqn-lem-geo-annuli-2} and 
the weak subsequential convergence $\frk a_{\ep}^{-1}  D_h^{\ep} \rta D_h$ in the lower semicontinuous topology.
\end{proof}

\begin{proof}[Proof of Proposition~\ref{prop-weyl-scaling}]
We first prove convergence of the metrics in the lower semicontinuous topology. Let $f_{\ep_n}^{*,n} = f^n*p_{\ep_n^2/2}$ be defined as in~\eqref{eqn-gff-convolve} with  $f^n$ in place of $h$. 
Then $f_{\ep_n}^{*,n} \rta f$ uniformly on compact subsets of $\BB C$.
By the definition~\eqref{eqn-lfpp} of LFPP, we have $D_{h+f^n}^{\ep_n} = e^{\xi f_{\ep_n}^{*,n}} \cdot D_h^{\ep_n}$. 
Thus, to prove the convergence of metrics in the proposition, it suffices to show that 
$
e^{f_n} \cdot D_n \rta e^{f_\infty} \cdot D_\infty
$ a.s.\ in the lower semicontinuous topology, where
\[ D_n = \frk a_{\ep_n}^{-1}   D_h^{\ep_n}, \qquad D_\infty = D_h, \qquad f_n = \xi f_{\ep_n}^{*,n}, \qquad \text{and} \quad f_\infty = \xi f. \]
By definition of the lower semicontinuous topology, it is enough to prove convergence of the metrics restricted to $B_r(0)$ on an event with probability $1 - \kappa$, for each fixed $r>0$ and $\kappa \in (0,1)$.

By Lemma~\ref{lem-geo-annuli}, we can choose $r',r'' > 0$ such that the following holds on an event $\mcl E_\kappa$ with probability at least $1 - \kappa$, for some sufficiently large integer $N$.  Suppose $D = D_n$ or $D = e^{f_n} \cdot D_n$ for $n = \infty$ or some integer $n > N$.  \changes{Then, if $P$ is a path between a pair of points $z,w$ in $B_r(0)$ (resp. $B_r'(0)$) with $D$-length less than $D(z,w) + 1/\kappa$, then $P$ is contained in $B_{r'}(0)$   (resp. $B_{r''}(0)$).  }

In other words, on the event $\mcl E_\kappa$, the metrics $D_n$ for $n \in \Ninf$ with $n \geq N$ satisfy the path boundedness condition of Lemma~\ref{lem-general-metric-weyl}  with $K = B_{r}(0)$, $K' = B_{r'}(0)$ and $K'' = B_{r''}(0)$.
Also, the local H\"older continuity assumption holds by~\cite[Theorem 1.3, Assertion 3]{dg-supercritical-lfpp}.  Thus, the a.s.\ convergence of the metrics  $\left( e^{f_n} \cdot D_n \right)|_{B_r(0)}$ on the event $\mcl E_\kappa$ follows from Lemma~\ref{lem-general-metric-weyl}.

Next, we prove the second part of the proposition.  Let $O,O'$ be rational circles, and choose $r>0$ so that $O,O' \subset B_r(0)$. Again, it suffices to prove that $e^{f_n} \cdot D_n(O,O') \rta e^{f_\infty} \cdot D_\infty(O,O')$  on an event with probability $1 - \kappa$ for each fixed $\kappa > 0$. We will prove this assertion by applying Lemma~\ref{lem-general-rational-weyl}.  To apply this lemma, the metrics $D_n$ must satisfy (for large enough $n$) both the hypotheses of   Lemma~\ref{lem-general-metric-weyl} (with $K = B_r(0)$) and properties~\ref{item-assumption-1} and~\ref{item-assumption-2} in the statement of Lemma~\ref{lem-general-rational-weyl}.  We have already shown in the previous paragraph that the metrics $D_n$ satisfy the hypotheses of   Lemma~\ref{lem-general-metric-weyl} for large enough $n$, and we have assumed in the statement of Proposition~\ref{prop-weyl-scaling} that the metrics satisfy property~\ref{item-assumption-1}.  Finally,  property~\ref{item-assumption-2} is immediate from~\cite[Lemma 5.12]{dg-supercritical-lfpp} in the special case in which $h$ is a whole-plane GFF.  This implies that property~\ref{item-assumption-2} must hold for general $h$, since for every bounded continuous function $f$, the metrics  $D_h^{\ep}$ and $D_{h+f}^\ep$ are bi-Lipschitz equivalent with bi-Lipschitz constant uniform in $\ep>0$.  Hence,
we may apply Lemma~\ref{lem-general-rational-weyl} to obtain the desired convergence $e^{f_n} \cdot D_n(O,O') \rta e^{f_\infty} \cdot D_\infty(O,O')$.
\end{proof}
\subsection{Checking Axioms~\ref{item-metric-length},~\ref{item-metric-translate} and~\ref{item-metric-coord}, and a locality property}
\label{sec-locality-2}
\label{sec-three-axioms-check}

In this section, we prove Proposition~\ref{prop-check-3.5-axioms}, which asserts that every weak subsequential limit $(h,D_h)$ satisfies Axioms~\ref{item-metric-length},~\ref{item-metric-translate} and~\ref{item-metric-coord}  with $\frk c_r = r^{\xi Q + o_r(1)}$ as $\BB{Q} \ni r \rta 0$ for some $Q>0$.  Moreover, the proposition states that, in the special case in which $h$ is a whole-plane GFF, $D_h$ is a complete metric on the set $U \backslash \{\text{singular points}\}$, as defined in Definition~\ref{defn-singular}, and $(h,D_h)$ is a $\xi$-additive local metric, as defined in Definitions~\ref{defn-local-metric} and~\ref{defn-xi-additive}.  
The $\xi$-additive locality property is a weaker version of \changes{Axiom~\ref{item-metric-local}}; in the sections that follow, we will combine it with the other properties listed in the proposition to obtain  \changes{Axiom~\ref{item-metric-local}.  }

Most of the work of this section consists of proving this $\xi$-additive locality property.  Our proof closely follows the approach of~\cite[Section 2.5]{lqg-metric-estimates}.  We begin by stating a version of this property in which we restrict to values of $r>0$  for which $B_r(z) \subset V$.

\begin{lem}
\label{lem-property-1-check-simp}
Let $(h,D_h)$ be as in~\eqref{eqn-limit-law}.  
Let $r> 0$ and $z \in \BB{C}$, and let $V$ be an open subset of $\BB{C}$ that contains $B_{r}(z)$ and whose closure is not all of $\BB{C}$.  Set $\hf = h - h_{r}(z)$. Then the internal metric $ D(\cdot,\cdot;  V)$ is conditionally independent from the pair $ (\hf  ,  D_\hf(\cdot,\cdot; U\setminus \ol V))$ given $\hf|_{\ol V}$.
\end{lem}

It is clear that, if $(h,D_h)$ is a $\xi$-additive local metric, then  Lemma~\ref{lem-property-1-check-simp} holds.  In fact, the two statements are equivalent, though we will not explain the proof of this equivalence in detail since the proof is identical to  the analogous result~\cite[Lemma 2.19]{lqg-metric-estimates} in the continuous metric setting.  The reason for considering the version of the $\xi$-additive locality property stated in Lemma~\ref{lem-property-1-check-simp} is that, with the restriction on the value of $r>0$, the field $\hf$ satisfies a Markov property for the whole-plane GFF from~\cite{gms-harmonic} (see also~\cite[Lemma 2.18]{lqg-metric-estimates}), which will be useful in our proof of Lemma~\ref{lem-property-1-check-simp}:

\begin{lem}[{\cite{gms-harmonic}}]
\label{lem-markov}
With $\hf$ and $V$ as in the statement of Lemma~\ref{lem-property-1-check-simp}, 
we can write
\eqb
\hf|_{\BB C\setminus \ol V} = \frk h + \rng h
\label{eqn-field-V-decomp}
\eqe
where $\frk h$ is a random harmonic function on $\BB C\setminus \ol V$ which is determined by $\hf|_{V}$ and $\rng h$ is a zero-boundary GFF in $\BB C\setminus \ol V$ which is independent from $\hf|_V$. 
\end{lem}

Roughly speaking, the idea of the proof of Lemma~\ref{lem-property-1-check-simp} is to pass a locality result for $D^\ep_h$ to the subsequential limit.  Note, however, that  $D^\ep_\hf$-lengths of paths are not locally determined by $\hf$, since the mollified version $\hf_\ep^*(z)$ of $\hf$ we defined in~\eqref{eqn-gff-convolve} does not \emph{exactly} depend locally on $\hf$.  Thus, we instead work with  a localized version $\wh D^\ep_\hf$ of the LFPP metric $D^\ep_\hf$ defined in terms of a localized variant of the mollification $\hf_\ep^*$.

\begin{defn}[Localized LFPP metric]
Let $h$ be a whole-plane GFF with some normalization plus a bounded continuous function.  We define $\wh h_\ep^*(z)$ as
the integral (interpreted in the sense of distributional pairing)
\eqb \label{eqn-localized-def}
\int_{\BB C} \psi(\ep^{-1/2}(z-w)) h(w) p_{\ep^2/2} (z,w) \, dw ,
\eqe
where $\psi$ is a deterministic, smooth, radially symmetric bump function with compact support \changes{contained in $B_1(0)$} that equals $1$ in a neighborhood of the origin.
We then define the \emph{localized $\ep$-LFPP metric} by
\eqb \label{eqn-localized-lfpp}
\wh D_{h}^\ep(z,w) := \inf_{P : z\rta w} \int_0^1 e^{\xi \wh h_\ep^*(P(t))} |P'(t)| \,dt ,
\eqe
where the infimum is over all piecewise continuously differentiable paths from $z$ to $w$. 
\end{defn}

The following lemma allows us to work with $\wh D_\hf^\ep$ instead of $D_\hf^\ep$ and use a locality property of $\wh D_h^\ep$.

\begin{lem}
\label{lem-localized-compare}
For any open set $U$,  the internal metric $\wh D_{h}^\ep(\cdot,\cdot; U)$ is a.s.\ determined by $h|_{B_{\ep^{1/2}}(U)}$.  Moreover, if $\ep_n$ is a sequence tending to zero such that $(h, \frk a_{\ep_n}^{-1} D_h^{\ep_n}) \rta (h, D_h)$ weakly in the lower semicontinuous topology as $n \rta \infty$, then $(h, \frk a_{\ep_n}^{-1} \wh D_h^{\ep_n}) \rta (h, D_h)$ weakly in the lower semicontinuous topology.
\end{lem}

\begin{proof}
The first assertion is immediate from the definition~\eqref{eqn-localized-def} of $\wh h_\ep^*$.  For the second assertion, by comparing $h_\ep^*$ and $\wh h_\ep^*$, one can show~\cite[Lemma 2.1]{lqg-metric-estimates} that
$
\lim_{\ep\rta 0} \sup_{z\in \ol U} |h_\ep^*(z) - \wh h_\ep^*(z)|  = 0,
$
which immediately implies that a.s.\ 
$
\lim_{\ep\rta 0}  \frac{\wh D_{h}^\ep(z,w;U)}{D_{h}^\ep(z,w;U)} = 1
$
uniformly over all $z,w\in U$ with $z\not=w$.  This proves the second assertion. \end{proof}

Lemma~\ref{lem-localized-compare} gives a locality property of the internal metric $\wh D_{\hf}^\ep(\cdot,\cdot; U)$ for any open set $U$. To translate this property into a locality property for the metric  $D_\hf(\cdot,\cdot;U)$, we need to know that the internal metrics  $\wh D_{\hf}^\ep(\cdot,\cdot; U)$ converge to $D_\hf(\cdot,\cdot;U)$.  The following lemma asserts that these internal metrics do converge when restricted to a sufficiently small neighborhood.

\begin{lem} \label{lem-internal-radius}
\changes{
Let $Y \subset X \subset \BB{C}$, and let $\{ D_n\}_{n \in \BB{N}}$ be a sequence of metrics in $X$ that converges to a metric $D_\infty$ in the lower semicontinuous topology. Fix $\ep>0$ such that $|x-\partial Y| > 2\ep$, and set $R = \frac{1}{4} \liminf_{n \rta \infty} D_n(B_\ep(x),B_\ep(\partial Y))$.  
Then $R>0$, and the internal metrics $D_n(\cdot,\cdot;Y)|_{B_R(x)}$ converge to $D_\infty(\cdot,\cdot;Y)|_{B_R(x)}$ in the lower semicontinuous topology.}
\end{lem}

\begin{proof}
For each $n \in \ol{\BB{N}}$, the metrics $D_n$ and $D_n(\cdot,\cdot;Y)$ agree on the ball centered at $x$ with radius $r_x^n := \frac{1}{2} D_n(x,\partial Y)$.   \changes{
Since $D_n \rta D_\infty$ in the lower semicontinuous topology, $R \leq r_x^\infty$  and
\[R =  \frac{1}{4} \liminf_{n \rta \infty} D_n(B_\ep(x),B_\ep(\partial Y)) \geq \frac{1}{4}  D_\infty(B_\ep(x),B_\ep(\partial Y)) > 0\]
Moreover,  $R \leq r_x^n$ for all $n \in \BB{N}$ sufficiently large, and $R \leq r_x^\infty$ since $D_n \rta D_\infty$ in the lower semicontinuous topology.  Hence, for $n = \infty$ and $n \in \BB{N}$ sufficiently large,} the metrics $D_n$ and $D_n(\cdot,\cdot;Y)$ agree on the ball $B_{R}(x)$.  This implies that $D_n(\cdot,\cdot;Y)|_{B_{R}(x)}$ converges to $D_\infty(\cdot,\cdot;Y)|_{B_{R}(x)}$ in the lower semicontinuous topology as $n \rta \infty$, as desired.  
\end{proof}

\begin{proof}[Proof of Lemma~\ref{lem-property-1-check-simp}]
To pass the locality property of Lemma~\ref{lem-localized-compare} to the distributional limit, we first rephrase this property as an exact independence statement. Let $W,W'$ be open sets whose closures are contained in $V$ and $\BB C \setminus \ol V$, respectively.  Also, let $\phi$ be a deterministic, smooth, compactly supported bump function which is identically equal to 1 on a neighborhood of $\ol W'$ and which vanishes outside of a compact subset of $\BB C\setminus \ol V$. Finally, let $\frk h$ and $\rng h$ be as in Lemma~\ref{lem-markov}.

Observe that the restrictions of the fields \changes{$\hf - \phi \frk h $} and $\rng h $ to the set $\phi^{-1}(1)\supset \ol W'$ are identical. 
By Lemma~\ref{lem-localized-compare}, if $\ep > 0$ is small enough that \changes{$B_{\ep^{1/2}}(W') \subset \phi^{-1}(1)$}, then the $\ep$-LFPP metric for \changes{$\hf - \phi \frk h $} satisfies 
  $\wh D_{\changes{\hf - \phi \frk h} }^\ep (\cdot,\cdot ; \ol W')  \in \sigma\left(\rng h \right)$. 
Similarly, for small enough $\ep > 0$ the metric $\wh D_{\changes{\hf}}^\ep(\cdot,\cdot ; \ol W)$ is a.s.\ determined by $\changes{\hf}|_V$. 
Since $\changes{\hf} |_V$ and $\rng h $ are independent, this implies that
\eqb \label{eqn-internal-metric-ind-ep}
\left( \hf|_V , \frk a_\ep^{-1} \wh D_\hf^\ep(\cdot,\cdot; \ol W) \right) \quad \text{and} \quad
\left( \changes{\rng h} , \frk a_\ep^{-1} \wh D_{\changes{\hf - \phi \frk h}}^\ep(\cdot,\cdot ; \ol W')\right) \quad \text{are independent} .
\eqe
Now, recalling the definition of $(h,D_h)$ in~\eqref{eqn-limit-law} as a weak limit of LFPP metrics, we now want to pass the independence~\eqref{eqn-internal-metric-ind-ep} through to the limit as $\ep \rta 0$ along some deterministic subsequence of the sequence $\{\ep_n\}$.  First, by Lemma~\ref{lem-localized-compare}, we have $(h ,   \frk h , \frk a_{\ep}^{-1} \wh D_h^{\ep}) \rta (h  , \frk h , D_h)  $ in law along some deterministic subsequence $\mcl E$ of $\{\ep_n\}$, where here the second coordinate is given the local uniform topology on $\BB C\setminus \ol V$. Second, by the analog of Proposition~\ref{prop-weyl-scaling} with $\wh D_\cdot^\ep$ in place of $D_\cdot^\ep$ (which is proven in an identical manner), if we set $D_{h-\phi\frk h} = e^{-\xi \phi \frk h} \cdot D_h$, then along $\mcl E$  we have the convergence of joint laws
\[
\left(h ,   \frk h , \frk a_{\ep}^{-1} \wh D_h^{\ep} ,  \frk a_{\ep}^{-1} \wh D_{h - \phi\frk h}^{\ep} \right)
 \rta \left( h  , \frk h , D_h , D_{h - \phi\frk h} \right).
\]
By Lemma~\ref{lem-internal-radius}, for each $x \in W$ and $x' \in W'$, we can choose $R_x,R_{x'} > 0$ such that
\[
\left( \hf|_V , R_x, \left.D_\hf(\cdot,\cdot; \ol W)\right|_{B_{R_x}(x)} \right) \quad \text{and} \quad
\left( \changes{\rng h} , R_{x'},  \left.D_{\changes{\hf - \phi \frk h}}(\cdot,\cdot ; \ol W')\right|_{B_{R_{x'}}(x')}\right) \quad \text{are independent} .
\]
Since we can cover every compact subset of $W$ by finitely many balls of the form $B_{R_x}(x)$, and the same is true of $W'$, we deduce that
\[
\left( \hf|_V , D_\hf(\cdot,\cdot; \ol W) \right) \quad \text{and} \quad
\left( \changes{\rng h} ,  D_{\changes{\hf - \phi \frk h}}(\cdot,\cdot ; \ol W')\right) \quad \text{are independent} .
\]
To complete the proof, we observe that, since $\frk h$ is determined by $h|_V$ \changes{(Lemma~\ref{lem-markov})}, we have that
\begin{itemize}
    \item \changes{$\hf$} is a measurable function of \changes{$\hf|_V$} and $\rng h$, and
\item $D_\changes{\hf}(\cdot,\cdot;\ol W')$ is a measurable function of \changes{$\hf|_V$} and $D_{\changes{\hf - \phi \frk h}}(\cdot,\cdot;\ol W')$.
\end{itemize}
Thus, $D_h(\cdot,\cdot;\ol W)$ is conditionally independent from $(h,D_h(\cdot,\cdot;\ol W'))$ given $h|_V$. The result follows from letting $W$ and $W'$ increase to $V$ and $\BB C \setminus \ol V$, respectively.
\end{proof}

\begin{proof}[Proof of Proposition~\ref{prop-check-3.5-axioms}]
    The metric $D_h$ satisfies Axiom~\ref{item-metric-length}  by~\cite[Theorem 1.3, Assertion 5]{dg-supercritical-lfpp} when $h$ is a whole-plane GFF; the axiom follows for general $h$ by applying Proposition~\ref{prop-weyl-scaling}.  Axiom~\ref{item-metric-translate} is immediate from the definition of LFPP. 

To prove Axiom~\ref{item-metric-coord},  \changes{we start with the scale invariance property
\[
\left( D_h^{\ep/r}(z,w) \right)_{z, w \in \BB{C}} \stackrel{d}{=} \left( r^{-1} e^{-\xi h_r(0)} D^\ep_h(rz,rw) \right)_{z,w \in \BB{C}}
\]
stated in~\cite[Equation 4.29]{dg-supercritical-lfpp} in the special case of a whole-plane GFF $h$.  We can extend this result to general $h$ by applying Proposition~\ref{prop-weyl-scaling}. From this, we deduce the following scale invariance property for a fixed annulus $\BB{A}$:}
\eqb
 e^{-\xi h_r(0)}  \frac{ \frk a_\ep}{r \frk a_{\ep/r}} \frk a_{\ep }^{-1} D_h^{\ep } (\text{\changes{around $r\BB{A}$}} )
\eqD    \frk a_{\ep/r}^{-1} D_{h  }^{\ep/r}(\text{\changes{around $\BB{A}$}}) ,\quad \forall r >0, \ep \in (0,r).
\label{eqn-equiv-cr}
\eqe
Set $\frk c_r$ equal to~\eqref{eqn-defn-cr} for $r \in \BB{Q}$ and $\frk c_r = \liminf_{\BB{Q} \ni q > r} \frk c_q$ for $r \notin \BB{Q}$.  Thus, for rational $r>0$, the laws of the \changes{distances} on the left-hand side of~\eqref{eqn-equiv-cr} converge to that of \eqb e^{-\xi h_r(0)} \frk c_r^{-1} D_h(\text{\changes{around $r\BB{A}$}}) \label{eqn-separate-line} \eqe along a subsequence. Moreover, the law of the \changes{distance}~\eqref{eqn-separate-line} for $ r \notin \BB{Q}$ can be expressed as the limit of the corresponding laws for rational values of $r$.  We deduce that, for all $r>0$, the law of~\eqref{eqn-separate-line} is a subsequential limit of the laws of the \changes{distances} $(\frk a_{\ep/r}^{-1} D_{h^r}^{\ep/r}(\text{\changes{around $\BB{A}$}}))_{\ep > 0}$.  Hence, the set of laws of the \changes{distances}~\eqref{eqn-separate-line} for  $r>0$ is tight.  \changes{We can apply the same argument to distances around annuli.} Finally, by~\cite[Theorem 2, Assertion 4]{dg-supercritical-lfpp}, $\frk c_r$   satisfies the required scaling property~\eqref{eqn-scaling-constant}. This proves Axiom~\ref{item-metric-coord}.

Also, the fact that $D_h$ is a complete \changes{and geodesic} metric on the set $U \backslash \{\text{singular points}\}$ when $h$ is a whole-plane GFF is~\cite[Theorem 1.3, Assertions 4 and 5]{dg-supercritical-lfpp}.

Finally, we check that $D_h$ is a $\xi$-additive local metric for $h$  when $h$ is a whole-plane GFF.  Indeed, this property is vacuously true when $\ol V = \BB{C}$.  \changes{For other choices of $V$, we have already stated in Lemma~\ref{lem-property-1-check-simp} that  $D_h$ is a $\xi$-additive local metric for $h$   in the special case $B_r(z) \subset V$.  To deduce the case for general $B_r(z)$, we can directly extend the result~\cite[Lemma 2.19]{lqg-metric-estimates} from the continuous metric setting.  This lemma asserts in the continuous metric setting that, if we know that  $D_h$ is a $\xi$-additive local metric for $h$  in the special case when  $B_r(z) \subset V$, then we can deduce the general case roughly as follows. (We are summarizing their argument briefly here; see the proof of~\cite[Lemma 2.19]{lqg-metric-estimates} for the full details.) If we first consider a field $h$ normalized by a ball $B_{r_0}(z) \subset V$ and then change to a field $\wt h$ normalized by a ball  $B_{r}(z)$, the two fields and internal metrics are related by 
\eqb
\label{eqn-rev-1}
\wt h - h = \wt h_{r_0}(z_0) \qquad \text{and} \qquad D_{\wt h}(\cdot,\cdot;V) = e^{-\xi \wt h_{r_0}(z_0)} D_h(\cdot,\cdot;V) \eqe
Since $B_{r_0}(z) \subset V$, the term $\wt h_{r_0}(z_0)$ in~\eqref{eqn-rev-1} is determined by the field $\wt h$ in $\ol V$. From this, it follows directly from the definition that the $\xi$-additive locality property extends from $h$ to $\wt h$.  Since this argument is equally valid in our more general setting, we may extend~\cite[Lemma 2.19]{lqg-metric-estimates} to our setting and conclude  that $D_h$ is a $\xi$-additive local metric for $h$  for general $B_r(z)$.}
\end{proof}

\subsection{Controlling the behavior of geodesic paths}
\label{sec-geo-behavior}

As we described in Section~\ref{sec-thm-outline}, the main difficulty in extending the methods of~\cite{local-metrics} to the \changes{non-}continuous setting is that, \emph{a priori}, a segment of a $D$-geodesic path in a neighborhood of microscopic Euclidean size could have constant order $D$-length; such neighborhoods roughly correspond to $Q$-thick points of the underlying field.  In this section, we show that this cannot happen.

Throughout this and the next section, we take $h$ to be a whole-plane GFF normalized so that $h_1(0) = 0$.
We begin by defining the event $E_r(z;C)$, for all $z \in \BB{C}$ and $r>0$ for which $B_{2r}(z) \subset U$, as the event that
\eqb
D(\text{across $\BB{A}_{r/2,r}(z)$}) > C^{-1} \frk c_r e^{\xi h_r(z)}
\label{eqn-dist-est-1}
\eqe and
\eqb
D(\text{around $\BB{A}_{r,2r}(z)$}) < C \frk c_r e^{\xi h_r(z)}.
\label{eqn-dist-est-2}
\eqe

The following crucial lemma is our starting point for analyzing the behavior of distances and geodesics for metrics satisfying the conditions of Proposition~\ref{prop-meas-general}.  

\begin{lem}
\label{lem-iterate}
Let $U \subset \BB C$, and let $(h,D)$ be a coupling of  $h$ with a random lower semicontinuous metric $D$ on $U$ that satisfies Axioms~\ref{item-metric-length},~\ref{item-metric-translate}, and~\ref{item-metric-coord}  and is a local $\xi$-additive metric for $h$.
For each $m > 0$ and each $q > 0$, there exists $C  = C(m,q) >1$ such that for each  $\BB r > 0$ and each $z \in U$, the event $E_r(z;C)$ occurs for at least $C^{-1} \log \ep^{-1}$ many values of $r\in [\ep^{1+m} \BB r , \ep \BB r] \cap \{4^{-k} \BB r\}_{k\in\BB N}$ 
with probability $1-O_\ep(\ep^q)$ as $\ep\rta 0$ (with the rate uniform in  $\BB r$ and $z \in K$ for each compact subset $K \subset U$).
\end{lem}

In this section, we always take the parameter $\BB{r}$ in the lemma to equal $1$, but we include the  statement for general $\BB{r} > 0$ because it is an important input for proving distance estimates like Proposition~\ref{prop-two-set-dist}. 

To prove Lemma~\ref{lem-iterate}, we apply a local independence property~\cite[Lemma 3.1]{local-metrics}  for events that depend on the field in disjoint concentric annuli, which we now state.

\begin{lem} \label{lem-annulus-iterate}
Fix $0 < s_1<s_2 < 1$. Let $\{r_j\}_{j\in\BB N}$ be a decreasing sequence of positive numbers with that $r_{j+1} / r_j \leq s_1$ for each $j\in\BB N$, and let $\{E_{r_j} \}_{j\in\BB N}$ be events such that \[ E_{r_j} \in \sigma\left( (h-h_{r_j}(0)) |_{\BB A_{s_1 r_j , s_2 r_j}(0)  } \right) \qquad \forall j \in \BB{N}.\]
 For each $a > 0$ and each $b\in (0,1)$, there exists $p  = p(a,b,s_1,s_2) \in (0,1)$ and $c = c(a,b,s_1,s_2) > 0$ such that, if $
\BB P\left[ E_{r_j}  \right] \geq p$ for all $j\in\BB N$, 
then for every $J \in \BB{N}$, the event $E_{r_j}$ occurs for at least $bJ$ values of $j$ with probability at least $1 - c e^{-aJ}$.
\end{lem}

\changes{Since the metric $D$ in Lemma~\ref{lem-iterate}  is a $\xi$-additive local metric, we can apply Lemma~\ref{lem-annulus-iterate} to prove Lemma~\ref{lem-iterate} by iterating the event $E_r(z;C)$ over concentric annuli:}

\begin{proof}[Proof of Lemma~\ref{lem-iterate}]
Let $K \subset U$ be compact. 
By Axiom~\ref{item-metric-coord}, for each $p \in (0,1)$ and each fixed $z \in K$, we can choose $C>1$ such that $\BB{P}(E_r(z;C)) \geq p$ for every sufficiently small $r>0$.  By Axiom~\ref{item-metric-translate}, for this value of $C$ and all $r>0$ sufficiently small (in a manner depending on $K$ but not on $z$), we have $\BB{P}(E_r(z;C)) \geq p$ for \emph{every} $z \in K$.

Next, \changes{since $D$ is a $\xi$-additive local metric,}  the event $E_r(z;C)$ is a.s.\ determined by the metric \eqb e^{-\xi h_{r}(z)} D(\cdot,\cdot;\BB A_{r/2,2r}(z)).\label{eqn-rev-22} \eqe  Since $D$ is a $\xi$-additive local metric for $h$ and the locality condition is preserved under translations and rescalings, the metric~\eqref{eqn-rev-22} is local for the field $(h(r \cdot + z) - h_{r}(z))|_{r^{-1}(U-z)}$.  
The latter field has the same law as $h$.  Thus, we obtain the desired result by applying Lemma~\ref{lem-annulus-iterate} with $a = q \log 2$ over logarithmically (in $\ep^{-1}$) many values of $r_k \in [\ep^{1+m} \BB r , \ep \BB r] \cap \{4^{-k} \BB r\}_{k\in\BB N}$.
\end{proof}

 \changes{
 Lemma~\ref{lem-iterate} easily implies that a.s.\ $D$-distances between pairs of rational circles are finite.

 \begin{lem}
 \label{lem-rational-dist-finite}
 Let $U \subset \BB{C}$, and let $(h,D)$ satisfy the conditions of Lemma~\ref{lem-iterate}.  Then, if $S$ and $S'$ are fixed connected subsets of $U$ with positive (Euclidean) diameter, then a.s.\ $D(S,S')$ is finite.  In particular, a.s.\ the $D$-distance between all pairs of rational circles in $U$ is finite.
 \end{lem}
 
\begin{proof} 
 Let $K \subset U$ be compact. By Lemma~\ref{lem-iterate}, a.s.\ for $\ep>0$ sufficiently small, we can cover $K$ by finitely many balls $\{B_{r_i}(x_i)\}_{i=1}^n$ with $r_i < \ep$ such that, for each $i$, the $D$-distance around the annulus $\BB{A}_{r_i,2r_i}(x_i)$ is finite.  We can choose $\ep$ sufficiently small such that, for some choice of $i,j \in [1,n]_{\BB{Z}}$, the sets $S$ and $S'$ intersect paths with finite $D$-length around the annuli $\BB{A}_{r_i,2r_i}(x_i)$ and $\BB{A}_{r_j,2r_j}(x_j)$, respectively. Therefore, we can construct a path with finite $D$-length between $S$ and $S'$  by concatenating finitely many segments of paths around small annuli.
 \end{proof}
 }

We now use Lemma~\ref{lem-iterate} to prove that, roughly speaking, paths with close to minimal $D$-length between pairs of points must avoid regions in which the field is too large.  

\begin{lem}
\label{lem-thickbound}
Let $U \subset \BB C$, and let $(h,D)$ be a coupling of  $h$ with a random lower semicontinuous metric $D$ on $U$ that satisfies Axioms~\ref{item-metric-length},~\ref{item-metric-translate}, and~\ref{item-metric-coord} with $\frk c_r = r^{\xi Q + o_r(1)}$ as $\BB{Q} \ni r \rta 0$ for some $Q>0$.  Suppose also that $D$ is a local $\xi$-additive metric for $h$.  For each fixed $M > 0$ \changes{and $q>0$}, we can deterministically choose  $C >1$, $\zeta < Q$, and a function $\delta: [0,\infty) \rta [0,\infty)$ with $\lim_{\ep \rta 0} \delta(\ep) = 0$, such that, for each $x \in U$, the following holds  \changes{on an event $\mcl F = \mcl F(x,\ep)$} with probability $1 - O_\ep(\ep^q)$ (with the rate uniform in $x \in K$ for each compact subset $K \subset U$). Suppose there is a path $P$ between two points $z,w \in U$ with $\len(P;D) \leq D(z,w) + \delta(\ep)$ such that $P$ hits $\partial B_{\ep^2}(x)$ between two times that it hits $\partial B_{2\ep}(x)$. Then there is some $r \in [\ep^{2}  , \ep ] \cap \{4^{-k} \}_{k\in\BB N}$  for which the event $E_r(x;C)$ holds and $\frac{h_r(x)}{\log r^{-1}} \leq \zeta$.  
\end{lem}

The proof of Lemma~\ref{lem-thickbound} will apply the following deterministic lemma.

\begin{lem}
\label{lem-deterministic}
Let $x_1,\ldots,x_N$ be nonnegative real numbers.  For each $c>0$,
\[
\# \left\{ j \in [1,N]_{\BB{Z}} : x_j > c \sum_{i=1}^{j-1} x_i \right\}
\leq \max\left\{ 1, \frac{1}{\log(1 + c)} \log\left(\frac{1}{x_1} \max_{j \in [1,N]_{\BB{Z}}} x_j \right) - \frac{\log c}{\log(1+c)} + 2 \right\}.
\]
\end{lem}

\begin{proof}
Suppose that $j_1 < \cdots < j_K$ are integers in $[1,N]_{\BB{Z}}$ with $K \geq 2$, such that
\eqb
x_{j_k} > c \sum_{i=1}^{j_k-1} x_i \qquad \forall k \in [1,K]_{\BB{Z}}
\label{eqn-defn-xjK}
\eqe
Then
\[
\sum_{i=1}^{j_k} x_i > 
(c+1) \sum_{i=1}^{j_k-1} x_i 
\geq (c+1) \sum_{i=1}^{j_{k-1}} x_i
\qquad \forall k \in [2,K]_{\BB{Z}}.
\]
Iterating this inequality $K-2$ times yields
\[
\sum_{i=1}^{j_{K-1}} x_i > 
(c+1)^{K-2} \sum_{i=1}^{j_1} x_i 
\geq (c+1)^{K-2}x_1.
\]
Therefore, by~\eqref{eqn-defn-xjK} with $k=K$,
\[
x_{j_K} > c (c+1)^{K-2}x_1.
\]
Taking the logarithms of both sides and rearranging gives the desired result.
\end{proof}

\begin{proof}[Proof of Lemma~\ref{lem-thickbound}]
Fix $q>0$, let $C =C(1,q)>0$ be as in Lemma~\ref{lem-iterate}, and for each $x \in U$ and $\ep \in (0,1)$, let $\mcl F = \mcl F(x,\ep)$ be the event of Lemma~\ref{lem-iterate} for these values of $m,q,$ and $C$ and $\BB{r} = 1$, so that $\BB{P}(\mcl F) = 1 - O_\ep(\ep^q)$, with the rate uniform in $x \in K$ for each compact subset $K \subset U$.  On the event $\mcl F$, setting $N := \left\lceil C^{-1}\log \ep^{-1} \right\rceil$, 
we can choose $r_1 < \cdots < r_N$ in $[\ep^{2}  , \ep ] \cap \{4^{-k} \}_{k\in\BB N}$ such that the event $E_{r_j}(x;C)$ holds for each $j \in [1,N]_{\BB{Z}}$.  Since we have $\frk c_r = r^{\xi Q + o_r(1)}$ as $r \rta 0$, we deduce that, for each $\lambda > 0$, the bounds
\eqb
D(\text{around $\BB{A}_{r_{j},2r_{j}}(x)$}) \leq C r_j^{\xi Q - \lambda} e^{\xi h_{r_j}(x)}
 \label{eqn-C-bound-around}
 \eqe
and
\eqb
 D(\text{across $\BB{A}_{r_{j}/2,r_{j}}(x)$}) \geq C^{-1} r_j^{\xi Q + \lambda} e^{\xi h_{r_j}(x)} 
  \label{eqn-C-bound-across}
\eqe
hold on the event $\mcl F$ for all $j \in [1,N]_{\BB{Z}}$ whenever $\ep>0$ is  smaller than some deterministic threshold.

\begin{figure}[ht!] \centering
\begin{tabular}{ccc} 
\includegraphics[width=0.4\textwidth]{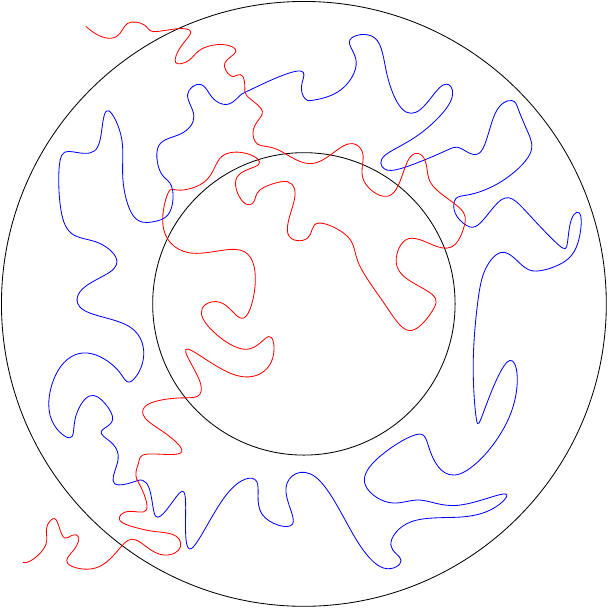}
& \qquad  \qquad &
\includegraphics[width=0.4\textwidth]{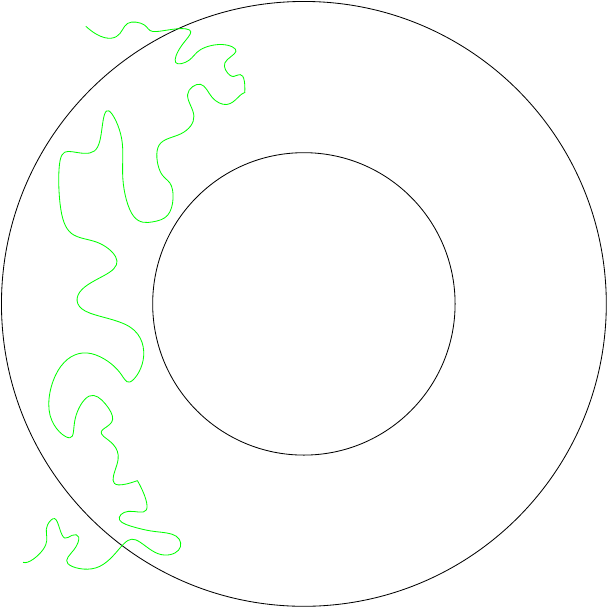}
\end{tabular}
\caption{An illustration of a key step of the proof of Lemma~\ref{lem-thickbound}. The idea is that, if an almost geodesic crosses an annulus, then we can bound the distance across the annulus from above in terms of the length of the path inside the region bounded by the inner circle. We do this by ``rerouting'' $P$ along a path around the annulus to avoid the region bounded by the inner circle. \textbf{Left:} The  path $P$ (red) between $z$ and $w$, whose length is at most $D(z,w) + \delta(\ep)$, hits the circle $\partial B_{r_j}(x)$ between two times that it hits $\partial B_{2r_j}(x)$.  Therefore, given a path $\wt P$ (blue) around the annulus $\BB{A}_{r_j,2r_j}(x)$, \textbf{Right:} we can construct a new path (green) from $z$ to $w$ whose length is at most $\len(\wt P) + \len(P \cap B_{r_j}(x)^c)$. This implies that the distance around the annulus is bounded below by $\len(P \cap B_{r_j}(x)) - \delta(\ep)$.}
\label{fig-avoids-path}
\end{figure}

Now, observe that if  $\wt P$ is any path around $\BB{A}_{r_j,2r_j}(x)$, then we can use $\wt P$ to construct a new path from $z$ to $w$ that avoids $B_{r_j}(x)$ entirely, and whose length is at most $\len(\wt P) + \len(P \cap B_{r_j}(x)^c)$.  (See Figure~\ref{fig-avoids-path}.) Therefore, for each $j \in [1,N]_{\BB{Z}}$,
\[
D(z,w) \leq D(\text{around $\BB{A}_{r_{j},2r_{j}}(x)$}) + \len(P \cap B_{r_j}(x)^c).\]
Since $\len(P;D) \leq D(z,w) + \delta(\ep)$, we deduce that, for each $j \in [1,N]_{\BB{Z}}$,
\[
\len(P;D) \leq D(\text{around $\BB{A}_{r_{j},2r_{j}}(x)$}) + \len(P \cap B_{r_j}(x)^c) + \delta(\ep)\]
and therefore \changes{(on the event $\mcl F$)} \[
D(\text{around $\BB{A}_{r_{j},2r_{j}}(x)$}) \geq \len(P \cap B_{r_j}(x)) - \delta(\ep) \geq \sum_{i=1}^j D(\text{across $\BB{A}_{r_{i}/2,r_{i}}(x)$}) - \delta(\ep).
\]
By~\eqref{eqn-C-bound-across} for $j=1$ and a Gaussian tail bound, we can choose $\delta(\cdot)$ deterministically such that $\delta(\ep) < \frac{1}{2} D(\text{across $\BB{A}_{r_{1}/2,r_{1}}(x)$})$ on an event $\mcl F^* \subset \mcl F$ with probability $1 - O_\ep(\ep^q)$ (also at a rate uniform in $x \in K$).  From now on, we work on the event $\mcl F^*$.  We have
\[
D(\text{around $\BB{A}_{r_{j},2r_{j}}(x)$}) \geq \frac{1}{2} \sum_{i=1}^j D(\text{across $\BB{A}_{r_{i}/2,r_{i}}(x)$}) \qquad \forall j \in [1,N]_{\BB{Z}}.
\]
Combining this inequality with~\eqref{eqn-dist-est-1} and~\eqref{eqn-dist-est-2} yields \[
D(\text{across $\BB{A}_{r_{j}/2,r_{j}}(x)$}) \geq \frac{1}{2C^2} \sum_{i=1}^j D(\text{across $\BB{A}_{r_{i}/2,r_{i}}(x)$}) \qquad \forall j \in [1,N]_{\BB{Z}}.
\]
We now set $x_j = D(\text{across $\BB{A}_{r_{j}/2,r_{j}}(x)$})$  for $j \in [1,N]_{\BB{Z}}$ and apply Lemma~\ref{lem-deterministic}.  We get
\[
N = \#\left\{ j: x_j > \frac{1}{2C^2} \sum_{i=1}^j x_i\right\} \leq \frac{1}{\log(1 + 1/(2C^2))} \log\left(\frac{1}{x_1} \max_{j \in [1,N]_{\BB{Z}}} x_j \right) + 2.
\]
Since $\max_{j \in [1,N]_{\BB{Z}}} x_j \leq \len(P)$, this implies
\eqb
N \leq c + \frac{1}{\log(1 + 1/(2C^2))} \log(1/x_1) ,
\label{eqn-N2}
\eqe
\changes{where $c>0$ is a function of $C$ and $\len(P)$.}  Now,~\eqref{eqn-C-bound-across} yields
\[
\frac{1}{x_1} \leq C r_1^{-\xi \left( Q +
\changes{\lambda/\xi} -  \frac{h_{r_1}(x)}{\log r_1^{-1}} \right)}
\]
Plugging this into~\eqref{eqn-N2} and using the fact that $N \geq C^{-1}\log r_1^{-1}$, we obtain
\[
C^{-1}\log r_1^{-1} \leq c + \frac{\xi}{\log(1 + 1/(2C^2))} \left( Q +\changes{\lambda/\xi} - \frac{h_{r_1}(x)}{\log r_1^{-1}} \right) \log r_1^{-1}
\]
\changes{
or
\[
\left[\frac{\xi}{\log(1 + 1/(2C^2))} \left( Q +\changes{\lambda/\xi} - \frac{h_{r_1}(x)}{\log r_1^{-1}}\right) - C^{-1}  \right]\log r_1^{-1} \geq -c.
\]
Therefore, for any fixed $\eta>0$, if $\ep>0$ is sufficiently small,
\[
\frac{\xi}{\log(1 + 1/(2C^2))} \left( Q +\changes{\lambda/\xi} - \frac{h_{r_1}(x)}{\log r_1^{-1}}\right) - C^{-1}  \geq -\eta
\]
In particular, if we take $\eta = 1/(2C)$, then for $\ep>0$  sufficiently small,
\eqb
\frac{h_{r_1}(x)}{\log r_1^{-1}} \leq  Q +\lambda- \frac{\log(1 + 1/(2C^2))}{2C\xi} \label{eqn-thick-lambda-corr}
\eqe
Fix some value of $\lambda$ small enough that the right-hand side of~\eqref{eqn-thick-lambda-corr} equals some $\zeta < Q$. Then, for this value of $\lambda$ we have shown that, on the event $\mcl F$, for any $\ep>0$ less than some deterministic threshold,
\[
\frac{h_{r}(x)}{\log r^{-1}} \leq  \zeta \qquad \text{for some $r \in [\ep^{2}  , \ep ] \cap \{4^{-k} \}_{k\in\BB N}$}.
\]
}
\end{proof}

From Lemma~\ref{lem-thickbound}, we deduce that \changes{almost} $D$-geodesic segments in small Euclidean neighborhoods must indeed have small $D$-length:

\begin{lem}
\label{lem-geo-doesnt-dawdle}
Suppose we are in the setting of Lemma~\ref{lem-thickbound}, and suppose that $D$ is a complete metric on $U \backslash \{\text{singular points}\}$. 
\changes{Let $K \subset U$ be a compact sets, let $z,w \in K $ be nonsingular points}, and let $P$ be a path from $z$ to $w$ in $K$ with $D$-length at most $D(z,w) + \delta$ for some $\delta \geq 0$.  Then a.s.\
\[
\limsup_{\ep \rta 0} \sup_{y \in K} \len(P \cap B_{\ep}(y)) \leq \delta
\]
\end{lem}
 
 \changes{
We use the following lemma to prove Lemma~\ref{lem-geo-doesnt-dawdle}.}

\changes{\begin{lem}
\label{lem-complete-tez}
In the setting of Lemma~\ref{lem-geo-doesnt-dawdle}, let $t_z^\ep$ denote the last exit of $P$ from $B_\ep(z)$, and  $t_w^\ep$ the first entry of $P$ into  $B_{\ep}(w)$. Then, in the $\ep \rta 0$ limit, we have $P(t_z^\ep) \rta z$ and $P(t_w^\ep) \rta w$ in $(U,D)$.
\end{lem}}

\changes{
\begin{proof}
Since $f(\ep) := \len(P[0,t_z^\ep])$ is a monotone function  $[0,|z-w|] \rta \BB{R}$, the right limit $\lim_{\ep \searrow 0} f(\ep)$ exists.  Therefore, for every $\eta>0$, we can choose $\ep>0$ small enough that $f(a) - f(b) = \len(P[t_z^{a},t_z^b]) < \eta$ for all $0<a<b<\ep$.   It follows that the sequence $\{P(t_z^{1/n})\}_{n \in \BB{N}}$ is a $D$-Cauchy sequence in $U \backslash \{\text{singular points}\}$.  Hence, the sequence converges to a nonsingular point $z' \in U$ in the metric $D$.  By lower semicontinuity of $D$, we have $D(z',z) \leq \liminf_{n \rta \infty} D(z',P(t_z^\ep)) = 0$, i.e., $z'=z$.  The same argument yields $P(t_w^\ep) \rta w$ in $(U,D)$.
\end{proof}
}

\begin{proof}[Proof of Lemma~\ref{lem-geo-doesnt-dawdle}]
Assume that $0< \ep < |z-w|/8$. \changes{As in the statement of the previous lemma, let $t_z^\ep$  denote the last exit of $P$ from $B_{\ep}(z)$, and $t_w^\ep$ the first entry of $P$ into  $B_{\ep}(w)$.}  We will show that, with probability tending to $1$ as $\ep \rta 0$,
\eqb
\changes{ \sup_{y \in K}} \len(P \cap B_{\ep^2/2}(y)) \leq \max\{
\len(P[0,t_z^{4\ep}])
, \len(P[t_w^\ep,T]), 
C \ep^b+\delta\}
\label{eqn-max-bound}
\eqe
where $C= C(1,3)>1$ is as in Lemma~\ref{lem-iterate} and $b$ is some positive constant. \changes{First, we explain why this implies the statement of the lemma. Suppose that~\eqref{eqn-max-bound} holds with probability tending to $1$ as $\ep \rta 0$. Then a.s.\ we can choose a sequence  of values of $\ep$ tending to $0$ such that a.s.\ ~\eqref{eqn-max-bound} holds for infinitely many values of $\ep$ in the sequence.  Since the left-hand side of~\eqref{eqn-max-bound} is nonincreasing and the right-hand side  a.s.\ converges to $\delta$ (by Lemma~\ref{lem-complete-tez}), this means that $\limsup_{\ep \rta 0} \sup_{y \in K}\len(P \cap B_\ep(y)) \leq \delta$.}

To prove that~\eqref{eqn-max-bound} holds with probability tending to $1$ as $\ep \rta 0$, we first note that $\ep$ is small enough that $B_{3 \ep}(y)$ cannot contain both $z$ and $w$ \changes{for any choice of $y \in K$}. Thus, we can classify the points $y \in K$ into two categories.

\medskip

\begin{figure}[ht!] \centering
\begin{tabular}{ccc} 
\includegraphics[width=1\textwidth]{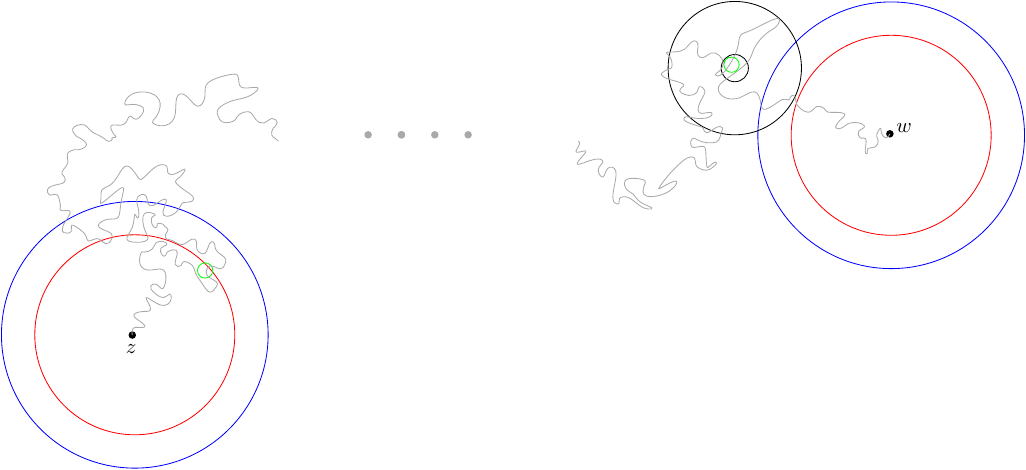}
\end{tabular}
\caption{A sketch of the proof of Lemma~\ref{lem-geo-doesnt-dawdle}.  We want to show that, by choosing $\ep>0$ sufficiently small, we can make the $D$-length of the path $P$ (gray) inside the ball $B_{\ep^2/2}(y)$ arbitrarily small simultaneously for \emph{all} choices of $y \in K$.  
To prove this fact, we choose balls of radius $3\ep$ (boundaries in red) and $4\ep$ (boundaries in blue) centered at the points $z$ and $w$, and we consider separately the cases where $y$ is or is not contained in one of the red balls. (In illustrating both cases, we have colored the boundary of $B_{\ep^2/2}(y)$ in green.)
If $y$ is in one of the red balls, then $B_{\ep^2/2}(y)$ must be contained in the corresponding blue ball, so the parts of the path $P$ in $B_{\ep^2/2}(y)$ must have small $D$-length by the completeness property of $D$. 
If $y$ is not in one of the red balls, then we can choose a point $x$ from a deterministic lattice of $\sim \ep^{-2}$ points such that $B_{\ep^2/2}(y)$ is separated from $z,w$ by the annulus $\BB{A}_{\ep^2,2\ep}(x)$ (black). By a union bound over the  deterministic lattice, we can apply Lemma~\ref{lem-thickbound} to upper bound the $D$-distance around the annulus $\BB{A}_{\ep^2,2\ep}(x)$. 
As we showed in Figure~\ref{fig-avoids-path}, we can use this upper bound to derive an upper bound for the $D$-length of $P$ in $B_{\ep^2}(x)$ (and therefore in $B_{\ep^2/2}(y)$) by ``rerouting'' the path $P$ to avoid the inner ball $B_{\ep^2}(x)$ of the annulus.}
\label{fig-avoids-path-3}
\end{figure}

\noindent\textit{Case 1: $B_{3 \ep}(y)$ contains exactly one of the points $z,w$.}
\medskip

\changes{First, we consider $y \in K$ for which} $z \in B_{3 \ep}(y)$ and $w \notin B_{3 \ep}(y)$.  Then $B_{\ep^2/2}(y) \subset B_{4\ep}(z)$, and the path $P$ can intersect $B_{\ep^2/2}(y)$ only until the last time $t_z^{4\ep}$ that $P$ exits $B_{4\ep}(z)$.  Thus, $\len(P \cap B_{\ep^2/2}(y)) \leq P([0,t_z^{4\ep}])$. Similarly, if \changes{$y$ is such that} $w \in B_{3\ep(y)}$ and $z \notin B_{3\ep}(y)$, then $\len(P \cap B_{\ep^2/2}(y)) \leq  \len(P[t_w^{4\ep},T])$.
\medskip

\noindent\textit{Case 2: $B_{3 \ep}(y)$ contains neither of the points $z,w$.}
\medskip

\changes{Next, we consider $y \in K$ for which} that $z, w \notin B_{\ep^{1/2}}(y)$.  \changes{Since we  must consider all such $y$ simultaneously,} we cannot apply Lemma~\ref{lem-thickbound} to $y$ directly; instead, choose $x \in B_\ep(K) \cap \left( \frac{\ep^2}{4} \BB{Z}^2 \right)$ such that $|x-y|<\frac{\ep^2}{2}$. Observe that $z,w$ lie outside the ball $B_{2\ep}(x)$. Also, if $P$ intersects the ball $B_{\ep^2/2}(y)$, then $P$ must intersect the ball $B_{\ep^2}(x)$.  
By Lemma~\ref{lem-thickbound} and a union bound, on an event $\mcl F_\ep$ with probability tending to $1$ as $\ep \rta 0$, we can find $r  \in [\ep^{2}  , \ep ] \cap \{4^{-k} \}_{k\in\BB N}$ for which the event $E_r(x;C)$ holds and $\frac{h_r(x)}{\log r^{-1}} \leq \zeta$, where $\zeta < Q$.  Henceforth, we assume that $\mcl F_\ep$ occurs. As in the proof of Lemma~\ref{lem-thickbound}, we can use any path $\wt P$ around $\BB{A}_{r,2r}(x)$ to construct a path from $z$ to $w$  that does not intersect $B_{r}(x)$ and whose length is at most $\len(\wt P) + \len(P \cap B_{r}(x)^c)$.  Therefore
\eqb
\len(P \cap B_{\ep^2}(x)) \leq D(\text{around $\BB{A}_{r,2r}(x)$}) + \delta
\label{eqn-B2ep}
\eqe
Since $E_r(x;C)$ occurs and $\frac{h_r(x)}{\log r^{-1}} \leq \zeta$,~\eqref{eqn-B2ep} implies that
\[
\len(P \cap B_{\ep^2}(x)) 
\leq C \frk c_r r^{\xi h_r(x)} + \delta \leq C \frk c_r r^{-\xi \zeta} + \delta.\]  Since $\frk c_r = r^{\xi Q + o_r(1)}$ as $r \rta 0$, we can choose $b>0$ such that $\len(P \cap B_{\ep^2}(x)) 
\leq C  r^{b} + \delta$  for all  $\ep>0$ sufficiently small.
Since $B_{\ep^2/2}(y)$ is contained in $B_{\ep^2}(x)$, we conclude that, if $z, w \notin B_{\ep^{1/2}}(y)$, then, on the event $\mcl F_\ep$, we have $\len(P \cap B_{\ep}(y)) \leq C \ep^{ b} + \delta$.  
\end{proof}

We can also use Lemma~\ref{lem-thickbound} to  bound from above the thickness of points along $D$-geodesics by some deterministic constant strictly less than $Q$:

\begin{lem}
\label{lem-avoids-thick}
In the setting of Lemma~\ref{lem-thickbound},  there is a deterministic $\zeta < Q$ such that, almost surely, for every geodesic in $U$ and every point $y$ on the geodesic minus its endpoints, we have \[ \liminf_{r \rta 0} \frac{h_r(y)}{\log{r^{-1}}} \leq \zeta.\]
\end{lem}

\begin{proof}
Let $P$ be a geodesic in $U$, and let $K \subset U$ be a compact set that contains $P$.    As in the proof of Lemma~\ref{lem-geo-doesnt-dawdle}, we cannot apply Lemma~\ref{lem-thickbound} to a random point $y$ to bound the circle average of the field at $y$.  Instead, we will choose $x$ from a deterministic lattice, sufficiently close to $y$ that the circle averages at $x$ and $y$ are probably close  in value, and we apply Lemma~\ref{lem-thickbound} at $x$ to get a bound for the circle average at $x$.  
Specifically, we first observe that if $y$ is a point on the geodesic $P$ that is at distance greater than $2\ep + \frac{\ep^2}{2}$ from the endpoints of $P$, then we can choose  $x \in B_\ep(K) \cap \left( \frac{\ep^2}{4} \BB{Z}^2 \right)$ such that $P$ hits $\partial B_{\ep^2}(x)$ between two times that it hits $\partial B_{2\ep}(x)$. Now, for fixed $a \in \BB{R}$ and $b \in (0,1/2)$, let $\mcl F_{\ep,a,b}$ be the event that the following two properties hold:
\begin{enumerate}
    \item For every $x \in B_\ep(K) \cap \left( \frac{\ep^2}{4} \BB{Z}^2 \right)$, if a geodesic hits $\partial B_{\ep^2}(x)$ between two times that it hits $\partial B_{2\ep}(x)$, then $h_r(x) \leq a \log r^{-1}$ for some $r \in [\ep^{2}  , \ep ] \cap \{4^{-k} \}_{k\in\BB N}$.
    \item 
    For all $x,y \in K$ with $|x-y|<2r$ and $r \in (0,\ep)$,
    \[
    |h_r(x) - h_{r^{1-b}}(y)| \leq 3\sqrt{10b} \log r^{-1}.
    \]
    \end{enumerate}
By Lemma~\ref{lem-thickbound} and~\cite[Lemma 3.15]{ghm-kpz}, we can choose $a<Q$ such that, for each choice of $b \in (0,1/2)$, we have $\BB{P}(\mcl F_{\ep,a,b}) \rta 1$ as $\ep \rta 0$.  Hence, for such $a$ and $b$, it is almost surely the case that $\mcl F_{\ep,a,b}$ holds for arbitrarily small values of $\ep$. On the event $\mcl F_{\ep,a,b}$, every point $y$ on a geodesic $P$ that is at distance greater than $2\ep + \frac{\ep^2}{2}$ from the endpoints of $P$ must satisfy $h_{r^{1-b}}(y) \leq (a + 3\sqrt{10b}) \log r^{-1}$ for some $r \in [\ep^{2}  , \ep ] \cap \{4^{-k} \}_{k\in\BB N}$.   Since it is a.s.\ the case that $\mcl F_{\ep,a,b}$ holds for arbitrarily small values of $\ep$, this means that for every point $y$ that lies on some geodesic minus its endpoints, we have $h_{r^{1-b}}(y) \leq (a + 3\sqrt{10b}) \log r^{-1}$ for arbitrarily small values of $r$ and so \[ \liminf_{r \rta 0} \frac{h_r(y)}{\log r^{-1}} \leq \zeta:= \frac{a + 3\sqrt{10b}}{1-b}.\] By choosing $b \in (0,1/2)$  small enough that $\zeta < Q$, we obtain the desired result.
\end{proof}

\subsection{Proving Axiom~\ref{item-metric-local}}
\label{sec-measurability}

We now proceed with the proof of Proposition~\ref{prop-meas-general}. 

\changes{
From the fact that $D$ is a $\xi$-additive local metric, we already know that the $D_h$-internal metric $D_h(\cdot,\cdot;U)$ is conditionally independent from $h$ given $h|_U$.  If we also knew that this metric is a.s.\ determined by $h$, then this would imply that $D_h(\cdot,\cdot;U)$ is a.s.\ determined by $h|_U$,\footnote{\changes{This is a general probability fact: if $X,Y$ are random variables and $X$ and $f(X)$ are conditionally independent given $Y$ for some measurable $f$, then given an event $E$ in the $\sigma$-algebra generated by $f(X)$, we have $E \in \sigma(X)$ and therefore $P(E| \sigma(Y)) = P(E| \sigma(Y))^2$; i.e., $E \in \sigma(Y)$.}}  which is the statement of Axiom~\ref{item-metric-local}.} Therefore, to prove Proposition~\ref{prop-meas-general}, it is enough to show that $D$ is a.s.\ determined by $h$.  In other words, it suffices to show that, if $D$ and $\wt{D}$ are conditionally independent samples from the conditional distribution of $D$ given $h$, then a.s.\ $D=\wt{D}$.  

\changes{To prove this result in the continuous metric setting~\cite{local-metrics}, it is enough to show that a.s.\ $D(z,w)=\wt{D}(z,w)$ for each fixed $z,w \in U$, since this means that a.s.\ $D(z,w)=\wt{D}(z,w)$  simultaneously for all $z,w \in \BB{Q}^2 \cap U$, and we can then extend to all $z,w \in U$ since $D$ and $\wt D$ are continuous.  The challenge in our setting is that the metrics are required only to be lower semicontinuous, so we cannot extend the equality $D(z,w) = \wt D(z,w)$ for pairs of rational points $z,w \in U$ to all pairs of points in $U$. However, as we showed in footnote~\ref{footnote-rational}, if we can show that a.s.\ $D(O,O') = \wt D(O,O')$ for all pairs of rational \emph{circles} $O,O'$ in $U$, then we can apply the lower semicontinuity of $D$ and $\wt D$ to deduce that a.s.\ $D = \wt D$.  Thus, we can implement the general approach in~\cite{local-metrics}, where we modify their approach to analyze  distances between pairs of fixed rational circles rather than distances between pairs of fixed points.
}

We first prove a proposition that implies a weaker result: namely, that $D$ and $\wt{D}$ are a.s.\ bi-Lipschitz equivalent with deterministic Lipschitz constant.  \changes{(The proof is very similar to that of~\cite[Theorem 1.6]{local-metrics}.)}

\begin{lem}[Bi-Lipschitz equivalence]
\label{lem-bilipschitz}
Suppose that $(h,D,\wt{D})$ is a coupling of  $h$ with two random lower semicontinuous metrics $D,\wt{D}$ on a connected open set $U \subset \BB{C}$ that are \changes{conditionally iid given $h$. Moreover, assume that each metric}  is local and  $\xi$-additive for $h$ and satisfies Axioms~\ref{item-metric-length},~\ref{item-metric-translate} and~\ref{item-metric-coord}. Then there is a deterministic constant $C$ such that a.s., for any fixed rational circles $O,O'$ in $U$, \[ \wt D( O,O') \leq C D(O,O' ). \]  \changes{(In particular, by lower semicontinuity of $D$ and $\wt D$, this implies that a.s.\ $\wt D(z,w) \leq C D(z,w)$ for any fixed $z,w \in U$.)}
\end{lem}

\begin{proof}
Let $\delta > 0$, and let $P: [0,T] \rta U$ be a path (chosen in a measurable manner) from \changes{a point $z \in O$} to \changes{a point $w \in O'$} parametrized by $D$-length and  whose $D$-length $T$ is at most \changes{$D(O,O') + \delta$}. We may choose a compact subset $K \subset U$ such that $P \subset K$ on an event $\mcl F^*$ with probability $\geq 1 - \delta$. \changes{Exactly as in the continuous metric setting (see \cite[Corollary 1.8]{local-metrics}), we may deduce that $D$ and $\wt D$ are \emph{jointly} local and $\xi$-additive for $h$.  (We do not include the proof of this last statement, instead referring the reader to the proof of~\cite[Corollary 1.8]{local-metrics}, since their proof extends to our setting without any modifications.)  Moreover, since $D$ and $\wt D$ are conditionally iid given $h$, we can take the constants $\frk c_r$ corresponding to the two metrics to be the same.  Thus, by the same argument as in the proof of Lemma~\ref{lem-iterate},} we can choose $C>0$ deterministically such that, on an event $\mcl F \subset \mcl F^*$ with probability  $1  -\delta - O_{\ep}(\ep^q)$,  the following is true. For each $x \in U \cap B_\ep(K) \cap \left( \frac{\ep^2}{4} \BB{Z}^2 \right)$, we can choose $r \in [\ep^{1+m} \BB{r}, \ep \BB{r}]$ such that \eqb \wt D( \text{around $\BB A_{r,2r}(x)$} ) \leq C D( \text{across $\BB A_{r/2,r}(x)$} ) \label{eqn-around-C-across} \eqe 
For the rest of the proof, we work on the event $\mcl F$.  

We now construct a collection of balls $\{B_{r_j/2}(x_j)\}_{j \in [1,J]_{\BB{N}}}$ for some $J \in \BB{N}$, with \eqb x_j \in U \cap B_\ep(K) \cap \left( \frac{\ep^2}{4} \BB{Z}^2 \right) \qquad \text{and} \qquad  r_j \in [\ep^{1+m}, \ep ] \qquad \forall j,\label{eqn-xr} \eqe such that~\eqref{eqn-around-C-across} holds with $x = x_j$ and $r = r_j$ for each $j$.
We construct this collection of balls by the following finite inductive procedure:
\begin{itemize}
    \item 
First, we choose $x_1$ and $r_1$ satisfying~\eqref{eqn-xr}  such that $z \in B_{r_1/2}(x_1)$  and such that~\eqref{eqn-around-C-across} holds with $x = x_1$ and $r = r_1$.  We also set $t_1 = 0$.
\item
Suppose that we have defined $x_j$ and $r_j$ for some $j$, and also some time $t_j \in [0,T]$.  If the ball $B_{r_{j}/2}(x_j)$ contains the entire segment $P([t_{j},T])$, then we set $j = J$ and the procedure terminates.  Otherwise, we let $t_{j+1}$ be the first time after $t_{j}$ that $P$ exits the ball $B_{r_{j}}(x_{j})$, and we choose $x_{j+1}$ and $r_{j+1}$ satisfying~\eqref{eqn-xr} such that $P(t_{j+1}) \in B_{r_{j+1}/2}(x_{j+1})$ and such that~\eqref{eqn-around-C-across} holds with $x = x_{j+1}$ and $r = r_{j+1}$. 
\end{itemize}
Let $C_{r_j}(x_j)$ be a path that disconnects $\partial B_{r_j}(x_j)$ and $\partial B_{2r_j}(x_j)$ whose $\wt{D}$-length is at most $\ep/J$ plus the $\wt{D}$-distance around the annulus $\BB A_{r_j,2r_j}(x_j)$.  
Since the balls $B_{r_{j/2}}(x_{j})$ for $j=[1,J]_{\BB{N}}$ cover $P([0,t_J])$, we can deduce from topological considerations (as in~\cite[Lemma 4.3]{local-metrics}) that the union of the paths $C_{r_{j}}(x_{j})$ is connected.  Moreover, the union of the paths $C_{r_{j}}(x_{j})$ contains a path from $B_{2\ep}(O)$ to $B_{2\ep}(O')$. 
 It follows that, on the event $\mcl F$, the $\wt D$-distance from 
 $B_{2\ep}(O)$ to $B_{2\ep}(O')$ is at most 
\alb
\sum_{j=1}^J \len(C_{r_{j}}(x_{j});\wt D) 
&\leq\sum_{j=1}^J \left[ \wt D( \text{around $\BB A_{r_j,2r_j}(x_j)$} ) + \ep/J\right] \\
&\leq \ep + C \sum_{j=1}^J D( \text{across $\BB A_{r_j/2,r_j}(x_j)$} ) \\
&\leq \ep +  C \sum_{j=1}^J (t_{j+1} - t_{j}) \\
&\leq \ep + C T \leq C D(O,O') + C \delta.
\ale
Since $\wt D$ is lower semicontinuous, sending $\ep \rta 0$ and then $\delta \rta 0$ yields the desired result.
\end{proof}

The proof of Proposition~\ref{prop-meas-general}, like its counterpart in the continuous metric setting, involves applying the Efron-Stein inequality.  To set up our application of this inequality, we introduce some notation.

\begin{defn}[Defining a $\theta$-shifted grid of $\ep \times \ep$ squares]
For $\theta \in [0,1]^2$, and let $\mcl G_\theta$ be the randomly shifted square grid given by the horizontal and vertical lines joining points of $\BB Z^2 + \theta$.  We define the set of squares $\mcl S_\theta^\ep$ as the set of open $\ep \times \ep$ squares given by the connected components of the complement of the rescaled grid $\ep \mcl G_\theta$. 
\end{defn}

We will use this grid of squares, with $\theta$ chosen randomly from Lebesgue measure on $[0,1]^2$, to decompose the variance of $D(O,O')$ given $h$ and $\theta$, for a pair of rational circles $O,O'$, in terms of the metrics $D^S$ for $S \in \mcl S_{\theta}^\ep$, where $D^S$ denotes the new metric obtained by $D$ by resampling $D(\cdot,\cdot;S\cap U)$ from its conditional law given $(h,\theta)$.
We will apply the Efron-Stein inequality to write
\[
\Var[D(O,O')|h,\theta] \leq \frac{1}{2}\sum_{S \in \mcl S_{\theta}^\ep} \BB E\left[(D^S(O,O') - D(O,O'))^2 | h,\theta\right].
\]
To justify this application of the Efron-Stein inequality, we require the following two lemmas, which generalize the results stated in~\cite[Lemmas 5.1-5.4]{local-metrics}.

\begin{lem}
\label{lem-5.1}
In the setting of Proposition~\ref{prop-meas-general}, we can choose a deterministic $C>0$ such that, for each fixed $V \subset U$ and rational circles $O,O' \subset V$, a.s.\ \eqb 
\label{eqn-5.1}
C^{-1} \BB E[D(O,O';V)|h] \leq D(O,O';V) \leq C \BB E[D(O,O';V)|h]
\eqe 
and\changes{, for any fixed $q>0$,} \eqb 
 \BB{E}[D(O,O';V)^q|h] < \infty \qquad 
 \label{eqn-finite-cond-moments}
 \eqe
 Moreover, if $(h,D) \eqD (h,D')$, then a.s.\ $D,D'$ are bi-Lipschitz equivalent with Lipschitz constant $C^2$. 
\end{lem}

\begin{proof}
The proof of~\eqref{eqn-5.1} is identical to  the proof of~\cite[Lemma 5.1]{local-metrics} with $z,w$ replaced by $O,O'$. The bound~\eqref{eqn-finite-cond-moments} follows  immediately from~\eqref{eqn-5.1} \changes{and Lemma~\ref{lem-rational-dist-finite}}.
Finally, if $(h,D) \eqD (h,D')$, then~\eqref{eqn-5.1} implies that \[
C^{-2} D(O,O') \leq D'(O,O')
\leq C^{2} D(O,O')
\] for all $O,O'$; since $D,D'$ are lower semicontinuous, this implies that a.s.\ $D,D'$ are bi-Lipschitz equivalent with Lipschitz constant $C^2$.   
\end{proof} 

\begin{lem}
\label{lem-threelemmas}
In the setting of Proposition~\ref{prop-meas-general},
let $\theta$ be sampled uniformly from Lebesgue measure on $[0,1]^2$, and let $\mcl S_\theta^\ep$ be the set of open $\ep \times \ep$ squares
in $U$ with vertices in $\ep(\BB{Z}^2 + \theta)$. Then
\begin{itemize}
    \item For any fixed $\ep>0$ and any path $P$ in $U$ with finite $D$-length chosen in a manner depending only on $h$ and $D$, \changes{a.s.\ }
    \eqb
    \len(P;D) = \sum_{S \in \mcl S_{\theta}^\ep} \len(P \cap S;D).
    \label{eqn-5.3}
    \eqe
    \item The metric $D$ is a.s.\ determined by $h, \theta$ and the set of internal metrics $\{D(\cdot,\cdot;S\cap U): S \in \mcl S_{\theta}^\ep\}$.
    \item The internal metrics $\{D(\cdot,\cdot;S\cap U): S \in \mcl S_{\theta}^\ep\}$ are conditionally independent given $h$ and $\theta$.
\end{itemize}
\end{lem}

\begin{proof}
These three assertions were stated as~\cite[Lemmas 5.2]{local-metrics},~\cite[Lemmas 5.3]{local-metrics} and~\cite[Lemmas 5.4]{local-metrics}, respectively, in the case in which $D$ is continuous.  The proof of the first and third assertions in that setting apply directly here as well.  

To prove the second assertion, we condition on $h, \theta$, and the internal metrics $\{D(\cdot,\cdot;S\cap U): S \in \mcl S_{\theta}^\ep\}$, and we consider two conditionally independent samples $D$ and $D'$ from the conditional law of $D$.  To prove the assertion, it is enough to show that $D=D'$ almost surely.  Suppose that $O,O'$ are fixed rational balls in $U$, and let $P$ be a path from $O$ to $O'$, chosen in a manner depending only on $D$.  The first assertion of the lemma gives
\[
    \len(P;D) = \sum_{S \in \mcl S_{\theta}^\ep} \len(P \cap S;D) \qquad \text{a.s.}
\]
\changes{In other words, if $T_\zeta$ is the $\zeta$-neighborhood of the grid boundary of the squares in $S_{\theta}^\ep$, then a.s.\ $\lim_{\zeta \rta 0} \len(P \cap T_\zeta ;D) = 0$.  It follows from~\eqref{eqn-5.1} of Lemma~\ref{lem-5.1} that a.s.\ the same limit holds with $D$ replaced by $D'$. Thus,}
\[
    \len(P;D') = \sum_{S \in \mcl S_{\theta}^\ep} \len(P \cap S;D') \qquad \text{a.s.}
\]
Since the internal metrics of $D$ and $D'$ on $S\cap U$ agree for each $S \in \mcl S_{\theta}^\ep$, we deduce that $P$ has the same length in the metrics $D$ and $D'$.  Thus, a.s.\ $
D(O,O') =
D'(O,O')$. Since $D,D'$ are lower semicontinuous,  this implies that $D=D'$ a.s., proving the second assertion of the lemma.
\end{proof}

We now prove Proposition~\ref{prop-meas-general} in a manner analogous to~\cite[Theorem 1.7]{local-metrics}.  \changes{The} key difference in our proof is the application of Lemma~\ref{lem-geo-doesnt-dawdle} to control the length of $P$ in a small square. %Observe that, in contrast to the proof ~\cite[Theorem 1.7]{local-metrics}, our proof requires the introduction of the second small parameter $\delta > 0$.
 
\begin{proof}[Proof of Proposition~\ref{prop-meas-general}]
As noted at the beginning of this section, it suffices to prove that $D$ is a.s.\ determined by $h$, since Axiom~\ref{item-metric-local} then follows from the fact that $D$ is a $\xi$-additive local metric.

First, since $D$ is almost surely determined by the internal metric $D(\cdot,\cdot;V)$ for bounded and connected subsets $V \subset U$, it is enough to prove the result for $U$ a bounded and connected open set.
As in the proof of Lemma~\ref{lem-5.1}, it suffices to show that, for each fixed pair of rational circles $O,O' \subset U$, the quantity $D(O,O')$ is a.s.\ determined by $h$.

By Lemma~\ref{lem-threelemmas}, $D$ is a.s.\ determined by its internal metrics on the squares $S \in \mcl S_{\theta}^\ep$, and these internal metrics are conditionally independent given $h$. Thus, by the Efron-Stein inequality, if we let $D^S$ denote the new metric obtained by $D$ by resampling $D(\cdot,\cdot;S\cap U)$ from its conditional law given $(h,\theta)$, then
\eqb
\Var[D(O,O')|h,\theta] \leq \frac{1}{2}\sum_{S \in \mcl S_{\theta}^\ep} \BB E\left[(D^S(O,O') - D(O,O'))^2 | h,\theta\right].
\label{eqn-efron1}
\eqe
Since $(h,D^S) \eqD (h,D)$, the law of $D^S(O,O') - D(O,O')$ is symmetric about the origin, and so $\BB E\left[(D^S(O,O') - D(O,O'))^2 | h,\theta\right] = 2\BB E\left[(D^S(O,O') - D(O,O'))_+^2 | h,\theta\right]$, where $x_+ = \max\{x,0\}$.  Hence, we can rewrite the inequality~\eqref{eqn-efron1} as
\eqb
\Var[D(O,O')|h,\theta] \leq \sum_{S \in \mcl S_{\theta}^\ep} \BB E\left[(D^S(O,O') - D(O,O'))_+^2 | h,\theta\right].
\label{eqn-efron2}
\eqe
Since $(h,D^S) \eqD (h,D)$,  Lemma~\ref{lem-5.1} implies that we can choose a deterministic $C>0$ such that a.s.
\eqb
\label{eqn-DDS-lipschitz}
C^{-2} D(O,O')
\leq
D^S(O,O')
\leq C^{2} D(O,O') \qquad \text{\changes{for all rational circles $O,O'$ in $U$}}.
\eqe
%Fix $\delta> 0$, and let $P$ be a fixed path from $O$ to $O'$ that depends only on $(h,D)$ with $D$-length at most $D(O,O') + \delta$. 
\changes{Since $D$ is lower semicontinuous, there exist points $z \in O$ and $w \in O'$ such that $D(z,w) = D(O,O')$.  Let $P$ be a geodesic joining $z$ and $w$, so that $\len(P;D) = D(O,O')$.} For each $S' \in \mcl S_{\theta}^\ep$ different from $S$, the $D^S$-length and $D$-length of $P \cap S'$ agree.  Thus,~\eqref{eqn-5.3} of Lemma~\ref{lem-threelemmas} combined with~\eqref{eqn-DDS-lipschitz} gives
\eqb
C^{-2} \len(P \cap S; D) 
\leq \len(P \cap S; D^S)
\leq C^{2} \len(P \cap S; D).
\label{eqn-dds-c2}
\eqe
As in the proof of the second assertion of Lemma~\ref{lem-threelemmas}, this implies that
\[
\len(P;D^S) = \sum_{S' \in S_{\theta}^\ep} \len(P \cap S';D^S)
\]
By~\eqref{eqn-dds-c2}, the latter is at most $\changes{\len(P;D)}+ C^2 \len(P \cap S; D)$.  Therefore, a.s.
\[
(D^S(O,O')-D(O,O'))_+ \leq C^2  \len(P \cap S;D)
\]
Plugging this into~\eqref{eqn-efron2} and applying the Cauchy-Schwartz inequality yields
\alb
\Var[D(O,O')|h,\theta]
&\leq \BB E\left[ \sum_{S' \in \mcl S_{\theta}^\ep} (C^2  \len(P \cap S;D)  )^2 | h,\theta \right]\\
&\leq C^4 \BB E\left[ \sum_{S' \in \mcl S_{\theta}^\ep} \len(P \cap S;D) \left( \max_{S \in \mcl S_\theta^\ep} \len(P \cap S;D)\right)| h,\theta \right] \\%+ 2\delta^2 \# \mcl S_{\theta}^\ep \\
&\leq C^4 \BB E\left[(D(O,O') )^2| h,\theta \right]^{1/2} \BB E \left[ \left( \max_{S \in \mcl S_\theta^\ep} \len(P \cap S;D)\right)^2| h,\theta \right]^{1/2} %+ 2\delta^2 \# \mcl S_{\theta}^\ep
\ale   
By Lemma~\ref{lem-5.1}, the quantity $\BB E\left[(D(O,O'))^2| h,\theta \right]^{1/2}$ is finite almost surely. Moreover, by Lemma~\ref{lem-geo-doesnt-dawdle}, \[  
\lim_{\ep \rta 0} \max_{S \in \mcl S_\theta^\ep} \len(P \cap S;D) =0 \qquad \text{a.s.}.
\] Since $\len(P \cap S;D) \leq D(O,O') $ for all $S$, the bounded convergence theorem gives \[ \lim_{\ep \rta 0} \BB E \left[ \left( \max_{S \in \mcl S_\theta^\ep} \len(P \cap S;D)\right)^2| h,\theta \right]^{1/2} = 0 \qquad \text{a.s.}
\]
Therefore a.s.\ $\Var[D(O,O')|h,\theta] \rta 0$ as $\ep \rta 0$, so $D(O,O')$ is a.s.\ determined by $h$, as desired.
\end{proof}

We can now check that subsequential limits of the rescaled $\ep$-LFPP metrics~\eqref{eqn-defn-rescaled-LFPP} satisfy Axiom~\ref{item-metric-local}.

\begin{proof}[Proof of Lemma~\ref{cor-axiom-2-check}]
In the special case in which $h$ is a whole-plane GFF normalized so that $h_1(0) = 0$, the result follows from combining Propositions~\ref{prop-check-3.5-axioms} and~\ref{prop-meas-general}.  For  general $h$, let $V, W$ be bounded open sets with $\ol V \subset W$ and $\ol W \subset U$. 
By Lemma~\ref{lem-internal-radius}, for each $x \in \ol V$, we can choose $R_x \in \sigma\left( h|_U \right)$ such that $\frk a_{\ep_n}^{-1} {\wh D^{\ep_n}}_h(\cdot,\cdot;W)$ converges a.s. to $D_h(\cdot,\cdot;W)$ on $B_{R_x}(x)$.  Therefore, $h|_U$ a.s. determines the pair $
(R_x, \frk a_{\ep}^{-1} D_h(\cdot, \cdot; W)|_{B_{R_x}(x)}) \in \sigma\left( h|_U \right)$. Since we can cover $\ol V$ by finitely many such balls $B_{R_x}(x)$, we deduce that $h|_U$ a.s. determines the internal metric $D_h(\cdot, \cdot; W)$ on $V$. 
Letting $V$ and $W$ increase to all of $U$ proves the lemma.
\end{proof}

\subsection{Completing the proof of Theorem~\ref{thm-lfpp-axioms}}
\label{sec-complete}

To complete the proof of Theorem~\ref{thm-lfpp-axioms}, we show that we can choose a subsequence $\ep_n \rta 0$ for which LFPP distances between pairs of rational circles converge a.s.\ to the corresponding $D_h$-distances.
    
    \begin{lem}
    \label{lem-as-circles}
    For every sequence of values of $\ep$ tending to zero as $\ep \rta 0$, we can choose a subsequence $\ep_n$ for which  \changes{ $a_{\ep_n}^{-1} D^{\ep_n}_h(O,O') \rta D_h(O,O')$} a.s.\ for all pairs of rational circles $O,O'$. 
    \end{lem}

 \changes{We note that a weaker version of this lemma was proved in~\cite[Theorem 1.2]{dg-supercritical-lfpp}.  This lemma generalizes this earlier result in three ways: by considering rational circles and not just rational points, by strengthening the topology of convergence, and by considering a whole-plane GFF plus a bounded continuous function (and not just a whole-plane GFF).

To prove Lemma~\ref{lem-as-circles}, we first show that, for some choice of subsequence $\{\ep_n\}$, the $a_{\ep_n}^{-1} D^{\ep_n}_h$-distances between rational circles and around rational annnuli converge a.s.\ to \emph{some} collection of random variables (for $h$ a whole-plane GFF).

\begin{lem}
\label{lem-as-circles-prelim}
Let $h$ be a whole-plane GFF.  We can define random variables $\rng{D}_h(O,O')$  and $\rng{D}_h(A)$, for all pairs of rational circles $O,O'$ and all rational annuli $A$, such that the following is true. (Note that $\rng{D}_h$ is \emph{not} defined as a metric.)
\begin{enumerate}
\item
    For every sequence of values of $\ep$ tending to zero as $\ep \rta 0$, we can choose a subsequence $\ep_n$ for which a.s.\  \eqb
\label{eqn-joint-conv}
\frk a_{\ep_n}^{-1} D_h^{\ep_n}(O,O') \rta \rng{D}_h(O,O') \qquad \text{for all rational circles $O,O'$}
\eqe
and
\eqb
\label{eqn-joint-conv-around}
\frk a_{\ep_n}^{-1} D_h^{\ep_n}(\text{around $A$}) \rta \rng{D}_h(\text{around $A$}) \qquad \text{for all rational annuli $A$}
\eqe
\item
If $O_n$ and $O_n'$ are sequences of rational circles surrounding $z$ and $w$, respectively, and whose radii shrink to zero, then the random variables $\rng{D}_h(O_n,O_n')$ converge a.s.\ to $D_h(z,w)$.
\end{enumerate}
\end{lem}

\begin{proof}
We begin by restating~\cite[Equations 5.1-5.3]{dg-supercritical-lfpp}. These equations assert that, for every sequence of $\ep$-values tending to zero, we can choose a subsequence $\ep_n \rta 0$ along \changes{which}  the joint law of $h$, the distances 
\[
\frk a_{\ep_n}^{-1} D_h^{\ep_n}(O,O'), \qquad  \text{for all rational circles $O,O'$}
\]
and the distances
\[
 \frk a_{\ep_n}^{-1} D_h^{\ep_n}(\text{around $A$}), \qquad \text{for all rational annuli $A$}
\]
converges to the joint law of $h$ and some collection of random variables $\rng{D}_h(O,O')$ and $\rng{D}_h(\text{around $A$})$.  (The convergence of the fields $h \rta h$ is in the distributional topology.)  By Lemmas~\ref{cor-axiom-2-check} and~\ref{lem-prob}, this convergence occurs in probability, and therefore a.s.\ along a further subsequence.

Finally, the second assertion is a restatement of~\cite[Lemma 5.3]{dg-supercritical-lfpp}.\footnote{\changes{More precisely, that lemma asserts that the limit of the random variables $\rng{D}_h(O_n,O_n')$ exists a.s.\ and defines $D_n(z,w)$ as its limit.  The authors in~\cite{dg-supercritical-lfpp} then show that the limiting function $D_h$ is a metric, is lower semicontinuous, and satisfies the convergence statement~\eqref{eqn-limit-law}.}}
\end{proof}
}
\begin{proof}[Proof of Lemma~\ref{lem-as-circles}]
By Proposition~\ref{prop-weyl-scaling}, it suffices to consider the case in which $h$ is a whole-plane GFF.  Let $\{\ep_n\}$ and the random variables $D_h^{\ep_n}(O,O')$ and $D_h^{\ep_n}(\text{around $A$})$ be defined as in Lemma~\ref{lem-as-circles-prelim}.  By Lemma~\ref{lem-as-circles-prelim}, it suffices to show that  a.s.\
\eqb
\label{eqn-agree-rational}
D_h(O,O') = \rng{D}_h(O,O') \qquad  \text{for all pairs of rational circles $O, O'$}
\eqe
The inequality $D_h(O,O') \leq \rng{D}_h(O,O')
$ follows directly from the definition of the lower semicontinuous topology. For the converse inequality, let $z \in O$ and $w \in O'$ be such that $D_h(z,w) = D_h(O,O')$.   \changes{We now apply~\cite[Lemma 5.12]{dg-supercritical-lfpp}, which asserts that there exist nested sequences $O_z^n$ and $O_w^n$ of rational circles with positive radii that shrink to $z$ and $w$, respectively, such that the following is true.  Let $\hat{O}_z^n$ denote the rational circle with the same center as $O_z^n$  and twice the radius, and let $A_z^n$ denote the annulus bounded by $O_z^n$ and $\hat{O}_z^n$.  We define $\hat{O}_w^n$ and $A_w^n$ analogously. Then a.s.
\[
\lim_{n \rta \infty} \rng{D}_h(\text{around $A^n_z$}) = \lim_{n \rta \infty} \rng{D}_h(\text{around $A^n_w$}) = 0.
\]
By~\eqref{eqn-joint-conv} and~\eqref{eqn-joint-conv-around}, 
\[
\rng{D}_h(O,O') \leq \lim_{n \rta \infty} \rng{D}_h(O_z^n,O_w^n) +
\lim_{n \rta \infty} \rng{D}_h(\text{around $A^n_z$}) + \lim_{n \rta \infty} \rng{D}_h(\text{around $A^n_w$}) = \lim_{n \rta \infty} \rng{D}_h(O_z^n,O_w^n).
\]
By Lemma~\ref{lem-as-circles-prelim}, the latter a.s.\ equals $D_h(z,w) = D_h(O,O')$. This proves $\rng{D}_h(O,O') \leq D_h(O,O')$.}
\end{proof}
\begin{proof}[Proof of Theorem~\ref{thm-lfpp-axioms}]
By Proposition~\ref{prop-weyl-scaling},  for every sequence of values of $\ep$ tending to zero, we can find a subsequence $\ep_n$ for which the weak limit~\eqref{eqn-limit-law} exists.  By combining Propositions~\ref{prop-weyl-scaling} and~\ref{prop-check-3.5-axioms} and Lemma~\ref{cor-axiom-2-check}, we conclude that each subsequential limiting metric $D_h$ satisfies the five axioms in Definition~\ref{defn-weak-metric} of a weak LQG metric.  Moreover, Axiom~\ref{item-metric-local} and Lemma~\ref{lem-prob} together imply that the convergence~\eqref{eqn-limit-law} occurs almost surely.  Finally, Lemma~\ref{lem-as-circles} gives the desired a.s.\ convergence of distances between pairs of rational circles, and~\eqref{eqn-defn-cr} is satisfied by our definition of $\frk c_r$ in the proof of Proposition~\ref{prop-check-3.5-axioms}.

To complete the proof of the theorem, we now extend the definition of $D_h$ to the case of a whole-plane GFF $h$ plus an unbounded continuous function $f$, so that we have a well-defined weak LQG metric $h \mapsto D_h$. \footnote{(The convergence statements in Theorem~\ref{thm-lfpp-axioms} are stated only for $h+f$ where $f$ is bounded; however, the definition of a weak LQG metric stipulates that $h \mapsto D_h$ must satisfy the five axioms even for unbounded $f$.) }

 For each open and bounded subset $V \subset \BB C$, we set  $D_{h+f}^V := D_{h+\phi_V f}(\cdot,\cdot;V)$, where $\phi_V$ is a smooth compactly supported bump function which is identically equal to 1 on $V$.
By applying Axiom~\ref{item-metric-local} in the setting of a whole-plane GFF plus a bounded continuous function, we deduce that the metric $D_{h+f}^V$ is a.s.\ determined by $(h+\phi_V f)|_V = (h+f)|_V$, in a manner which does not depend on $\phi_V$.
We define the metric $D_{h+f}$ as the length metric such that, if $P$ is a continuous path in $\BB{C}$, its 
$D_{h+f}$-length is equal to its $D_{h+f}^V$-length, where $V\subset\BB C$ is a bounded open set that contains $P$.   
We define $D_{h+f}(z,w)$ for $z,w\in\BB C$ to be the infimum of the $D_{h+f}$-lengths of continuous paths from $z$ to $w$. Then $D_{h+f}$ is a length metric on $\BB C$ which is a.s.\ determined by $D_{h+f}$ and which satisfies $D_{h+f}(\cdot,\cdot;V) = D_{h+f}^V$ for each bounded open set $V\subset\BB C$.  It is easy to check that $D_{h+f}$ satisfies the axioms in Definition~\ref{defn-weak-metric}, so that the mapping   $h \mapsto D_{h}$ is a weak $\gamma$-LQG metric, as desired.
\end{proof}

\section{Properties of weak LQG metrics}
\label{sec-properties}

Throughout this and the next section, we let $h$ be a whole-plane GFF plus a continuous function (unless specified otherwise), and we let $D_h$ be the metric associated to $h$ by a weak LQG metric $h \mapsto D_h$.  Our main building blocks for proving properties of $D_h$ are Propositions~\ref{prop-two-set-dist} and~\ref{prop-metric-scaling}, which are proven exactly as in the continuous metric setting.  \changes{In what follows, we outline the proofs of these propositions in the continuous setting, and we describe how to adapt those proofs to our more general setting.}

\begin{proof}[Proof of Proposition~\ref{prop-two-set-dist}]
\changes{
This property was stated in the continuous metric setting as~\cite[Proposition 3.1]{lqg-metric-estimates}.\footnote{\changes{We will not restate~\cite[Proposition 3.1]{lqg-metric-estimates} here since the statement is identical to Proposition~\ref{prop-two-set-dist} except that~\cite[Proposition 3.1]{lqg-metric-estimates} applies only to the continuous metric setting.}} 
Their proof extends to our setting with almost no modifications. We now outline their proof and point out what needs to be adjusted for our setting, referring the reader to their work for further details.

The idea of their proof is to cover $\BB{r} U$ by small annuli $\BB{A}_{r/2,r}(z)$ with radii $r \in [\ep^2  \BB{r}, \ep \BB{r}]$, such that with extremely high probability, the $D_h$-distance across $\BB{A}_{r/2,r}(z)$ is at least a constant times $\frk c_r e^{\xi h_r(z)}$, and the $D_h$-diameter of the circle $\partial B_r(z)$ of the  annulus is at most  a constant times $\frk c_r e^{\xi h_r(z)}$.  Their proof uses this set of annuli to obtain both upper and lower bounds on $D_h(\BB r K_1,\BB r K_2 ; \BB r U)$.
\begin{itemize}
\item
Every path between the two sets $\BB{r} K_1$ and $\BB{r} K_2$ must cross at least one of these annuli, so they can bound its length from below in terms of the infimum of possible values $\frk c_r e^{\xi h_r(z)}$ for the set of annuli.
\item
We can ``string together'' a collection of the annuli to get a path from  $\BB{r} K_1$ to $\BB{r} K_2$ whose length is similarly bounded above
\end{itemize}
They then finish the proof by applying bounds for the terms $\frk c_r$ and $e^{\xi h_r(z)}$ that hold equally in our setting.  

To adapt this proof to our setting, we apply Lemma~\ref{lem-iterate} to define a collection of annuli with the same property, with one important difference: instead of having an upper bound for the $D_h$-diameter of the circle $\partial B_r(z)$, we have an upper bound for the length around the annulus $\BB{A}_{r/2,r}(z)$.  This modification, however, does not affect the rest of the proof, since one can still ``string together'' these paths around the annuli $\BB{A}_{r/2,r}(z)$ to get a path from  $\BB{r} K_1$ to $\BB{r} K_2$.}
\end{proof}

\changes{
\begin{proof}[Proof of Proposition~\ref{prop-metric-scaling}]
Proposition~\ref{prop-metric-scaling} was stated in the continuous metric setting as~\cite[Theorem 1.5]{lqg-metric-estimates}, and we can directly apply the proof of this theorem to our setting.  In what follows, we  describe the general idea of their proof and how to make the minor modifications to adapt it to our setting.

The idea of the proof of~\cite[Theorem 1.5]{lqg-metric-estimates} is to compare $D_h$-distances to 
distances in what they call the \emph{discretized $\ep$-LFPP metric} $\wt D^\ep_h$.\footnote{
\changes{If we consider $U \cap (\ep \BB{Z}^2)$ as a graph with two edges adjacent if the points are at distance $\ep$ or $\sqrt{2} \ep$, then the $\wt D^\ep_h$-distance between two vertices of the graph is the minimum, over all paths between the two points, of the sum of values $e^{\xi h_\ep(z)}$ for all vertices $z$  along the path.}}
%The reason they consider a version of LFPP distances is that the key ingredient they use to analyze the quotient $\frk c_{\BB{r}}/\frk c_{\delta \BB{r}}$ is the estimate for distances between sets in  Proposition~\ref{prop-two-set-dist}, which relates  The circle average term $e^{\xi h_{\BB{r}}}$ term in that estimate makes it possible to relate $\frk c_{\BB{r}}/\frk c_{\delta \BB{r}}$ to LFPP distances.  
By applying their version of Proposition~\ref{prop-two-set-dist} in the continuous metric setting~\cite[Proposition 3.1]{lqg-metric-estimates}, the authors of~\cite{lqg-metric-estimates} show that, with high probability, 
\eqb
\frk c_{\BB{r}}/\frk c_{\delta \BB{r}} = \delta^{o_1(\delta)} \cdot \left(\text{$\wt D^{\delta \BB{r}}_h$-distance between two sides of $\BB{r} \BB{S}$}\right)
\label{eqn-rev-23}
\eqe
To complete their proof of~\cite[Theorem 1.5]{lqg-metric-estimates}, they apply a lemma~\cite[Lemma 3.6]{lqg-metric-estimates} which asserts that this $\wt D^{\delta \BB{r}}_h$-distance is of order $\delta^{-\xi Q + o_\delta(1)}$, uniformly in $\BB{r}$.

The proof of~\eqref{eqn-rev-23} in~\cite{lqg-metric-estimates} extends directly to our setting without any modifications, but their proof of the estimate for the $\wt D^{\delta \BB{r}}_h$-distance relies on computations in the continuous metric setting. The specific issue is their justification of equation (3.19) in their proof of~\cite[Lemma 3.6]{lqg-metric-estimates}, which  asserts that, with probability tending to $1$ as $\delta \rta 0$, the $\wt D^{\delta \BB{r}}_h$-distance in the square $\BB{S}$ between the leftmost and rightmost sets of vertices in $\BB{S} \cap (\ep \BB{Z}^2)$ is $\delta^{-\xi Q + o_\delta(1)}$.  To justify this estimate in our more general setting, we refer to~\cite[Lemma 2.11]{dg-supercritical-lfpp}, which is a slightly different formulation of the same estimate that is stated for general $\xi>0$. We will not reproduce the statement of that lemma here, but we will mention that it differs from the estimate~\cite[Equation 3.19]{lqg-metric-estimates} just described in two ways.  First, the distance is across an annulus instead of across a square; and second,  their  definition of the discretized LFPP metric differs slightly from the definition in ~\cite{lqg-metric-estimates}.  These differences are relatively minor: their lemma applies to squares as well as annuli, and the two discretized LFPP metrics are the same up to grid rescaling since the  corresponding discrete Gaussian free fields have the same covariance structure.\footnote{\changes{We note that the version of the discrete LFPP metric in~\cite[Lemma 3.6]{lqg-metric-estimates} is defined only at dyadic scales.  This  does not present an issue, since we need only prove~\eqref{eqn-metric-scaling} for $\delta$ and $r$ of the form $2^k$ for $k \in \BB{Z}$.  We can then directly deduce~\eqref{eqn-metric-scaling} for general $\delta$ and $r$ since, by~\eqref{eqn-scaling-constant}, there is some value of $C>1$ for which $C^{-1} \leq \frk c_{r'} / \frk c_{r} \leq C$ whenever $1/2 \leq r'/r \leq 2$.  Alternatively, one could adapt the treatment in~\cite{dg-supercritical-lfpp} to consider the same discretized LFPP metric  $\wt D^{\ep}_h$ from~\cite{lqg-metric-estimates}, which makes sense at all real scales.} }
\end{proof}
}

We can apply Proposition~\ref{prop-two-set-dist} directly to prove the following tightness result for distances between compact sets.

\begin{lem}
If $U\subset\BB C$ is open and $K_1,K_2\subset U$ are connected, disjoint compact sets (that are allowed to be singletons), then the laws of 
\eqb\label{eq:tight}
\frk c_r^{-1} e^{-\xi h_r(0)} D_h (r K_1, rK_2; rU) \quad \text{and} \quad \left( \frk c_r^{-1} e^{-\xi h_r(0)} D_h (r K_1, rK_2; rU) \right)^{-1}
\eqe
for varying $r$ are tight.
\label{lem-tight}
\end{lem}

One immediate consequence of Lemma~\ref{lem-tight} is that the $D_h$-distance between two fixed points is a.s.\  finite.  

Our proof of Lemma~\ref{lem-tight} uses the following elementary lemma for lower semicontinuous metrics:

\begin{lem}
Let $D$ be a lower semicontinuous metric on a set $U \subset \BB{C}$. For every $X \subset U$, we have $\diam(X;D) = \diam(\ol X;D)$, where $\ol X$ is the closure of $X$ in the Euclidean topology on $U$. 
\label{lem-closure}
\end{lem}

\begin{proof}
If $z,w \in \ol X$, then we can choose sequences $z_n,w_n$ of points in $X$ with $z_n \rta z$ and $w_n \rta w$. Since $D$ is lower semicontinuous, $D(z,w) \leq \liminf_{n \rta \infty} D(z_n,w_n) \leq \diam(X;D)$.
\end{proof}

\begin{proof}[Proof of Lemma~\ref{lem-tight}]
If $K_1$ and $K_2$ are not singletons, then the result follows  immediately \changes{for a whole-plane GFF} from Proposition~\ref{prop-two-set-dist}\changes{, and for general $h$ by combining this proposition with} Axiom~\ref{item-metric-f}.  To prove the result with $K_1$ and $K_2$ allowed to be singletons, it suffices to show that, for each fixed \changes{$z \in U$ and} $p \in (0,1)$, we can choose $N \in \BB{N}$ and $C>0$ large enough so that, for each $\BB{r}>0$,  it is the case with probability at least $1 - p$ that \eqb
C^{-1} \leq D_h(\BB{r} z, \partial B_{2^{-N} \BB{r}} (\BB{r} z)) \leq C. \label{eqn-point-to-point-from-center}
\eqe
By Axiom~\ref{item-metric-f}, we may assume without loss of generality that $h$ is a whole-plane GFF normalized so that $h_1(0) = 0$.

To prove the bounds~\eqref{eqn-point-to-point-from-center}, we first choose some value of $\zeta \in (0,Q/3)$ which will remain fixed. By Proposition~\ref{prop-two-set-dist}, it holds with superexponentially high probability as $n \rta \infty$, at a rate uniform in $\BB{r} > 0$, that 
\eqb 
2^{-\xi \zeta n} \leq
 \frk c_{2^{-n} \BB{r}}^{-1} e^{-\xi h_{2^{-n}\BB{r}}(\BB{r} z)} D_h(\text{across $\BB{A}_{2^{-n} \BB{r}, 2^{-n+2} \BB{r}}(\BB{r} z)$})
 \leq
 2^{\xi \zeta n}
\label{eqn-point-to-point-proto-a}
\eqe
and
\eqb 
2^{-\xi \zeta n} \leq
 \frk c_{2^{-n} \BB{r}}^{-1} e^{-\xi h_{2^{-n}\BB{r}}(\BB{r} z)} D_h(\text{around $\BB{A}_{2^{-n} \BB{r}, 2^{-n+1} \BB{r}}(\BB{r} z)$})
 \leq
 2^{\xi \zeta n}.
 \label{eqn-point-to-point-proto-b}
 \eqe
Now, from~\eqref{eqn-metric-scaling} we deduce that, for all $n$ larger than some fixed positive integer independent of $\BB{r}$,
\changes{ \eqb
2^{-\xi (Q + \zeta) n} \leq \frac{\frk c_{2^{-n} \BB{r}}}{\frk c_{\BB{r}}} \leq 2^{-\xi (Q - \zeta) n}.
 \label{eqn-point-to-point-markov}
 \eqe }
Moreover, \changes{we can apply the Gaussian tail bound
\[
\BB P\left[\sup_{r\in  [\ep^{1+\nu} \BB r , \ep  \BB r]} |h_r(w) - h_{\ep\BB r}(w)| \leq s \log\ep^{-1} \right] \geq 1 - O_\ep\left( \ep^{ s^2/(2\nu)} \right) \qquad \text{for all $s,\nu>0$}
\]
stated in~\cite[Equation (3.8)]{lqg-metric-estimates} to get}
\eqb
\BB P\left[  |  h_{2^{-n}\BB{r}}(\BB{r} z) - h_{\BB{r}}(0)| \leq \zeta n \log 2 \right] \geq 1- O_n\left( 2^{ -\zeta^2 n /2 } \right),
\label{eqn-point-to-point-gaussian}
\eqe
where this lower bound does not depend on $\BB{r}$. Combining~\eqref{eqn-point-to-point-markov} and~\eqref{eqn-point-to-point-gaussian} with the bounds~\eqref{eqn-point-to-point-proto-a} and~\eqref{eqn-point-to-point-proto-b}, we deduce that it holds with exponentially high probability as $n \rta \infty$, at a rate uniform in $\BB{r} > 0$, that 
\eqb
2^{-\xi (Q + 3 \zeta) n} \leq 
 \frk c_{\BB{r}}^{-1} e^{-\xi h_{\BB{r}}(0)} D_h(\text{across $\BB{A}_{2^{-n} \BB{r}, 2^{-n+2} \BB{r}}(\BB{r} z)$})  \leq 2^{-\xi (Q - 3 \zeta) n}
 \label{eqn-point-to-point-a}
\eqe
and
\eqb
2^{-\xi (Q + 3 \zeta) n} \leq 
 \frk c_{\BB{r}}^{-1} e^{-\xi h_{\BB{r}}(0)} D_h(\text{around $\BB{A}_{2^{-n} \BB{r}, 2^{-n+1} \BB{r}}(\BB{r} z)$}) \leq 2^{-\xi (Q - 3 \zeta) n}.
 \label{eqn-point-to-point-b}
\eqe
It follows that, for each fixed $p \in (0,1)$, we can choose $N \in \BB{N}$ large enough so that, for each $\BB{r}>0$,  it is the case with probability at least $1-p$ that the distance bounds~\eqref{eqn-point-to-point-a} and~\eqref{eqn-point-to-point-b} hold for all $n \geq N$.

We can now prove the desired bounds~\eqref{eqn-point-to-point-from-center} by iterating the bounds~\eqref{eqn-point-to-point-a} and~\eqref{eqn-point-to-point-b} across concentric annuli:
\begin{itemize}
    \item Observe that every path across the annulus  $\BB{A}_{2^{-n} \BB{r}, 2^{-n+2} \BB{r}}(\BB{r} z)$ intersects both every path around  $\BB{A}_{2^{-n+1} \BB{r}, 2^{-n+2} \BB{r}}(\BB{r} z)$ and every path around  $\BB{A}_{2^{-n} \BB{r}, 2^{-n+1} \BB{r}}(\BB{r} z)$.  Thus, we can string together paths across the annuli $\BB{A}_{2^{-n} \BB{r}, 2^{-n+2} \BB{r}}(\BB{r} z)$ and paths around the annuli $\BB{A}_{2^{-n} \BB{r}, 2^{-n+1} \BB{r}}(\BB{r} z)$ to obtain a path from  $\partial B_{2^{-N} \BB{r}} (\BB{r} z)$ to an arbitrarily small neighborhood of $\BB{r} z$.  Therefore, we can apply the upper bounds in~\eqref{eqn-point-to-point-a} and~\eqref{eqn-point-to-point-b} on distances around and across annuli to bound from above the distance from  $\partial B_{2^{-N} \BB{r}} (\BB{r} z)$ to an arbitrarily small neighborhood of $\BB{r} z$.  By Lemma~\ref{lem-closure}, this yields an upper bound for the distance from $\partial B_{2^{-N} \BB{r}} (\BB{r} z)$  to $\BB{r} z$.
\item
Since every path from $\BB{r} z$ to $\partial B_{2^{-N} \BB{r}} (\BB{r} z)$ must cross the annuli $\BB{A}_{2^{-n} \BB{r}, 2^{-n+2} \BB{r}}(\BB{r} z)$ for all $n \geq N+2$, we can sum the lower bound in~\eqref{eqn-point-to-point-a} over all $n \geq N+2$ to obtain the desired lower bound in~\eqref{eqn-point-to-point-from-center}.
\end{itemize}
\end{proof}

\subsection{Bounding distances across scales: square annuli and hashes}
\label{sec-annuli-hashes}

Proposition~\ref{prop-two-set-dist} gives an a.s.\ distance estimate for distances near a fixed point at varying scales.  However, to prove properties of $D_h$, we need estimates on distances throughout a fixed bounded \emph{region} of the plane.  To obtain the estimates we need, we partition the region of interest into \emph{dyadic squares}.  

\begin{defn}
\label{defn-dyadic}
For $n \in \BB{Z}$ and a Borel set $X$, let $\mcl S^n(X)$ denote the set of squares of the form $[a,a+2^{-n}] \times [b, b + 2^{-n}]$ intersecting $X$ with $a,b \in 2^{-n} \BB{Z}$. When $X$ is a singleton $\{x\}$, we will denote $\mcl S^n(\{x\})$ simply by $\mcl S^n(x)$.

We write $\mcl S^n = \mcl S^n(\BB{C})$, and we call $S$ a \emph{dyadic square} if $S \in \mcl S^n$ for some integer $n$. We can decompose  every dyadic square $S$ into four squares with side length $|S|/2$; we call these four squares the \emph{dyadic children} of $S$, and $S$ their \emph{dyadic parent}.  
\end{defn}

For a fixed bounded open set in the plane, 
we will prove two distance bounds that hold a.s.\ for every sufficiently small dyadic square $S$ intersecting the fixed set: a lower bound on the distance across a square annulus that surrounds $S$, and a\changes{n} upper bound on the diameter of a network of four paths in $S$ that criss-cross $S$.  The reason for considering these two sets in particular is that they satisfy several geometric properties that allow us to string together upper and lower distance bounds across scales.  

We begin by  defining both of these sets and stating their useful geometric properties.  (We will not formally prove these properties; they are all easy to see by inspection of Figure~\ref{fig-sq-ann-hash}.)

\begin{lem}[The square annulus associated to a dyadic square]
\label{lem-defn-sq-annulus}
For a square $S$, let \changes{$v_S$ denote its center, and let} $A_S$ denote the \changes{(open)} square annulus centered at $v_S$ formed by squares with side lengths $2|S|$ and $3|S|$.  These annuli satisfy the following properties:
\begin{enumerate}[(a)]
    \item 
    If $S'$ is a dyadic child of $S$, then $A_{S'}$ is nested in and disjoint from $A_S$. 
     \item If $S,S' \in \mcl S^n(S)$ share an edge, then every path around $A_S$ intersects every path around $A_{S'}$.
    \item
    For all $z,w \in \BB{C}$ with $|z-w| \geq 2^{-n}$, we can choose $S \in \mcl S^{n+2}$ with $z \in S$ and $w$ outside $A_S$.
\end{enumerate}
\end{lem}

See the left side of Figure~\ref{fig-sq-ann-hash} for an illustration of $A_S$. The first property of the lemma implies that, if we can bound the distance across each square annulus from below, we can string these lower bounds together across scales---as we do in the proof of Proposition~\ref{prop-singular} below. 

\begin{defn}[A hash associated to a dyadic square]
\label{defn-hash}
Consider the four rectangles $R_1,\ldots,R_4$ formed by unions of pairs of dyadic children of $S$.  We define a \emph{hash}  associated to $S$ as a union of four paths $P_1 \cup \cdots \cup P_4$, where $P_j$ is a path \changes{in $R_j$} between the two short sides of $R_j$. 
\end{defn}

We call the union of four paths in Definition~\ref{defn-hash} a hash because they form the shape of a hash symbol (\#).  See the right side of Figure~\ref{fig-sq-ann-hash} for an illustration.  The reason for considering this particular configuration of paths is the following lemma, which allows us to string upper bounds for the diameter of hashes together across scales.

\begin{lem}
\label{lem-ann-hash-geometry}
If $S'$ is a dyadic child of $S$, then every hash associated to $S$ intersects both every hash associated to $S'$ and every path around $A_{S'}$.
\end{lem}

\begin{figure}[ht!]
\begin{center}
\begin{tabular}{ccc} 
\includegraphics[width=0.4\textwidth]{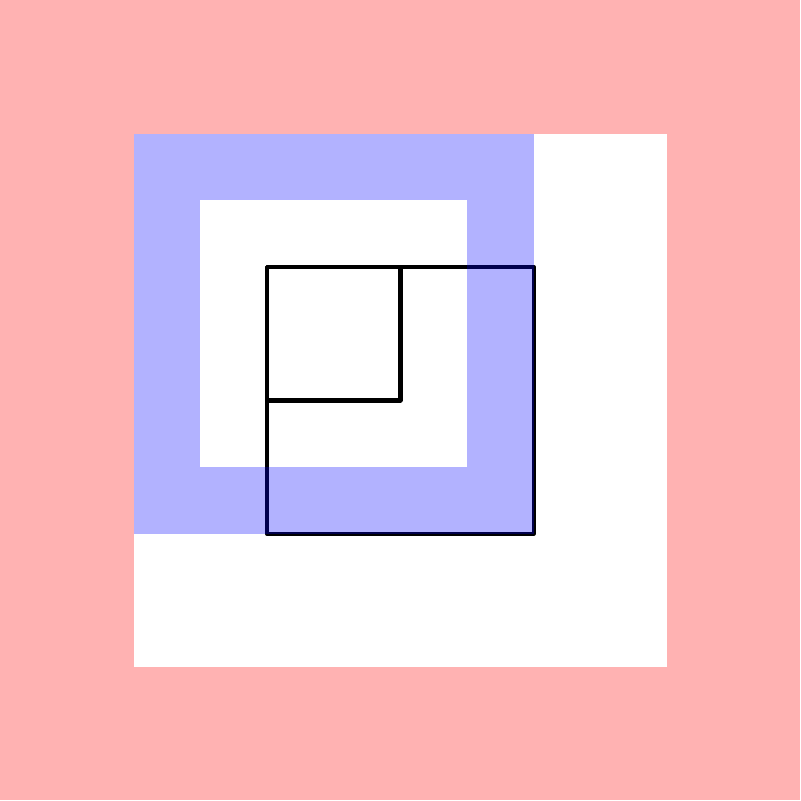}
& \qquad  \qquad &
\includegraphics[width=0.4\textwidth]{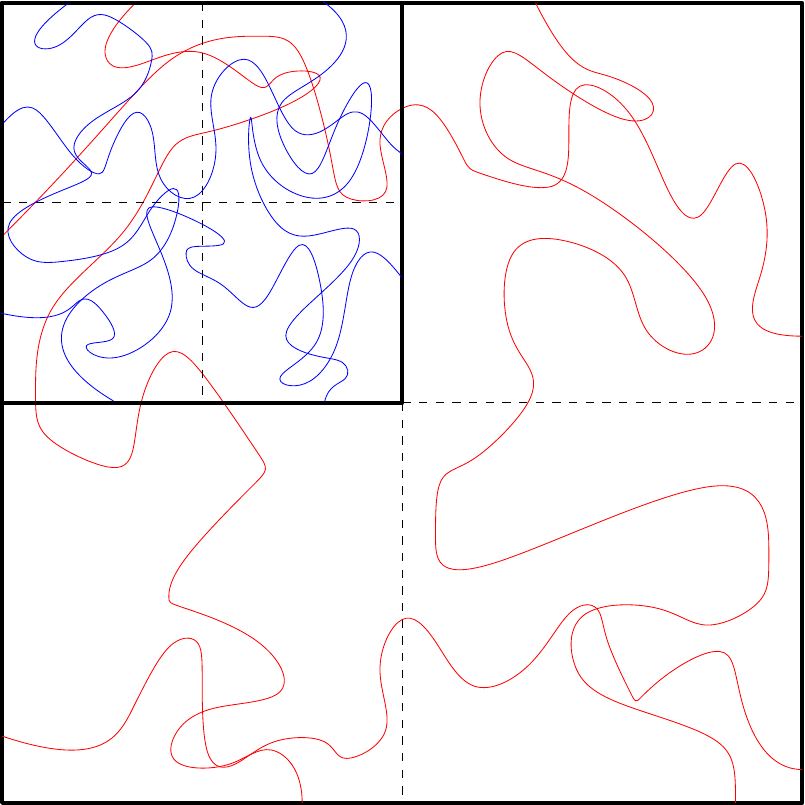}
\end{tabular}
\caption{An illustration of the sets defined in Lemma~\ref{lem-defn-sq-annulus} and~\ref{defn-hash}.  In each figure, we consider a  square $S$ along with one of its dyadic children $S'$ (both black and in bold).  In the figure on the \textbf{left}, we consider the corresponding square annuli $A_S$ and $A_{S'}$ (colored in red and blue, respectively); observe that the two annuli are nested and do not intersect.  In the figure on the \textbf{right}, we consider a hash associated to $S$ (red) and a hash associated to $S'$ (blue); observe that the two hashes necessarily intersect.
} \label{fig-sq-ann-hash}
\end{center}
\end{figure}
 
\begin{lem}
\label{lem-ann-hash-dist} Let $h$ be a whole-plane GFF normalized so that $h_1(0) = 0$.
Let $V \subset \BB{C}$ be a bounded open set, and let $\zeta > 0$.  For each $\BB{r} > 0$, the following bounds hold for each $S \in \mcl S^n(V)$ with superexponentially high probability as $n \rta \infty$, at a rate which is uniform in $\BB{r}$.
\begin{enumerate}
    \item \textit{(Distance lower bound.)}
\eqb
D_h(\text{across $A_{\BB{r} S}$}) \geq 2^{-\xi (Q+\zeta) n} \frk c_{\BB{r}}  e^{\xi h_{2^{-n}\BB{r}}(v_{\BB{r} S})}.
\label{eqn-sq-ann-bound-across}
\eqe
     \item \textit{(Distance upper bounds.)} 
     \eqb
D_h(\text{around $A_{\BB{r} S}$}) \leq 2^{-\xi (Q-\zeta) n}  \frk c_{\BB{r}} e^{\xi h_{2^{-n} \BB{r}}(v_{\BB{r} S})},
\label{eqn-sq-ann-bound-around}
\eqe
     and we can choose a hash $\#_{\BB{r} S}$ associated to ${\BB{r} S}$ with
\eqb
\diam(\#_{\BB{r} S};D_h) \leq 2^{-\xi (Q-\zeta) n} \frk c_{\BB{r}} e^{\xi h_{2^{-n} \BB{r}}(v_{\BB{r} S})}.
\label{eqn-sq-hash-bound}
\eqe
\end{enumerate}
\end{lem}

We note that, except for the proof of Proposition~\ref{prop-holder-r} below, all our proofs that use Lemma~\ref{lem-ann-hash-dist} apply the lemma in the special case $\BB{r} = 1$.

\begin{proof}
By Proposition~\ref{prop-two-set-dist} (applied with $A = 2^{\xi \zeta n/2}$), Axiom~\ref{item-metric-translate}, and a union bound over all $S \in \mcl S^n(V)$,
it is the case with superexponentially high probability as $n \rta \infty$  (at a rate that is uniform in $\BB{r}>0$) that, for each $S \in \mcl S^n(V)$,
\eqb
D_h(\text{across $A_{\BB{r} S}$}) \geq 2^{-\xi \zeta n/2} \frk c_{2^{-n} \BB{r}}  e^{\xi h_{2^{-n} \BB{r}}(v_{\BB{r} S})}
\label{eqn-sq-ann-bound-1}
\eqe
and
\eqb
D_h(\text{around $A_{\BB{r} S}$})  \leq 2^{\xi \zeta n/2} \frk c_{2^{-n} \BB{r}}  e^{\xi h_{2^{-n} \BB{r}}(v_{\BB{r} S})},
\label{eqn-sq-ann-bound-1a}
\eqe
and there is a hash $\#_{\BB{r} S}$ associated to ${\BB{r} S}$ with 
\eqb
\diam(\#_{\BB{r} S};D_h) \leq 2^{\xi \zeta n/2} \frk c_{2^{-n} \BB{r}}  e^{\xi h_{2^{-n} \BB{r}}(v_{\BB{r} S})}.
\label{eqn-hash-bound-1}
\eqe \changes{
 Moreover, from~\eqref{eqn-metric-scaling} we deduce that, for all $n$ larger than some fixed positive integer independent of $\BB{r}$, 
 \eqb
2^{-\xi (Q + \zeta/2) n} \leq \frac{\frk c_{2^{-n} \BB{r}}}{\frk c_{\BB{r}}} \leq 2^{-\xi (Q - \zeta/2) n}.
 \label{eqn-ann-hash-markov-bound}
 \eqe}
 The result follows from applying~\eqref{eqn-ann-hash-markov-bound} to the bounds~\eqref{eqn-sq-ann-bound-1}-~\eqref{eqn-hash-bound-1}.
\end{proof}

\subsection{H\"older continuity, singular points, completeness, and geodesics}
\label{sec-holder-singular}

It was proved in~\cite[Proposition 5.20]{dg-supercritical-lfpp} that  subsequential limits of the rescaled $\ep$-LFPP metrics~\eqref{eqn-defn-rescaled-LFPP} in the supercritical phase are a.s.\ reverse  H\"older continuous with respect to the Euclidean metric. We now prove Proposition~\ref{prop-holder}, which generalizes this property to weak LQG metrics.  Specifically, we will prove the following more quantitative version of this property which is required to be uniform across scales.  \changes{(We remind the reader that, throughout this and the next section, $h$ is taken to be a whole-plane GFF plus a continuous function, unless specified otherwise.)}

\begin{prop}
\label{prop-holder-r}
Fix $\chi > \xi(Q+2)$ and a compact set $K \subset \BB{C}$. For each $\BB{r}>0$, it holds with polynomially high probability as $\ep \rta 0$, at a rate which is uniform in $\BB{r}$, that 
\eqb
\frk c_{\BB{r}}^{-1} e^{-\xi h_{\BB{r}}(0)} D_h(z,w) \geq \left( |z-w| / \BB{r} \right)^{\chi} \qquad \text{$\forall z,w \in \BB{r} K$ with $|z-w| \leq \ep \BB{r}$}.
\label{eqn-holder-r}
\eqe
\end{prop}

\begin{proof}[Proof of Proposition~\ref{prop-holder-r}]
We may assume without loss of generality that $h$ is a whole-plane GFF normalized so that $h_1(0) = 0$.
Fix $\zeta>0$ and a compact set $K \subset \BB{C}$.  First, by the Gaussian tail bound and a union bound, we have
    \eqb
|h_{2^{-n} \BB{r}}(v_{\BB{r} S}) - h_{\BB{r}}(v_{\BB{r} S})| \leq (2 + \zeta) n \log 2 \qquad \forall S \in \mcl S^n(K) \label{eqn-holder-2}  \eqe
    with probability  $1 - O_n(2^{-[(2+\zeta)^2/2 - 2]n})$, with the rate uniform in $\BB{r}$. Combining the result with Lemma~\ref{lem-ann-hash-dist}, we deduce that the bounds~\eqref{eqn-sq-ann-bound-across} and~\eqref{eqn-holder-2} hold with exponentially high probability as $n \rta \infty$, at a rate which is uniform in $\BB{r}$.  This means that for each $N \in \BB{N}$, we can define an event $\mcl E_N$ that holds with exponentially high probability as $N \rta \infty$ (at a rate which is uniform in $\BB{r}$), such that  on the event $\mcl E_N$, the bounds~\eqref{eqn-sq-ann-bound-across} and~\eqref{eqn-holder-2} hold for all $n \geq N$.
    
    We now fix $N \in \BB{N}$ and work on the event $\mcl E_N$.  Suppose that $z,w \in \BB{r} K$  with $|z-w| \leq 2^{-N} \BB{r}$. Let $n >N$ be the integer satisfying  $2^{-n} < |z-w|/\BB{r} \leq 2^{-n+1}$.  By Lemma~\ref{lem-defn-sq-annulus}(c), we can choose $S \in \mcl S^{n+2}(K)$ containing $z$ such that the point $w$ lies outside $A_S$, which means that $D_h(z,w)$ is bounded from below by the distance across $A_S$. Since we are working on the event $\mcl E_N$, we can combine~\eqref{eqn-sq-ann-bound-across} and~\eqref{eqn-holder-2} to bound $D_h(\text{across $A_S$})$ from below by $2^{-\xi(Q+2+2\zeta)n} \frk c_{\BB{r}} e^{\xi h_{\BB{r}}(0)}$, simultaneously for all possible choices of $S \in \mcl S^{n+2}(K)$.  We deduce that, on the event $\mcl E_N$,
    \eqb
\frk c_{\BB{r}}^{-1} e^{-\xi h_{\BB{r}}(0)} D_h(z,w) \geq 2^{-n\xi(Q+2 + 2\zeta)} \qquad \text{$\forall z,w \in \BB{r} K$ with $|z-w| \leq 2^{-N} \BB{r}$}.
\label{eqn-holder-almost-done}
\eqe
    Since $2^{-n} < |z-w| / \BB{r}$, the upper bound~\eqref{eqn-holder-almost-done} immediately yields~\eqref{eqn-holder-r}.
\end{proof}

\begin{proof}[Proof of Proposition~\ref{prop-holder}]
The proposition follows from Proposition~\ref{prop-holder-r}.
\end{proof}

As we noted in Section~\ref{sec-intro-properties}, when $\changes{\cc \in (-\infty,1)}$, the inverse map of the map in Proposition~\ref{prop-holder} is a.s.\ locally H\"older continuous with any exponent smaller than $\xi(Q-2)$~\cite[Theorem 1.7]{lqg-metric-estimates}.  When $\cc \in (1,25)$, however, this inverse map is not continuous, since a.s.\ the metric $D_h$ has \emph{singular points}, as defined in Definition~\ref{defn-singular}. 
We now prove Proposition~\ref{prop-singular}, which characterizes the set of singular points of the metric $D_h$ when $\cc \in (1,25)$.
The proof of Proposition~\ref{prop-singular} uses the following continuity estimate for the circle average process.

\begin{lem}
Let $h$ be a whole-plane GFF normalized so that $h_1(0) = 0$.
For each $r>0$, let $n = n(r) \in \BB{Z}$ be such that $2^{-n-1} \leq r < 2^{-n}$.  Almost surely, for every $z \in \BB{C}$ the quantity
\[
\frac{h_r(z)}{\log{r^{-1}}} - \frac{h_{2^{-n}}(v_S)}{n \log 2}
\]
tends to $0$ as $r \rta 0$ uniformly in the choice of $S \in \mcl S^{n}(z)$.
\label{lem-limsup-Q}
\end{lem}

\begin{proof}
Observe that, if $S \in \mcl S^{n}(z)$ for some $z \in \BB{C}$, then $|z-v_S| < 2^{-n} \leq 2r$.  Therefore, we can apply the circle average continuity estimate~\cite[Lemma 3.15]{ghm-kpz} to compare $h_r(z)$ and $h_r(v_S)$. 
We deduce from this estimate that a.s., for every $z \in \BB{C}$ the quantity
\eqb
\frac{h_r(z)}{\log{r^{-1}}} - \frac{h_{r}(v_S)}{\log r^{-1}}
\label{eqn-two-fractions}
\eqe
tends to $0$ as $r \rta 0$ uniformly in the choice of $S \in \mcl S^{n}(z)$.
This convergence  still holds if we replace the denominator of the second term of~\eqref{eqn-two-fractions} by $n \log 2$. 
Finally, we can replace $h_{r}(v_S)$ in the second numerator by $h_{2^{-n}}(v_S)$, since by an elementary Brownian motion estimate and a union bound, a.s.\ 
\[
\max_{S \in \mcl S^n} \max_{|S|/2 \leq r < |S|} (h_{r}(v_S) - h_{2^{-n}}(v_S)) \leq n^{51/100} \log 2
\]
for all $n$ sufficiently large.  
\end{proof}

In the proof of both parts of Proposition~\ref{prop-singular}, we 
 assume without loss of generality that $h$ is a whole-plane GFF normalized so that $h_1(0) = 0$.

\begin{proof}[Proof of Proposition~\ref{prop-singular}(a)]
Let $\zeta > 0$. By Lemmas~\ref{lem-ann-hash-dist} and~\ref{lem-limsup-Q} and the Borel-Cantelli lemma, it is a.s.\ the case that, for every $z \in \BB{C}$ satisfying~\eqref{eqn-limsup-Q}, we can choose $\zeta = \zeta(z) > 0$ and $N = N(z) \in \BB{Z}$ for which
\eqb
D_h(\text{across $A_S$}) \geq 2^{\xi \zeta n}  \qquad \forall S \in \mcl S^n(z), \qquad \forall n \geq N.
\label{eqn-across-singular}
\eqe
By Lemma~\ref{lem-defn-sq-annulus}(a), for any two distinct points $z,w \in \BB{C}$, we can choose an integer $M \in \BB{Z}$ and a square $S_n \in \mcl S^n(z)$ for each $n \geq M$, such that the annuli $\{A_{S_n}\}_{n=M}^{\infty}$ are pairwise disjoint, and every path from $z$ to $w$ must cross every single annulus.  Hence, from the fact that summing the right-hand side of~\eqref{eqn-across-singular} over $n \in \BB{N}$ diverges, we may deduce that a.s.\ every $z \in \BB{C}$ satisfying~\eqref{eqn-limsup-Q} satisfies $D_h(z,w) = \infty$ for every $w \neq z$ in $\BB{C}$.  
\end{proof}

\begin{proof}[Proof of Proposition~\ref{prop-singular}(b)]
By Lemmas~\ref{lem-ann-hash-dist} and~\ref{lem-limsup-Q} and the Borel-Cantelli lemma, it is a.s.\ the case that, for every $z \in \BB{C}$ for which~\eqref{eqn-limsup-Q} is less than $Q$, we can choose $\zeta = \zeta(z) > 0$ and $N = N(z) \in \BB{Z}$ such that, for every $n \geq N$ and $S \in \mcl S^n(z)$,  we can choose a hash $\#_S$ associated to $S$ with
\[
\diam(\#_S;D_h) \leq 2^{-\xi \zeta n}. \]
Since a hash associated to a square must intersect a hash associated to each of its dyadic children, this implies that the union of hashes \[ H_S:= \bigcup_{m \in \BB{N}} \bigcup_{S \supset S' \in \mcl S^{m}(z)} \#_{S'} \] satisfies
\[
\diam(H_S;D_h) \leq 
2 \sum_{m = n}^{\infty} 2^{-\xi \zeta m} < \infty\]
and therefore, by Lemma~\ref{lem-closure},
\eqb
\diam(\ol H_S;D_h) \leq 
2 \sum_{m = n}^{\infty} 2^{-\xi \zeta m} < \infty.
\label{eqn-singular-converse-olHs}
\eqe
Since $z \in \ol H_S$ by construction of $H_S$,~\eqref{eqn-singular-converse-olHs} implies that $z$ is not a singular point. 
\begin{comment}
We now prove the second part of the statement. by Lemma~\ref{lem-ann-hash-dist}, a.s.\ for every $z \in \BB{C}$ for which~\eqref{eqn-limsup-Q} is less than $Q$, we have
\eqb
D_h(\text{around $A_S$}) \leq 2^{-\xi \zeta n} 
\label{eqn-prop-singular-b-around-As}
\eqe
for all large enough $n$ and $S \in \mcl S^n(z)$. 
Also, recall from Lemma~\ref{lem-ann-hash-geometry} that, if $S'$ is a dyadic child of $S$, then the set $H_S$ intersects every path around $A_{S'}$. Hence, from~\eqref{eqn-singular-converse-olHs} and~\eqref{eqn-prop-singular-b-around-As}, we deduce that a.s.\  for every choice of $z$  and for all large enough $n$ and $S \in \mcl S^n(z)$, there is a path $L$ around $A_S$ whose $D_h$-distance to $z$ is less than $\ep$.
\end{comment}
\end{proof}

We now prove several properties of $D_h$ on the set $\BB{C} \backslash \{\text{singular points}\}$.  Specifically, we prove Proposition~\ref{prop-complete}, which states that a.s.\ on the set $\BB{C} \backslash \{\text{singular points}\}$, the metric $D_h$ is complete and finite-valued, and every pair of points can be joined by a geodesic. 
To prove Proposition~\ref{prop-complete}, we first show that $D_h$ is a finite-valued metric on $\BB{C} \backslash \{\text{singular points}\}$.

\begin{lem}
\label{lem-finite}
A.s.\ for every $z,w,z',w' \in \BB{C}$ with $D_h(z,z') < \infty$ and $D_h(w,w') < \infty$, we have $D(z,w) < \infty$. In particular, if $z, w \in \BB{C}$ are such that $D_h(z,w) = \infty$, then at least one of the points $z,w$ must be singular.
\end{lem}

\begin{proof}
By Axiom~\ref{item-metric-f}, we may assume without loss of generality that $h$ is a whole-plane GFF normalized so that $h_1(0) = 0$. 
Also, it suffices to consider $z,z',w,w'$ in a fixed  connected open set $V \subset \BB{C}$.  
By Lemma~\ref{lem-ann-hash-dist} and the Borel-Cantelli lemma, a.s.\ there exists $N \in \BB{N}$ such that the $D_h$-distance around $A_S$ is finite for all $S \in \mcl S^n(V)$ with $n \geq N$.  By Lemma~\ref{lem-defn-sq-annulus}(b), this implies that a.s.\ for all $n \geq N$, we can choose paths around the annuli $\{A_S : S \in \mcl S^n(V)\}$ with finite $D_h$-length, such that the $D_h$-distance between any two of these paths is finite. Moreover, a.s.\ for every $z,w,z',w' \in \BB{C}$ with $D_h(z,z') < \infty$ and $D_h(w,w') < \infty$, if we choose $n \geq N$ sufficiently large, then there is an annulus $A_{S_z}$ in  the collection $\{A_S : S \in \mcl S^n(V)\}$ that separates $z$ and $z'$, and an annulus $A_{S_w}$ in the collection that separates $w$ and $w'$. This implies that the $D_h$-distance from $z$ (resp. $w$) to the path around the annulus $A_{S_z}$ (resp. $A_{S_w}$) is finite.
\end{proof}

To prove that $D_h$ is a.s.\ complete and geodesic  on $\BB{C} \backslash \{\text{singular points}\}$, we apply the following elementary lemma:

\begin{lem}
Let $D$ be a lower semicontinuous metric on \changes{$\BB{C}$} such that the Euclidean metric is continuous w.r.t. $D$. Then every $D$-Cauchy sequence converges in $(\BB{C},D)$. Moreover, if \changes{$(\BB{C},D)$} is a length space and every $D$-bounded set is Euclidean bounded, then every pair of points $z,w \in \changes{\BB{C}}$ with $D(z,w)< \infty$ can be joined by a geodesic.
\label{lem-complete-geodesic}
\end{lem}

\begin{proof}
If $(z_n)_{n \in \BB{N}}$ is $D$-Cauchy, then $(z_n)_{n \in \BB{N}}$ converges to some $z$ in the Euclidean topology. Therefore, for each fixed $\ep>0$, we have $D(z,z_n) \leq \liminf_{m \rta \infty} D(z_n,z_m) < \ep$ for all large enough $n$.  In other words, $(z_n)_{n \in \BB{N}}$ is convergent in $(\BB{C},D)$.

Now, suppose that every $D$-bounded set is Euclidean bounded, and let $z,w \in \changes{\BB{C}}$ be such that $\ell: = D(z,w)$ is finite. The boundedness assumption implies that the Euclidean closure $K$ of the set $\{x \in U : D(z,x) \leq \ell + 1\}$ is compact.  Therefore, there exists a continuous nondecreasing function $\delta: [0,\infty) \rta [0,\infty)$ with $\delta(0) = 0$ such that, for every $x,y \in K$, we have $|x-y| < \delta(D(x,y))$.  

For each $n \in \BB{N}$, let \changes{$P_n:[0,\ell_n] \rta \BB{C}$} be a path, parameterized by $D$-length, from $z$ to $w$ with \changes{length}  $\ell_n \leq \ell + 1/n$. Note that, for each $n$, \changes{we have $P_n([0,\ell_n]) \subset K$}, so that for all $s,t$ in the domain of $P_n$,
\eqb
|P_n(s) - P_n(t)| < \delta(D(P_n(s),P_n(t)) \leq \delta(\len(P_n[s,t];D)) = \delta(s-t).
\label{eqn-Pn-cont}
\eqe
\changes{By the  Arzel\`{a}–Ascoli theorem, we can choose a subsequence of $P_n|_{[0,\ell]}$ that converges uniformly to a continuous path $P: [0,\ell] \rta  K$.   By~\eqref{eqn-Pn-cont},  \[|P(\ell) - w| \leq \limsup_{n \rta \infty}
|P_n(\ell) - P_n(\ell_n)| < \limsup_{n \rta \infty} \delta(\ell-\ell_n) = 0,\]
i.e., $P$ is a path from $z$ to $w$.  Since, for each $
0 = t_0 < t_1 < \cdots < t_k = \ell$,
\begin{align*}
\sum_{j=0}^{k-1} D(P(t_j,t_{j+1}))
&\leq \liminf_{n \rta \infty} \sum_{j=0}^{k-1} D(P_n(t_j,t_{j+1})) \\
&\leq \liminf_{n \rta \infty} \len(P_n|_{[0,\ell]};D)\\
&\leq \ell,
\end{align*}
the path $P$ is a geodesic from $z$ to $w$.}
\end{proof}

To apply Lemma~\ref{lem-complete-geodesic} to prove Proposition~\ref{prop-complete}, we check that a.s.\ every $D_h$-bounded set is Euclidean bounded.

\begin{lem}
\label{lem-Dh-Eucl-bounded}
Almost surely, every $D_h$-bounded set is Euclidean bounded. \changes{Specifically, a.s.\ for every $r>0$ and $C>0$, there exists $R>0$ such that $D_h(\text{across $\BB{A}_{r,R}(0)$}) > C$.}
\end{lem}

\begin{proof}
Let $\zeta > 0$.
By Proposition~\ref{prop-two-set-dist} (applied with $A = 2^{\xi \zeta n}$), it is the case with superexponentially high probability as $n \rta \infty$ that 
\[
D_h(\text{across $\BB{A}_{2^n,2^{n+1}}(0)$}) \geq 2^{-\xi \zeta n} \frk c_{2^n} e^{\xi h_{2^n}(0)}.
\]
By~\eqref{eqn-metric-scaling}, $\frk c_{2^n} \geq 2^{\xi(Q-\zeta)n}$ for all sufficiently large $n$.  Also, since $h_{2^n}(0) - h_1(0)$ is a centered Gaussian with variance $n\log{2}$, we have $\BB{P}(h_{2^n}(0) - h_1(0) > \zeta n \log {2}) \leq 
2^{-\zeta^2 n / 2}$.  Hence, by the Borel-Cantelli lemma, \changes{a.s.}
\[
D_h(\text{across $\BB{A}_{2^n,2^{n+1}}(0)$}) \geq 2^{ \xi (Q - 3 \zeta) n}  e^{\xi h_1(0)}.
\]
for all $n \in \BB{N}$ sufficiently large.  This proves the second assertion of the lemma.
This means that a.s.\ a set that is $D_h$-bounded cannot cross the annuli $\BB{A}_{2^n,2^{n+1}}(0)$ for infinitely many values of $n \in \BB{N}$\changes{, hence is bounded in the Euclidean metric.}
\end{proof}

\begin{proof}[Proof of Proposition~\ref{prop-complete}]
The fact that $D_h$ is finite-valued follows from Lemma~\ref{lem-finite}.
The completeness property follows from Proposition~\ref{prop-holder} and Lemma~\ref{lem-complete-geodesic}, and the geodesic property follows from Lemmas~\ref{lem-complete-geodesic} and~\ref{lem-Dh-Eucl-bounded}.
\end{proof}

We conclude our study of the set $\BB{C} \backslash \{\text{singular points}\}$ by proving Proposition~\ref{prop-rational-dense}, which asserts that a.s.\ for every $z \in \BB{C} \backslash \{\text{singular points}\}$, there is a sequence of rational points converging to $z$ in the metric $D_h$, and a sequence of loops surrounding $z$ and shrinking to $z$ in the metric $D_h$. Note that the first part of the proposition implies that a.s.\ the $D_h$-closure of  $\BB{Q}^2$ is exactly the set $\BB{C} \backslash \{\text{singular points}\}$, since a.s.\ the $D_h$-closure of $\BB{Q}^2$  does not contain any singular points.

\begin{proof}[Proof of Proposition~\ref{prop-rational-dense}]
To prove the proposition,  we will fix a compact set $K \subset \BB{C}$ and prove that a.s.\ every nonsingular point $z \in K$ is a $D_h$-limit of points in $\BB{Q}^2$ and is surrounded by a loop with the desired properties.

For every nonsingular point $z \in K$, we can choose a sequence $w_n \in 2^{-2n} \BB{Z}^2 \cap B_1(K)$ with $|z - w_n| < 2^{-2n}$ for each $n \in \BB{N}$.  We will apply Lemma~\ref{lem-annulus-iterate} to define a sequence of (random)
rational radii $r_n \in [2^{-2n},2^{-n}]$ such that the annuli $\BB{A}_{r_n,2r_n}(w_n)$ satisfy a certain distance estimate.  We will use this estimate to choose a rational point in the annulus that is $D_h$-close to $z$, and a loop around the annulus (which surrounds $z$) with small Euclidean and $D_h$-length that is close to $z$ in both metrics.

\begin{comment}
combine this estimate with the fact that 
the annuli $\BB{A}_{r_n,2r_n}(w_n)$ surround the point $z$ and shrink to $z$ as $n \rta \infty$.

By Lemma
\item
We use Lemma~\ref{lem-annulus-iterate} to show that a.s.\ for all large enough $n$, it is the case that for every $z \in K$, there is a loop around the annulus $\BB{A}_{r_n,2r_n}(w_n)$ whose $D_h$-length and whose $D_h$-distance to the rational point  $w_n + 3r_n/2 \in \BB{Q}^2 \cap \BB{A}_{r_n,2r_n}(w_n)$ are bounded from above by a constant times the distance across the annulus $\BB{A}_{r,2r_n}(w_n)$.
\item
This will imply the proposition, since a.s.\ for every nonsingular point $z \in K$, both the $D_h$-distance across the annulus $\BB{A}_{r,2r_n}(w_n)$ and the $D_h$-distance from $z$ to the outside of the annulus $\BB{A}_{r_n,2r_n}(w_n)$ tend to $0$ as $n \rta \infty$.
\end{itemize}
\end{comment}

Specifically, we apply Lemma~\ref{lem-annulus-iterate} to the events  $F_r(w;C)$, defined for every $w \in \BB{Q}^2$, every rational $r>0$, and every $C>0$, as the event that
\[
D_h(w+3r/2, \partial B_r(w)) + D_h(w+3r/2, \partial B_{2r}(w)) + D_h(\text{around $\BB{A}_{r,2r}(w)$}) < C D_h(\text{across $\BB{A}_{r,2r}(w)$}).
\]
By Lemma~\ref{lem-tight}, the distance across and around a fixed annulus and the distance from a fixed point to a fixed circle are a.s.\ positive and finite.  By Axiom~\ref{item-metric-coord}, this means that for each $p \in (0,1)$ and each fixed $w \in \BB{Q}^2$, we can choose $C>1$ such that $\BB{P}(F_r(w;C)) \geq p$ for every sufficiently small rational $r>0$.  By Axiom~\ref{item-metric-translate}, for this value of $C$ and all rational $r>0$ sufficiently small (in a manner that does not depend on $w$), we have $\BB{P}(F_r(w;C)) \geq p$ for \emph{every} $w \in \BB{Q}^2$.

Therefore, by Lemma~\ref{lem-annulus-iterate}, a union bound, and the Borel-Cantelli lemma, we can choose $C>1$ such that a.s.\ for all $n$ sufficiently large and every $w \in 2^{-2n} \BB{Z}^2 \cap B_1(K)$, the event $F_{\changes{r_n}}(w;C)$ holds for some rational $r_n \in [2^{-2n},2^{-n}]$.  Hence, a.s.\ for every $z \in K$ and all $n$ sufficiently large, 
\eqb
D_h(z,w_n + 3r_n/2) <
D_h(z,\partial B_{2r_n}(w_n)) + C D_h(\text{across $\BB{A}_{\changes{r_n},2r_n}(w_n)$}).
\label{eqn-two-terms-right-side}
\eqe
and there is a loop around $\BB{A}_{r_n,2r_n}(w_n)$ whose $D_h$-length is at most $C D_h(\text{across $\BB{A}_{\changes{r_n},2r_n}(w_n)$})$ and whose $D_h$-distance from $z$ is at most $D_h(z,\partial B_{2r_n}(w_n))$.  Note that this loop surrounds $z$, and that both its Euclidean diameter and its Euclidean distance from $z$ tend to $0$ as $n \rta \infty$.

Thus, to complete the proof, it remains to check that a.s.\  \[  D_h(\text{across $\BB{A}_{\changes{r_n},2r_n}(w_n)$}) \vee D_h(z,\partial B_{2r_n}(w_n)) \rta 0 \qquad \text{as $n \rta \infty$.}\]
The convergence of the first quantity is immediate: a.s.\ if $D_h(\text{across $\BB{A}_{\changes{r_n},2r_n}(w_n)$})$ did not tend to zero as $n \rta \infty$, then the $D_h$-distance from $z$ to any other point would be infinite.  For the convergence of the second quantity, first observe that a.s.\ if $z \in K$ is not a singular point, Proposition~\ref{prop-complete} implies that we can choose a geodesic $P$ from $z$ to a point $x \in \BB{C}$.  For large enough $n$, let $x_n$ be the first point on $\partial B_{2r_n}(w_n)$ that the geodesic $P$ hits, so that
$
D_h(z,\partial B_{2r_n}(w_n)) \leq D_h(z,x_n).
$
Since $P$ is a geodesic, we have $D(z,x_n) + D(x_n,x) = D(z,x)$. Since $D_h$ is lower semicontinuous, we deduce that a.s.\ $D(z,x_n) \rta 0$ as $n \rta \infty$.  Thus, a.s.\ $D_h(z,\partial B_{2r_n}(w_n)) \changes{\rta 0}$
as $n \rta \infty$.
\end{proof}

\subsection{The metric ball for \texorpdfstring{$\cc \in (1,25)$}{matter central charge between 1 and 25}}
\label{sec-compact-ball}

We conclude Section~\ref{sec-properties} by proving Proposition~\ref{prop-compact}, which states several fundamental properties of $D_h$-balls  for $\cc \in (1,25)$ that contrast sharply with the properties of these balls for $\changes{\cc \in (-\infty,1)}$. See Section~\ref{sec-notation} for definitions of the notation used in this section. Before proving Proposition~\ref{prop-compact}, we make a couple of elementary observations about the sets $B_{r}[X;D_h]$ and $B_{r}(X;D_h)$.  First, the set $B_{r}[X;D_h]$ is closed in the Euclidean topology; this follows immediately from the fact that $D_h$ is lower semicontinuous:

\begin{lem}
\label{lem-closed-ball-is-closed}
 Suppose that $D$ is a lower semicontinuous metric on an open set $U \subset \BB{C}$.  For every $X \subset U$ and $r>0$, the set $B_r[X;D]$  is closed in the Euclidean topology on $U$.
\end{lem}
\begin{proof}
If $z_n \rta z$ in $U$ with $D(z_n,X) \leq r$ for each $n$, then lower semicontinuity gives $D(z,X) \leq \liminf_{n \rta \infty} D(z_n,X) \leq r$.
\end{proof}

Second, we observe that a.s.\ the $D_h$-boundary of the set $B_r(X;D_h)$ can be equivalently described as the set of points whose $D_h$-distance to $X$ is exactly $r$. In particular, the $D_h$-closure of $B_r(X;D_h)$ is exactly $B_r[X;D_h]$. This is a general result for lower semicontinuous length metrics:

  \begin{lem}
  Suppose that $D$ is a lower semicontinuous \changes{length} metric on an open set $U \subset \BB{C}$.  Then, for every $r>0$ and \changes{$X \subset U$}, the $D$-boundary of the set $B_r(X;D)$ is equal to the set of points in $U$ whose \changes{$D$-distance} to $X$ is exactly $r$.  In particular, the $D$-closure of $B_r(X;D)$ is exactly the set $B_r[X;D]$.
  \label{lem-boundary-r}
  \end{lem}
  
  \begin{proof}
First, suppose that $z$ is contained in the $D$-boundary of the set $B_r(X;D)$.  Then we can choose sequences $(z_n)_{n \in \BB{N}}$ and $(w_n)_{n \in \BB{N}}$ in $U$ converging to $z$ in the metric $D$, such that $D(z_n,X) \geq r$ and $D(w_n,X) < r$ for each $n$. Hence 
\changes{
\[
D(z,X) \geq D(z_n,X) - D(z,z_n) \geq r - D(z,z_n)
\]
and
\[
D(z,X) \leq D(w_n,X) + D(z,w_n) < r + D(z,w_n)
\]
for each $n$, and taking the $n \rta \infty$ limit of these two inequalities yields $D(z,X) = r$.}

Conversely, if $D(z,X) = r$, then for every $n \in \BB{N}$, we can choose a path $P_n$ from $z$ to a point in $X$ with $D$-length less than $r + 1/n$.  The path $P_n$ contains a point $z_n$ with $D(z_n,X) < r$ and $D(z_n,z) < 1/n$.  Therefore, the sequence $(z_n)_{n \in \BB{N}}$ is contained in $B_r(x;D)$ and converges to $z$ in the metric $D$.
\end{proof}

We now prove Proposition~\ref{prop-compact}, beginning with property~\ref{item-compact-prop-a}.  First, we prove a pair of lemmas that allow us to ``grow'' a $D_h$-geodesic of any specified $D_h$-length from the boundary of an open set.

\begin{lem}
\label{lem-not-singular-sequence}
\changes{Suppose that $D$ is a lower semicontinuous length metric on an open set $U \subset \BB{C}$.
If  $X \subset \BB{C}$ is a compact set and $z \in \BB{C} \backslash X$ is a singular point of the metric $D$,} then every sequence $(z_n)_{n \in \BB{N}}$ converging to $z$ in the Euclidean topology satisfies $D(z_n,X) \rta \infty$ as $n \rta \infty$.
\end{lem}

\begin{proof}
Suppose that $X \subset \BB{C}$, and that $(z_n)_{n \in \BB{N}}$ is a sequence in $\BB{C}$ converging to a point $z \in \BB{C}$ in the Euclidean topology, such that $D(z_n,X) \leq r$ for all $n$ and some $r>0$. Let $x_n \in X$ be such that $D(x_n,z_n) \leq r+1$\changes{, and let $x$ be a limit point of $\{x_n\}_{n \in \BB{N}}$.}  By lower semicontinuity of $D$, we have $D(x,z) \leq \liminf_{n \rta \infty} D(x_n,z_n) \leq r+1$, so $z$ is not a singular point.
\end{proof}

\begin{lem}
Suppose that $D$ is a lower semicontinuous length metric on $\BB{C}$ such that  both $\{\text{singular points}\}$ and $\BB{C} \backslash \{\text{singular points}\}$ are dense. Let $U \subset \BB{C}$ be an open set \changes{such that $\partial U$ is bounded and contains some nonsingular  point.} Then, for every $r>0$, there is a point $z \in U$ with $D(z,\partial U) = r$, and a $D$-geodesic $P$ from $\partial U$ to $z$ with $D$-length $r$, such that $P$ minus its initial point is contained in $U$. \label{lem-geodesic-stick}
\end{lem}

\begin{proof}
It suffices to prove the result for all positive $r \leq R$, for each fixed $R>0$.
Since  both $\{\text{singular points}\}$ and $\BB{C} \backslash \{\text{singular points}\}$ are dense, we can choose a singular point $w \in U$ which is the limit (in the Euclidean topology) of a sequence $(w_n)_{n \in \BB{N}}$ of nonsingular points.  By Lemma~\ref{lem-not-singular-sequence}, we have $D(w_n,\partial U) \rta \infty$ as $n \rta \infty$. In particular, we can choose $N \in \BB{N}$ such that $D(w_N,\partial U) > R$. 

Now, choose $(x^j)_{j \in \BB{N}}$ \changes{in $\partial U$} such that $D(w_N,x^j) \leq D(w_N,\partial U) + 1/j$.  \changes{Since $\partial U$ is bounded, the sequence $(x^j)_{j \in \BB{N}}$ has a limit point $x \in \partial U$.} By lower semicontinuity of $D$, we have $D(w_N,x) = D(w_N,\partial U)$. 

By Lemma~\ref{lem-complete-geodesic}, every two nonsingular points can be joined by a $D$-geodesic.
In particular, this means that there is a geodesic $P:[0,D(x,w_{\changes{N}})] \rta \ol U$ from $x$ to $w_N$. Since $x$ minimizes $D$-distance to $w_N$ over all points in $\partial U$, we have $P(r) \in U$ for every $0<r\leq D(x,w_{\changes{N}})$. Moreover, for every $0<r\leq D(x,w_{\changes{N}})$, we have $D(P(r),\partial U) = r$. Hence, the point $z= P(r)$ and the geodesic $P|_{[0,r]}$ has the desired properties.
\end{proof}

\begin{proof}[Proof of Proposition~\ref{prop-compact}, property~\ref{item-compact-prop-a}]
First, we observe that, since the set of singular points is a.s.\ dense (Proposition~\ref{prop-singular}), it suffices to consider Borel sets $X$ that are just a single point.  This is because, a.s.\ if $X$ is a Borel set \changes{containing no singular points} and $U$ is a complementary connected component of $B_r[z;D_h]$ for some $r>0$ and $z \in X$, then $U$ must contain a singular point and therefore a complementary connected component of $B_r[X;D_h]$.  
By Proposition~\ref{prop-rational-dense}, a.s.\ for every $\ep > 0$ and every nonsingular point $z \in \BB{C}$, there is an annulus around $z$ with Euclidean diameter less than $\ep$ and with the $D_h$-distance around the annulus less than $\ep$. Hence, a.s.\ for every $\ep > 0$ and every geodesic $P$ in $\BB{C}$, we can choose a collection $\mcl L = \mcl L(P)$ of pairwise disjoint loops around annuli (measurable with respect to $P$), such that
\begin{itemize}
\item every loop in $\mcl L$ crosses $P$ and has $D_h$-length less than $\ep$, and
    \item $|\mcl L| > C_P \ep^{-1}$ for some constant $C_P$ depending on $P$ but not on $\ep$.
\end{itemize}
Now, let $\delta > 0$. By Lemma~\ref{lem-complete-geodesic}, a.s.\ for every \changes{nonsingular} $z \subset \BB{C}$, the set $B_\delta[z;D_h]$ contains a geodesic segment $P \subset \BB{C}$.  
Observe that,  a.s.\ for every geodesic $P$ in $B_\delta[z;D_h]$, every loop in $\mcl L = \mcl L(P)$ surrounds a singular point and therefore a complementary connected component of $B_r[z;D_h]$ for every $r>\delta + \ep$.  

Hence, a.s.\ for every $\ep,\delta>0$ and \changes{nonsingular} $z \in \BB{C}$, the set $B_r[z;D_h]$ has at least $C \ep^{-1}$ complementary connected components for every $r>\delta + \ep$, for some random $C>0$ independent of $\ep$.
The result follows from taking $\ep$ and $\delta$ arbitrarily small.
\end{proof}
 
Next, we prove property~\ref{item-compact-prop-b} by applying property~\ref{item-compact-prop-a}, Lemma~\ref{lem-geodesic-stick} above, and the following elementary metric space lemma.

\begin{lem}
\label{lem-sum-L}
Suppose that $D$ is a \changes{length} metric defined on a set $U \subset \BB{C}$, \changes{and let $X \subset U$ and $r>0$}. If $V$ is a connected component of $U \backslash B_r[X,D]$, and $z \in V$, then $D(z,\partial V) + r = D(z,X)$.
\end{lem}

\begin{proof}
Let $\delta>0$. If $P$ is a path from $z$ to a point $v \in \partial V$ with $D$-length at most $D(z,\partial V) + \delta$, then \[
D(z,X) \leq D(v,X) + D(v,z) \leq r + D(z,\partial V) + \delta.\] Conversely, if $P'$ is a path from $z$ to a point $x \in X$ with $D$-length at most $D(z,X) + \delta$, then
\[
D(z,\partial V) + D(X, \partial V) \leq \len(P) \leq D(z,X) + \delta.\]
\end{proof}

\begin{proof}[Proof of Proposition~\ref{prop-compact}, property~\ref{item-compact-prop-b}]
\changes{
To prove the assertion, it suffices to show that a.s.\ for every nonsingular point $z \in \BB{C}$ and $s<r$, the set $B_r[z;D_h]$  cannot be covered by finitely many $D_h$-balls of radius $s$.  By property~\ref{item-compact-prop-a}, a.s., for every $r>0$ and every nonsingular point $z \subset \BB{C}$, the set $B_{r-s}[z;D]$ has infinitely many complementary connected components $U_1,U_2,\ldots$.   The sets $\partial U_n$ are $D$-bounded, hence Euclidean bounded by Lemma~\ref{lem-Dh-Eucl-bounded}.
By Lemmas~\ref{lem-geodesic-stick} and~\ref{lem-sum-L}, a.s.\ each $U_n$ contains a point $z_n \in B_r[z;D]$ at $D_h$-distance $s$ from $\partial U_n$. Almost surely, for each $n\neq m$, since $D_h(z_{n},z_{m}) \geq s+s = 2s$, the points $z_n$ and $z_m$ cannot both be contained in a $D_h$-ball of radius $<s$.}
\end{proof}

Finally, we prove property~\ref{item-compact-prop-c} of Proposition~\ref{prop-compact} by applying the following lemma.

\begin{lem}
\label{lem-sequence-Uk}
\changes{Let $D$ be a lower semicontinuous length metric on $\BB{C}$ such that  both $\{\text{singular points}\}$ and $\BB{C} \backslash \{\text{singular points}\}$ are dense,} and such that every closed $D$-ball with positive radius \changes{(centered at a nonsingular point)} has infinitely many complementary connected components.
Let $r>\ep>0$, let $X \subset \BB{C}$, and let $X_r$ denote the $D$-boundary of the set $B_r(X;D)$.  Let $U$ be a \changes{bounded} complementary connected component of $B_{r-\ep}[X;D]$. Then $U$ contains infinitely many complementary connected components $U_1,U_2,\ldots$ of $B_{r-\ep/4}[X;D]$ such that $U_k \cap X_r \neq \emptyset$ and 
\[
\diam(U_k \cap X_r; D) \leq 5\ep/2.\]
\end{lem}
\begin{proof}
Observe that $U$ is simply connected since $B_{r-\ep}[X;D]$ is connected. Also,  $U$ is open in the Euclidean topology by Lemma~\ref{lem-closed-ball-is-closed} \changes{and $\partial U$ is bounded since $U$ is bounded}.  Thus, by Lemma~\ref{lem-geodesic-stick}, we can choose $z \in U$ with $D(z,\partial U) = \ep$, and a $D$-geodesic $P:[0,\ep] \rta \ol U$, parametrized by $D$-length, from $\partial U$ to $z$ with $P(0,\ep] \subset U$.  By hypothesis, the complement of the ball $B_{\ep/4}[P(\ep/2);D]$ has infinitely many complementary connected components.  Since $B_{\ep/4}[P(\ep/2);D]$ is contained in $U$ and $U$ is simply connected, we deduce that the set $U \backslash B_{\ep/4}[P(\ep/2);D]$ has infinitely many  connected components $V_0,V_1,\ldots$. Take $V_0$ to be the unique component in this collection with $\partial V_0 \cap \partial U \neq \emptyset$. Let $k \in \BB{N}$.  Observe that every point $v \in \partial V_k$ satisfies
\[
D(v, P(\ep/2)) \leq \ep/4 \qquad \text{and} \qquad D(v,X) \leq (r-\ep) + \ep/2 + \ep/4 = r - \ep/4;\]
i.e., $\partial V_k \subset B_{r-\ep/4}[X;D]$.  
Hence, $V_k$ contains a connected component $U_k$ of $\BB{C} \backslash B_{r-\ep/4}[X;D]$.  \changes{Since $\partial U_k$ is bounded, we may apply Lemmas~\ref{lem-geodesic-stick} and~\ref{lem-sum-L} to deduce that $U_k$ contains a point at $D$-distance exactly $r$ from $X$. By  Lemma~\ref{lem-boundary-r}, the set of points at $D$-distance exactly $r$ from $X$ is equal to the set $X_r$.  We deduce that $U_k \cap X_r \neq \emptyset$.} Thus, Lemmas~\ref{lem-geodesic-stick} and~\ref{lem-sum-L} imply that $U_k \cap X_r \neq \emptyset$.  Moreover, by 
Lemma~\ref{lem-sum-L}, every point in $U_k \cap X_r$ is at $D$-distance exactly $\ep$ from $\partial U$. Since a path from a point in $U_k$ to $\partial U$ must intersect $\partial V_k$, we deduce that any two points in $U_k \cap X_r$ 
are at most $D$-distance $2\ep + 2 D(v,P(\ep/2)) \leq 2\ep + \ep/2$ apart.
\end{proof}

  \begin{proof}[Proof of Proposition~\ref{prop-compact}, property~\ref{item-compact-prop-c}]
  Let $X_r$ denote the $D_h$-boundary of $B_r(X;D_h)$.
  Fix $d>0$.  
  It suffices to show that a.s., for every collection $\frk W$ of subsets of $X_r$ with 
  \eqb
  \sum_{W \in \frk W} \diam(W;D_h)^d < \infty \qquad \text{and} \qquad \max_{W \in \frk W} \diam(W;D_h) < r/4,
  \label{eqn-W-conditions}
  \eqe
  the collection $\frk W$ does not cover $X_r$.
    To prove this assertion, we begin by partitioning  $\frk W$ as \[
  \frk W = \bigsqcup_{n \in \BB{N}} \frk W_n,\] where \[ \frk W_n := \left\{W \in  \frk W : \diam(W;D_h) \in [r/4^{n+1}, r/4^n)\right\}.\]
  Observe that, if $U$ is a complementary connected component of $B_{r-r/4^n}[X;D_h]$, then any set $W \in \frk W_n$ that intersects $U \cap X_r$ must be contained in  $U$. Since $|\frk W_n| < \infty$ by~\eqref{eqn-W-conditions}, we deduce that a.s.\ the collection of sets in $\frk W_n$ intersect at most finitely many complementary connected components of $B_{r-r/4^n}[X;D]$.  \changes{Also, a.s.\ the complement of $B_{r - r/4}[X;D_h]$ has infinitely many bounded connected components.\footnote{If $z \in X$, then $B_{r - r/4}[z;D_h]$ has infinitely many bounded complementary connected components. Each of these components contains a singular point, hence a complementary connected component of $B_{r - r/4}[X;D_h]$.}
  Therefore, a.s.\ we can  choose a bounded complementary connected component $U^1$ of $B_{r - r/4}[X;D_h]$ whose closure is disjoint from every set $W \in \frk W_1$ and such that (by Lemma~\ref{lem-geodesic-stick}) $U^1 \cap X_r \neq \emptyset$. Moreover,} by repeatedly applying Lemma~\ref{lem-sequence-Uk}, we can a.s.\ define a nested sequence of open sets $U^1 \supset U^2 \supset U^3 \supset \cdots$ such that 
  \eqb
  \text{$\ol{U^n} \cap W = \emptyset$ for every $n \in \BB{N}$ and every $W \in \frk W_n$}
  \label{eqn-Un-W}
  \eqe
  and with $U^n \cap X_r  \neq \emptyset$ and $\diam(U^n \cap X_r; D) \leq (5/2)(r/4^{n-1})$ for each $n>1$.  A.s.\ we can choose $z^n \in U^n \cap X_r$ for each $n$, and the diameter bound implies that $z^n$ is $D_h$-Cauchy.  By Lemma~\ref{lem-complete-geodesic}, a.s.\ $z^n$ converges to a point $z$ in the metric $D_h$; since $X_r$ is $D_h$-closed, $z \in X_r$.  On the other hand,~\eqref{eqn-Un-W} implies that $z$ is not contained in any of the sets in $\frk W$.
   \end{proof}

\section{The KPZ formula}
\label{sec-kpz}

In this section, we prove Theorem~\ref{thm-kpz}, which is a version of the KPZ formula that holds for all $\changes{\cc \in (-\infty,25)}$ that is stated in terms of two dual notions of fractal dimension: Hausdorff dimension and packing dimension.  We begin by defining these concepts.

\begin{defn}[Hausdorff and packing dimension of a set]
\label{defn-dim}
Let $(X,D)$ be a metric space. For a subset $S \subset X$ and $s>0$, we define the $s$-dimensional \emph{$D$-Hausdorff measure} of $S$ as
\eqbn
\mcl H^s(S\changes{;D}) =  \lim_{r \rta 0} \inf\left\{\sum_{j=1}^\infty r_j^s :
 \begin{array}{l}
    \text{there is a cover of $S$ by $D$-balls}  \\
    \text{with radii $\{r_j\}_{j\geq 0}$ less than $r$ }
  \end{array}
\right\} 
\eqen
and the $s$-dimensional \emph{$D$-packing measure} of $S$  as
\eqbn
\mcl P^s(S;D) =  \inf\left\{\sum_{j=1}^\infty \mcl P_0^s(S_j;D): S \subset \bigcup_{j=1}^{\infty} S_j \right\},
\eqen
where $\mcl P_0^s(S;D)$ is the pre-measure defined as
\[
\mcl P_0^s(S;D) =  \lim_{r \rta 0} \changes{\sup}\left\{\sum_{j=1}^\infty r_j^s : 
 \begin{array}{l}
    \text{there are pairwise disjoint $D$-balls with centers}  \\
    \text{in $S$ and with radii $\{r_j\}_{j\geq 0}$ less than $r$ }
  \end{array}
\right\}.
\]
We define the \emph{$D$-Hausdorff dimension} of $S$ as
\[
\dim_{\mcl H}(S;D) := \inf\{s > 0: \mcl H^s(S\changes{;D}) = 0\},
\]
and the \emph{$D$-packing dimension} of $S$ as
\[
\dim_{\mcl P}(S;D) := \inf\{s > 0: \mcl P^s(S\changes{;D}) = 0\}.
\]
When $D$ is the Euclidean metric, we drop the prefix $D$-, and we denote the Hausdorff and packing dimensions of a set $S$ simply by $\dim_{\mcl H} S$ and $\dim_{\mcl P} S$.
\end{defn}

It is well-known that both Hausdorff and packing dimension have the following \emph{countable stability} property: if $(X,D)$ is a metric space and $S_j$ are subsets of $X$ for $j \in \BB{N}$, then the $D$-dimension of the union $\bigcup_{j \in \BB{N}} S_j$ is equal to the supremum over all $j \in \BB{N}$ of the $D$-dimensions of the sets $S_j$.

We noted in the introduction that, unlike its counterpart~\cite[Theorem 1.4]{gp-kpz} in the setting $\changes{\cc \in (-\infty,1)}$, the KPZ formula stated in Theorem~\ref{thm-kpz} is \emph{not} an equality: the lower bound is in terms of Hausdorff dimension while the upper bound is in terms of packing dimension. We also mentioned that it is not clear to us whether the KPZ formula is true for $\cc \in (1,25)$ with packing dimension replaced by Hausdorff dimension, although the two notions of dimension coincide for many of the sets we would like to analyze, such as SLE-type sets~\cite{schramm-sle, beffara-dim,lawler-rezai-nat}.\footnote{\changes{Specifically, the result~\cite[Theorem 1.1]{lawler-rezai-nat} implies that the Minkowski dimension of SLE paths exists and is equal to Hausdorff dimension of SLE paths identified in~\cite{beffara-dim}.}}

The reason that we consider packing dimension instead of Hausdorff dimension in the upper bound is that, when $\cc \in (1,25)$, we encounter the following technical difficulty.  To prove an upper bound on $D_h$-Hausdorff dimension, we use a moment bound on $D_h$-diameters of sets in $\BB{C}$. When $\changes{\cc \in (-\infty,1)}$, we can bound moments of $D_h$-diameters of open subsets of the plane~\cite[Proposition 3.9]{lqg-metric-estimates} by bounding the $D_h$-diameters of hashes and connecting these hashes across scales by applying Lemma~\ref{lem-ann-hash-geometry}.  When $\cc\in (1,25)$, however, 
the singular points are dense (Proposition~\ref{prop-singular}), and so the $D_h$-diameter of an open set is a.s.\ infinite. This corresponds to the fact that we cannot bound the $D_h$-diameters of hashes associated to all squares at a given dyadic scale.  Therefore, we instead bound the $D_h$-diameter of a open set restricted to a smaller-dimensional fractal $X$.  The idea is that this smaller-dimensional fractal $X$ intersects a smaller order number of dyadic squares at each scale, so that we can apply a union bound to bound the $D_h$-diameters of hashes associated to these smaller number of squares.

The difficulty with this approach is that this moment bound requires an upper bound on the order of the number of dyadic squares intersecting $X$ \emph{at each scale}.  More precisely, it is not enough to stipulate that $X$ has bounded Hausdorff dimension; instead, we require the stronger assumption that $X$ has bounded \emph{upper Minkowski dimension}.

\begin{defn}
For a subset $S$ of a metric space $(X,D)$, the upper Minkowski dimension $\ol{\dim_{\mcl M}}(S;D)$ of $S$ in the metric $D$ is defined as
\[
\ol{\dim_{\mcl M}}(S;D) := \limsup_{\ep \rta 0} \frac{\log(\text{number of $D$-balls of radius $\ep$ needed to cover $S$})}{\log \changes{\ep^{-1}}}
\]
In our context, in which we take $(X,D)$ to be $\BB{C}$ equipped with the Euclidean metric, we can equivalently define upper Minkowski dimension in terms of dyadic squares rather than balls.
\end{defn}

By implementing the approach just described, we can show that, with $f$ defined in the statement of Theorem~\ref{thm-kpz}, 
\[
\dim_{\mcl H}(X;D_h) \leq f(\ol{\dim_{\mcl M}} X).
\]
A priori, this statement is weaker than the upper bound in the statement of Theorem~\ref{thm-kpz}, since, for a general set $S$ in a metric space $(X,D)$,
\[
{\dim_{\mcl H}}(S;D) \leq {\dim_{\mcl P}}(S;D) \leq \ol{\dim_{\mcl M}}(S;D).
\]
However, it is well-known   (see, e.g.,~\cite[Proposition 3.8]{falconer}) that every set $S$ can be expressed as a countable union of sets with upper Minkowski dimension \changes{$t$, for each  $t>\dim_{\mcl P} S$}.  By countable stability of Hausdorff dimension, this means that the upper bound of Theorem~\ref{thm-kpz} is equivalent to the statement with packing dimension replaced by upper Minkowski dimension.

\subsection{A KPZ formula for the subset of thick points in a fractal}

Before proving Theorem~\ref{thm-kpz}, we prove Theorem~\ref{thm-kpz-thick}, which relates, a.s.\ for all Borel sets $X \subset \BB{C}$, the Hausdorff dimensions of $X \cap \mcl T_h^\alpha$ with respect to the Euclidean metric and the metric $D_h$. The statement of Theorem~\ref{thm-kpz-thick} is easy to check for $\alpha \in (Q,2]$:

\begin{proof}[Proof of Theorem~\ref{thm-kpz-thick} for $\alpha > Q$.]
The result follows  immediately from the fact that every $\alpha$-thick point is a singular point (Proposition~\ref{prop-singular}).
\end{proof}

We now prove Theorem~\ref{thm-kpz-thick} for $\alpha < Q$.  As in the proof of~\cite[Theorem 1.5]{gp-kpz}, instead of considering the set $ X \cap \mcl T^\alpha_h$ directly, we  consider a set of points for which $h_\ep(z) / \log(\ep^{-1})$ is \emph{uniformly} close to $\alpha$ (note that the rate of convergence in~\eqref{eqn-thick-pts} can depend on $z$). The following lemma, stated in the classical LQG regime as~\cite[Lemma 2.1]{gp-kpz}, asserts that we can find such a set whose Hausdorff dimensions in the Euclidean metric and in $D_h$ are not too  different from those of $X\cap \mcl T_h^\alpha$. 

\begin{lem}  \label{lem-thick-subset}
Let $h$ be a whole-plane GFF normalized so that $h_1(0) = 0$, \changes{and $D_h$ be the metric associated to $h$ by a weak LQG metric for some $\xi > 0$.} Let $\alpha \in \BB R$ and $\zeta  > 0$. Almost surely, for each bounded Borel set $X\subset \BB C$ there exists a random $ \overline{\ep} >0$ depending on $\alpha $, $\zeta$, and $X$ such that the following is true. If we set
\begin{equation}\label{eqn-defn-X-alpha-zeta}
X^{\alpha, \zeta} := \left\{ z \in X: \frac{h_{\ep}(z)}{\log (\ep^{-1})} \in [\alpha - \zeta, \alpha+ \zeta], \forall \ep \in (0, \overline{\ep}] \right\}
\end{equation}
then $\dim_{\mcl{H}} X^{\alpha, \zeta}  \geq \dim_{\mcl H} (X \cap \mcl T_h^\alpha)- \zeta$ and $\dim_{\mcl{H}}(X^{\alpha, \zeta};D_h)  \geq \dim_{\mcl H}(X \cap \mcl T_h^\alpha;D_h) - \zeta$.  \end{lem}

The advantage of  considering this set of ``uniformly thick'' points is that, for a bounded domain $V \subset \BB{C}$, we can uniformly bound the circle average of points in a sufficiently small neighborhood of $V^{\alpha,\zeta}$.  \changes{The following is a restatement of~\cite[Lemma 2.2]{gp-kpz}.

\begin{lem} \label{lem-circle-avg-bound}
Let $\alpha \in \BB R$ and $\zeta \in (0,1)$.
Almost surely, for each bounded Borel set $X\subset \BB C$, there exists a $\delta > 0$ (depending on $\alpha,\zeta$, and $X$) such that for each $r \in (0,\delta)$ and each $z \in \BB C$ which lies at  (Euclidean)  distance at most $r$ from a point in $X^{\alpha,\zeta}$, 
\[
\frac{h_{r}(z)}{\log (r^{-1})} \in [m_{\alpha,\zeta}, M_{\alpha,\zeta}] ,
\]
where $m_{\alpha,\zeta} := (1-\zeta^2) (\alpha - \zeta) - 3\sqrt{10} \zeta$ and $M_{\alpha,\zeta} := (1-\zeta^2) (\alpha + \zeta) + 3\sqrt{10} \zeta$.
\end{lem}
}

Combining this circle average estimate with Lemma~\ref{lem-ann-hash-dist} yields the following distance bounds:

\begin{lem}
\label{lem-ann-hash-dist-thick}
Let $V \subset \BB{C}$ be a bounded open set.  Let $h,\alpha,\zeta$  be as in Lemma~\ref{lem-thick-subset}, and set $r^{\pm}_{\alpha,\zeta} := (1-\zeta^2) (\alpha \pm \zeta) \pm (3\sqrt{10} + 1) \zeta$.
 For each $\BB{r} > 0$, the following bounds hold for each $S \in \mcl S^n(V^{\alpha,\zeta})$ with superexponentially high probability as $n \rta \infty$, at a rate which is uniform in $\BB{r}$.
\begin{enumerate}
    \item \textit{(Distance lower bound.)}
\eqb
D_h(\text{across $A_S$}) \geq 2^{-\xi (Q-r^-_{\alpha,\zeta}) n}.
\label{eqn-sq-ann-bound-across-thick}
\eqe
     \item \textit{(Distance upper bound.)} 
     \eqb
D_h(\text{around $A_S$}) \leq 2^{-\xi (Q-r^+_{\alpha,\zeta}) n},
\label{eqn-sq-ann-bound-around-thick}
\eqe
     and we can choose a hash $\#_S$ associated to $S$ with
\eqb
\diam(\#_S;D_h) \leq 2^{-\xi (Q-r^+_{\alpha,\zeta}) n}.
\label{eqn-sq-hash-bound-thick}
\eqe
\end{enumerate}
\end{lem}

\begin{proof}
\changes{The lemma follows immediately from Lemmas~\ref{lem-circle-avg-bound} and~\ref{lem-ann-hash-dist}.}
\end{proof}

We now prove Theorem~\ref{thm-kpz-thick}.  By the countable stability of Hausdorff dimension, we can restrict our attention to Borel sets $X$ that are contained in some deterministic bounded open subset $V \subset \BB{C}$.  Moreover, by Weyl scaling (Axiom~\ref{item-metric-f}), we can assume without loss of generality that $h$ is a whole-plane GFF normalized so that $h_1(0) = 0$.

\begin{proof}[Proof of Theorem~\ref{thm-kpz-thick} for $\alpha < Q$, upper bound.]
By Lemma~\ref{lem-ann-hash-dist-thick}, a.s.\ there exists $N \in \BB{N}$ such that for every $n \geq N$ and $S \in \mcl S^n(V^{\alpha,\zeta})$,  we can choose a hash $\#_S$ associated to $S$ with
\[
\diam(\#_S;D_h) \leq 2^{-\xi (Q - r^+_{\alpha,\zeta}) n}. \]
Since a hash associated to a square must intersect a hash associated to each of its dyadic children, this implies that, for every $n \geq N$ and $S \in \mcl S^n(V^{\alpha,\zeta})$, the union of hashes 
\[ H_S:= \bigcup_{m \in \BB{N}} \bigcup_{S \supset S' \in \mcl S^{m}(V^{\alpha,\zeta})} \#_{S'} \] a.s.\ satisfies
\[
\diam(H_S;D_h) \leq 2 \sum_{m = n}^{\infty} 2^{-\xi (Q - r^+_{\alpha,\zeta}) m},
\] and therefore, by Lemma~\ref{lem-closure},
\eqb
\diam(\ol H_S;D_h) \leq 2 \sum_{m = n}^{\infty} 2^{-\xi (Q - r^+_{\alpha,\zeta}) m}. 
\label{eqn-Hs-bound}
\eqe
Since $S \cap V^{\alpha,\zeta} \subset \ol H_S$ by construction of $H_S$,~\eqref{eqn-Hs-bound} implies
\[
\diam(S \cap V^{\alpha,\zeta}; D_h) \leq \diam(\ol H_S;D_h) \leq 2 \sum_{m = n}^{\infty} 2^{-\xi (Q - r^+_{\alpha,\zeta}) m}
\] and therefore
\[
\diam(S \cap V^{\alpha,\zeta}; D_h)^{1/[\xi(Q - r^+_{\alpha,\zeta})]} \leq C \diam(S)
\]
for some deterministic constant $C = C(\xi,\alpha,\zeta,Q) > 0$ independent of $n,S$.

We conclude that a.s., for every Borel set $X \subset V$,
\[
\dim_{\mcl H}(X^{\alpha,\zeta};D_h) \leq \frac{1}{\xi(Q - r^+_{\alpha,\zeta})} \dim_{\mcl H} X^{\alpha,\zeta}.
\]
Since, by Lemma~\ref{lem-thick-subset}, a.s.\ $\dim_{\mcl H}(X^{\alpha,\zeta};D_h) \geq \dim_{\mcl H}(X \cap \mcl T_h^\alpha;D_h) - \zeta$, we see that sending $\zeta \rta 0$ yields the desired result.
\end{proof} 

\begin{proof}[Proof of Theorem~\ref{thm-kpz-thick} for $\alpha < Q$, lower bound.]
\changes{By Lemma~\ref{lem-ann-hash-dist-thick}, a.s.\ there exists $N \in \BB{N}$ such that the bounds \eqref{eqn-sq-ann-bound-across-thick},~\eqref{eqn-sq-ann-bound-around-thick}, and~\eqref{eqn-sq-hash-bound-thick} hold for each $n \geq N$ and $S \in \mcl S^n(V^{\alpha,\zeta})$.} Suppose that $W$ is a set that intersects $V^{\alpha,\zeta}$ with sufficiently small Euclidean diameter---specifically, with
 \eqb
 \diam(W) \in [2^{-n+2},2^{-n+3}) \qquad \text{for some $n \geq N$.}
 \label{eqn-W-diam-cond}
 \eqe
 By Lemma~\ref{lem-defn-sq-annulus}(c), we can choose $S \in \mcl S^{n}(V^{\alpha,\zeta})$ intersecting $W$ such that some point of $W$ lies outside $A_S$.  This means that $\diam(W)$ is bounded from below by the distance across $A_S$.  From Lemma~\ref{lem-ann-hash-dist-thick}, we  deduce that a.s.\ for every $W$ intersecting $V^{\alpha,\zeta}$ and satisfying~\eqref{eqn-W-diam-cond},
 \eqb
 \diam(W; D_h) \geq 2^{-\xi (Q - r^-_{\alpha,\zeta})n} \geq (\diam(W)/8)^{\xi (Q - r^-_{\alpha,\zeta})}.
 \label{eqn-lower-bound-3}
 \eqe
 By Proposition~\ref{prop-holder}, the condition~\eqref{eqn-W-diam-cond} is satisfied whenever $W$ has sufficiently small $D_h$-diameter. Hence a.s., the inequality~\eqref{eqn-lower-bound-3} holds simultaneously for every set $W$ intersecting $V^{\alpha,\zeta}$ with $D_h$-diameter less than some random threshold.

We conclude that a.s., for every Borel set $X \subset V$,
\[
\dim_{\mcl H}(X^{\alpha,\zeta};D_h) \geq \frac{1}{\xi(Q - r^-_{\alpha,\zeta})} \dim_{\mcl H} X^{\alpha,\zeta}
\]
Since, by Lemma~\ref{lem-thick-subset}, a.s.\ $\dim_{\mcl H} X^{\alpha,\zeta} \geq \dim_{\mcl H} (X \cap \mcl T_h^\alpha) - \zeta$, we see that sending $\zeta \rta 0$ yields the desired result.
\end{proof}

\subsection{A moment bound for diameters}

Theorem~\ref{thm-kpz-thick} immediately yields a proof of the lower bound of Theorem~\ref{thm-kpz} on  $\dim_{\mcl H}(X;D_h)$.

\begin{proof}[Proof of Theorem~\ref{thm-kpz}, lower bound]
First, assume that $\dim_{\mcl H} X < Q^2/2$. \changes{Combining Theorem~\ref{thm-kpz-thick} with~\eqref{eqn-thick-dim} gives a lower bound on $\dim_{\mcl H}\left( X \cap \mcl T_h^{\alpha}; D_h \right)$ in terms of $\dim_{\mcl H} X$.  This lower bound} is maximized by $\alpha = Q - \sqrt{Q^2 - 2 \dim_{\mcl H} X}$. (Note that $\alpha \neq Q$.) 
For this value of $\alpha$, we have
\eqb
\dim_{\mcl H} \left( X \cap \mcl T_h^\alpha; D_h \right) \geq \frac{1}{\xi} \left( Q - \sqrt{Q^2 - 2 \dim_{\mcl H} X} \right).
\label{laststep}
\eqe
Since $\dim_{\mcl H}(X;D_h) \geq \dim_{\mcl H} \left( X \cap \mcl T_h^\alpha ; D_h \right)$, this yields the desired lower bound on $\dim_{\mcl H}(X;D_h)$. 

Next, if $\dim_{\mcl H} X > Q^2/2$, then by applying  Theorem~\ref{thm-kpz-thick} and~\eqref{eqn-thick-dim} for values of $\alpha < Q$ arbitrarily close to $Q$, we deduce that $\dim_{\mcl H}(X;D_h)$ is infinite.

Finally, if $\dim_{\mcl H} X = Q^2/2$, then applying  Theorem~\ref{thm-kpz-thick} and~\eqref{eqn-thick-dim} gives
\[
\dim_{\mcl H} \left( X \cap \mcl T_h^\alpha; D_h \right) \geq
\frac{1}{\xi(Q-\alpha)} \left( \frac{Q^2}{2} - \frac{\alpha^2}{2} \right) = \frac{Q+\alpha}{2\xi}
\]
for each $\alpha < Q$, so sending $\alpha \rta Q$ yields the desired lower bound.
\end{proof}

Proving the upper bound of Theorem~\ref{thm-kpz} on  $\dim_{\mcl H}(X;D_h)$ is considerably more difficult. To prove this upper bound, we will need the following moment bound on diameters of sets with bounded upper Minkowski dimension, which generalizes~\cite[Proposition 3.9]{lqg-metric-estimates}. 

\begin{prop}[Moment bound for diameters] \label{prop-internal-moment}  Let $h$ be a whole-plane GFF normalized so that $h_1(0) = 0$.
For each $p \in \BB{R}$, we can choose a deterministic constant $C_p>0$ such that the following is true.
Let $X$ be a Borel set contained in the closed Euclidean square  $\BB{S}: = [0,1]^2$, such that $|\mcl S^n(X)| \leq T 2^{xn}$ for some $x,T>0$ and all $n \in \BB{N}$. Then, for every real number $p <  \xi^{-1} (Q + \sqrt{Q^2 - 2x})$ and every $\BB r >0$,
\eqb \label{eqn-internal-moment}
\BB E\left[\left( \frk c_{\BB r}^{-1} e^{-\xi h_{\BB r}(0)} \sup_{u,v \in \BB{r} X} D_h(u,v;\BB{r} \BB S)  \right)^p \right] \leq T C_p. 
\eqe
\end{prop}

\begin{proof}[Proof of Proposition~\ref{prop-internal-moment}]
For $p < 0$, the bound~\eqref{eqn-internal-moment} follows from the lower bound of Proposition~\ref{prop-two-set-dist}. 
We will bound the positive moments up to order $\xi^{-1} (  Q + \sqrt{Q^2-2x}  )$.  

Fix $q\in (2,Q)$ which we will eventually send to $Q$.
By a circle average estimate,\footnote{We can prove this circle average estimate in the exact same manner as~\cite[Lemma 3.11]{lqg-metric-estimates}.  The only difference is that we define the events $E_{\BB{r}}^n$  with the supremum taken over all centers of squares in $\BB{r} \mcl S^n(\BB{r} X)$ rather than all $w \in B_{R\BB{r}} \cap (2^{-n-1} \BB{r} \BB{Z}^2)$.  This means that, in the estimates (3.30)-(3.33) in that proof,  instead of taking a union bound over $2^{2n}$ points for every $n \in \BB{N}$, we  take a union bound over $T 2^{xn}$ points.  We refer the reader to the proof of~\cite[Lemma 3.11]{lqg-metric-estimates} for the details.} it holds with probability $1 - T C^{-q-\sqrt{q^2-2x} + o_C(1)} $, with the $o_C(1)$ uniform in $\BB{r}, T,x,X$, that \eqb \label{eqn-use-circle-avg-all}
 \sup\left\{ |h_{2^{-n}\BB r}(v_{\BB{r} S}) - h_{\BB r}(0)|  :  S \in \mcl S^n(X)  \right\} \leq   \log(C 2^{q n})  ,\quad\forall n \in \BB{N}.
\eqe
Now fix $\zeta\in (0,Q-q)$, which we will eventually send to zero.  By Proposition~\ref{prop-two-set-dist} (applied with $A = C^\zeta 2^{\xi \zeta n}$), Axiom~\ref{item-metric-translate}, and a union bound over all $S \in \mcl S^n(\BB{S})$,
it is the case with superpolynomially high probability as $C \rta \infty$ (at a rate that is uniform in $\BB{r}>0$ but depends on $\zeta$) that, for each $\BB{r} > 0$, $n \in \BB{N}$ and $S \in \mcl S^n(\BB{S})$, we can choose a hash $\#_{\BB{r} S}$ associated to $\BB{r} S$ with
\eqb
\sup_{u,v \in \#_{\BB{r} S}} D_h(u,v;\BB{r} \BB S) 
 \leq C^{\zeta} \changes{2^{\xi \zeta n}} \frk c_{2^{-n} \BB{r}} e^{\xi h_{2^{-n} \BB{r}}(v_{\BB{r} S})}.
\label{eqn-sq-hash-bound-int}
\eqe
Combining~\eqref{eqn-sq-hash-bound-int} with~\eqref{eqn-use-circle-avg-all} and~\eqref{eqn-metric-scaling}, we deduce that for all $C>1$, it holds with probability $1 -  T C^{-q-\sqrt{q^2-2x} + o_C(1)} $ (with the $o_C(1)$ uniform in $\BB{r}, T,x,X$) that, for all $\BB{r}>0$, $n \in \BB{N}$ and $S \in \mcl S^n(X)$, 
\eqb
\sup_{u,v \in \#_{\BB{r} S}} D_h(u,v;\BB{r} \BB S) \leq C^{\xi + \zeta} 2^{-\xi (Q-q-\zeta) n + o_n(n)}  \frk c_{\BB{r}} e^{\xi h_{\BB{r}}(0)},
\label{eqn-internal-hash}
\eqe
with the $o_n(n)$ uniform in $\BB{r}, T,x,X$.
As in the proofs of Proposition~\ref{prop-singular}(b) and the upper bound of Theorem~\ref{thm-kpz-thick}, we consider
the union of hashes \[ H_{\BB{r} \BB{S}}:= \bigcup_{n = 0}^\infty \bigcup_{S \in \mcl S^{n}(X \cap \BB{S})} \#_{\BB{r} S}. \] 
By Lemma~\ref{lem-closure} and the fact that a hash associated to a dyadic square necessarily intersects a hash associated to its dyadic parent,~\eqref{eqn-internal-hash} implies that
\[
\sup_{u,v \in \ol{H_{\BB{r} \BB{S}}}} D_h(u,v;\BB{r} S)
\leq 2
C^{\xi + \zeta}   \frk c_{\BB{r}} e^{\xi h_{\BB{r}}(0)} \sum_{n = 0}^{\infty} 2^{-\xi (Q-q-\zeta) n + o_n(n)} \leq O_C(C^{\xi+\zeta}) \frk c_{\BB{r}} e^{\xi h_{\BB{r}}(0)},
\]
with the $O_C(\cdot)$ uniform in $\BB{r}, T,x,X$.
Since $\ol H_{\BB{r} \BB{S}}$ contains $\BB{r} X$, we have 
$
\sup_{u,v \in \BB{r} X} D_h(u,v;\BB{r} \BB S) 
\leq 
\sup_{u,v \in \ol{H_{\BB{r} \BB{S}}}} D_h(u,v;\BB{r} S)$
.  Thus, setting $\wt C := C^{\xi + \zeta}$, we have
\eqbn
\BB P\left[ \frk c_{ \BB r}^{-1} e^{-\xi h_{\BB r}(0)} \sup_{u,v \in \BB{r} X} D_h(u,v;\BB{r} \BB S)  > \wt C \right] \leq T \wt C^{-(\xi+\zeta)^{-1} (  q + \sqrt{q^2-2x}  ) + o_{\wt C}(1)},
\eqen
with the $o_{\wt C}(1)$  uniform in $\BB{r}, T,x,X$.
For positive $p < \xi^{-1} (  q + \sqrt{q^2-2x}  ) $, we can multiply this last estimate by $\wt C^{p-1}$ and integrate to get the desired $p$-th moment bound~\eqref{eqn-internal-moment}.  The result follows from taking $\zeta>0$ arbitrarily small and $q$ arbitrarily close to $Q$.
\end{proof}

From Proposition~\ref{prop-internal-moment}, we obtain the following necessary ingredient for proving the upper bound of Theorem~\ref{thm-kpz}.

\begin{lem} \label{lem-internal-moment}  
Let $h$ be a whole-plane GFF normalized so that $h_1(0) = 0$.
For each $p \in \BB{R}$ and $\zeta>0$ and every bounded subset $K \subset \BB{C}$, we can choose a deterministic constant $C_{p,K,\zeta}>0$ such that the following is true.
Let $m \in \BB{N}$ and $z \in K$, and let $X \subset 2^{-m} \BB{S} + z$ be a Borel set such that $|\mcl S^n(X)| \leq T 2^{x(n-m)}$ for some $x,T>0$ and all integers $n\geq m$. Then, for every real number $p <  \xi^{-1} (Q + \sqrt{Q^2 - 2x})$ and $\zeta > 0$, we have
\eqb \label{eqn-lem-internal-moment}
\BB E\left[\left( \sup_{u,v \in X} D_h(u,v; 2^{-m} \BB S + z)  \right)^p \right] \leq T C_{p,K,\zeta} 2^{-m\xi Q p + m \xi^2 p^2/2 + m\zeta}.
\eqe
\end{lem}

\begin{proof}
First, observe that we can choose $C_p > 0$ such that, for every $z \in \BB{C}$ and every real number $p <  \xi^{-1} (Q + \sqrt{Q^2 - 2x})$,
\eqb
\BB E\left[\left( \frk c_{2^{-m}}^{-1} e^{-\xi h_{2^{-m}}(0)} \sup_{u,v \in X} D_h(u,v; 2^{-m} \BB S + z)  \right)^p \right] \leq T C_p. 
\label{eqn-lem-internal-moment-1}
\eqe
Indeed, we obtain~\eqref{eqn-lem-internal-moment-1} for $z=0$ by setting $\BB{r} = 2^{-m}$ in Proposition~\ref{prop-internal-moment}; the result then follows for general $z \in \BB C$ by Axiom~\ref{item-metric-translate} and the translation invariance property $h(\cdot + z) - h_{2^{-m}}(z) \eqD h - h_{2^{-m}}(0)$.

Now, by Axioms~\ref{item-metric-local} and~\ref{item-metric-f}, the internal metric
\eqbn
e^{-\xi h_{2^{-m}}(z)} D_h\left( u,v ; 2^{-m} \BB{S} + z \right) 
= D_{h-h_{2^{-m}}(z)} \left( u,v ; 2^{-m} \BB{S} + z \right) 
\eqen
is a.s.\ determined by $(h-h_{2^{-m}}(z))|_{B_{2^{-m}(z)}}$
On the other hand, the random variable $h_1(z) - h_{2^{-m}}(z)$ is independent from $(h-h_{2^{-m}}(z))|_{B_{2^{-m}}(z)}$~\cite[Section 3.1]{shef-kpz}. Therefore, with $p$ as above, we can decompose
\eqb
 \BB E\left[ \left( e^{-\xi h_1(z)} \sup_{u,v\in X}  D_h\left( u, v ; 2^{-m} \BB{S} + z   \right) \right)^p \right]  \label{eqn-lem-internal-moment-2}
\eqe
as a product of expectations
\[
\BB E\left[ e^{\xi p (h_{2^{-m}}(z) - h_1(z))} \right] \BB E\left[ \left( e^{-\xi h_{2^{-m}}(z)} \sup_{u,v\in X} D_h\left( u, v ; 2^{-m} \BB{S} + z   \right) \right)^p \right],
\]
which is at most $2^{-m\xi Q p +m \xi^2 p^2 / 2}  TC_p$ by~\eqref{eqn-lem-internal-moment-1} and the fact that $h_1(z) - h_{2^{-m}}(z)$ is centered Gaussian with variance $m \log 2$~\cite[Section 3.1]{shef-kpz}.  To deduce~\eqref{eqn-lem-internal-moment} from this bound on~\eqref{eqn-lem-internal-moment-2}, we apply H\"older's inequality.  Since $h_1(z)$ is a Gaussian with variance bounded uniformly in $z \in K$, we can choose a constant $C_K$ depending only on $K$ such that, for each $q>1$ sufficiently small,
\allb \label{eqn-diam-moment-holder}
& \BB E\left[ \left(  \sup_{u,v\in X} D_h\left( u, v ; 2^{-m} \BB{S} + z   \right) \right)^p \right] \notag \\
&\qquad \leq \BB E\left[ e^{\xi q h_1(z)/(q-1)} \right]^{1-1/q} \BB E\left[ \left( e^{- \xi h_1(z)}  \sup_{u,v\in X} D_h\left( u, v ; 2^{-m} \BB{S} + z   \right) \right)^{q p} \right]^{1/q}  \notag \\
&\qquad= C_K^{q/(q-1)} 2^{-m\xi Q p   +m \xi^2 p^2 q / 2} \changes{ (T C_{pq})^{1/q}}.
\alle  
The result follows from setting $\zeta = \frac{\xi^2 p^2}{2} (q-1)$.
\end{proof}

\subsection{Completing the proof of the KPZ upper bound}

\begin{proof}[Proof of Theorem~\ref{thm-kpz}, upper bound]
As we explained after stating Theorem~\ref{thm-kpz}, it suffices to prove the upper bound of Theorem~\ref{thm-kpz} with packing dimension replaced by upper Minkowski dimension.  In other words, we assume that the upper Minkowski dimension of $X$ is less than some $x > 0$, and we show that
\eqb
\dim_{\mcl H}(X; D_h) \leq f(x),
\label{eqn-kpz-goal}
\eqe
with $f$ defined in the statement of the theorem.  
So, suppose that the upper Minkowski dimension of $X$ is less than $x$. We may assume that $x < Q^2/2$, so that $f(x) = \frac{1}{\xi} (Q - \sqrt{Q^2 - 2x})$, since otherwise the bound~\eqref{eqn-kpz-goal} is trivial.  By Weyl scaling (Axiom~\ref{item-metric-f}), we can assume without loss of generality that $h$ is a whole-plane GFF normalized so that $h_1(0) = 0$.

By definition of Minkowski dimension, we can choose  \changes{$\wt T>1$} such that $
|\mcl S^n(X)| \leq \wt T 2^{nx}$ for each $n \in \BB{N}$.
Observe, for integers $m \leq n$, every square in $\mcl S^n$ intersects at most nine squares of $\mcl S^m$ (or four if $m<n$). We deduce that, with $T = 9\wt T$,
\eqb
\sum_{S \in \mcl S^m} |\mcl S^n(X \cap S)| \leq T 2^{nx} \qquad \forall m \leq n \in \BB{N}
\label{eqn-kpz-mink}
\eqe
Fix $\zeta>0$ and $m \in \BB{N}$. For each $A>0$,
\[
\left|\left\{ S \in \mcl S^m(X) :  |\mcl S^n(X \cap S)| > A T  2^{(n-m)(x+\zeta)} \right\}\right| \leq A^{-1}   2^{m(x+\zeta) - n \zeta}  \qquad \forall n \geq m
\]
so summing over all integers $n \geq m$ yields
\eqb
\label{eqn-kpz-mink-m}
\left| \left\{ S \in \mcl S^m(X) : \text{ $|\mcl S^n(X \cap S)| > A T  2^{(n-m)(x+\zeta)}$ for some $n \geq m$} \right\} \right| \leq C_\zeta A^{-1}   2^{m(x+\zeta)}
\eqe
for some constant $C_\zeta>1$ depending only on $\zeta$.

We can rephrase~\eqref{eqn-kpz-mink-m} as follows.  For each $S \in \mcl S^m(X)$, we set
\eqb
\label{eqn-TS-def}
T_S = \sup_{n \geq m} 2^{(m-n)(x+\zeta)} |\mcl S^n(X \cap S)|.
\eqe
\changes{For all $S \in \mcl S^m(X)$, we have 
\[ T_S = 
\sup_{n \geq m} 2^{(m-n)(x+\zeta)} |\mcl S^n(X \cap S)| 
\leq 
2^{m(x+\zeta)} \sup_{n \geq m} 2^{-nx} |\mcl S^n(X \cap S)| 
\leq
2^{m(x+\zeta)} T,
\]
where the last inequality follows from~\eqref{eqn-kpz-mink}. Thus,} we can partition the set of possible values of $T_S$ into the finite collection of intervals $I_k$ for $k = 1,\ldots,\lceil m (x+\zeta) \rceil$, where $I_1 = (-\infty, 2 T]$ and $I_k = (2^{k-1} T, 2^{k} T]$ for $k > 1$.   Then~\eqref{eqn-kpz-mink-m}  implies that, for $k>1$, 
\[
\left| \left\{ S \in \mcl S^m(X) : T_S \in I_k \right\} \right| \leq C_\zeta 2^{-k+1+m(x+\zeta)}.
\]
Since  \changes{$T>1$ and} the total number of squares $S \in \mcl S^m(X)$ is at most $T 2^{m(x+\zeta)}$, we deduce that, for all $k$ including $k=1$,
\eqb
\left| \left\{ S \in \mcl S^m(X) : T_S \in I_k \right\} \right| \leq C_\zeta T 2^{-k+1+m(x+\zeta)}.
\label{eqn-kpz-mink-Ik}
\eqe
Now, for each $S \in \mcl S^{m}(X)$, Lemma~\ref{lem-internal-moment} with $p = \xi^{-1}(Q - \sqrt{Q^2 - 2(x+3\zeta)})$ gives
\eqbn
\BB E\left[\left( \sup_{u,v \in X \cap S} D_h(u,v; S)  \right)^p \right] \leq T_S C_{X,x,\zeta} 2^{-m\xi Q p + m \xi^2 p^2/2 + m\zeta} = T_S C_{X,x,\zeta} 2^{-m(x+2\zeta)} 
\eqen
for some constant $C_{X,x,\zeta}>0$ depending on $X,x,\zeta$ but not on $S,m$.  (In the last step, we assume that $\zeta>0$ is small enough so that $p$ is real.) For each $k = 1,\ldots,\lceil m (x+\zeta)\rceil$, we can sum this bound over all $S$ with $T_S \in I_k$. By~\eqref{eqn-kpz-mink-Ik}, this gives
\[
\sum_{S \in \mcl S^{m}(X) \, : \, T_S \in I_k} \BB E\left[\left( \sup_{u,v \in X \cap S} D_h(u,v; S)  \right)^p \right]
\leq
T^2  2^{-m \zeta} C_{X,x,\zeta}.\]
Summing over $k$ yields
\[
\sum_{S \in \mcl S^{m}(X)} \BB E\left[\left( \sup_{u,v \in X \cap S} D_h(u,v; S)  \right)^p \right]
\leq
T^2 \lceil m (x+\zeta)\rceil 2^{-m \zeta} C_{X,x,\zeta}.\]
By Markov's inequality, this means that for any fixed $\ep>0$, the probability that
\eqb
\sum_{S \in \mcl S^{m}(X)}  \left( \sup_{u,v \in X \cap S} D_h(u,v; S)  \right)^p < \ep 
\label{eqn-kpz-markov}
\eqe
tends to 1 as $m \rta \infty$.  By the Borel-Cantelli lemma, a.s.\ ~\eqref{eqn-kpz-markov} holds for arbitrarily small values of $m$.  This implies that the Hausdorff dimension of $X$ in the metric $D_h$ is at most $p$.  Sending $\zeta \rta 0$ gives~\eqref{eqn-kpz-goal}.
\end{proof}

\newcommand{\etalchar}[1]{$^{#1}$}
\def\cprime{$'$}


\begin{thebibliography}{DFG{\etalchar{+}}20}

\bibitem[ADF86]{adf-critical-dimensions}
J.~Ambj{\o}rn, B.~Durhuus, and J.~Fr\"{o}hlich.
\newblock The appearance of critical dimensions in regulated string theories.
  {II}.
\newblock {\em Nuclear Phys. B}, 275(2):161--184, 1986. \MR{858659}

\bibitem[ADJT93]{adjt-c-ge1}
J.~{Ambj{\o}rn}, B.~{Durhuus}, T.~{J{\'o}nsson}, and G.~{Thorleifsson}.
\newblock {Matter fields with c $>$ 1 coupled to 2d gravity}.
\newblock {\em Nuclear Physics B}, 398:568--592, June 1993,
  \arxiv{hep-th/9208030}.

\bibitem[Ald91a]{aldous-crt1}
D.~Aldous.
\newblock The continuum random tree. {I}.
\newblock {\em Ann. Probab.}, 19(1):1--28, 1991. \MR{1085326 (91i:60024)}

\bibitem[Ald91b]{aldous-crt2}
D.~Aldous.
\newblock The continuum random tree. {II}. {A}n overview.
\newblock In {\em Stochastic analysis ({D}urham, 1990)}, volume 167 of {\em
  London Math. Soc. Lecture Note Ser.}, pages 23--70. Cambridge Univ. Press,
  Cambridge, 1991. \MR{1166406 (93f:60010)}

\bibitem[Ald93]{aldous-crt3}
D.~Aldous.
\newblock The continuum random tree. {III}.
\newblock {\em Ann. Probab.}, 21(1):248--289, 1993. \MR{1207226 (94c:60015)}

\bibitem[Amb94]{ambjorn-remarks}
J.~Ambj{\o}rn.
\newblock Remarks about {$c>1$} and {$D>2$}.
\newblock {\em Teoret. Mat. Fiz.}, 98(3):326--336, 1994. \MR{1304731}

\bibitem[Aru15]{aru-kpz}
J.~Aru.
\newblock K{PZ} relation does not hold for the level lines and {SLE$_\kappa$}
  flow lines of the {G}aussian free field.
\newblock {\em Probab. Theory Related Fields}, 163(3-4):465--526, 2015,
  \arxiv{1312.1324}. \MR{3418748}

\bibitem[{Aru}17]{aru-gmc-survey}
J.~{Aru}.
\newblock {Gaussian multiplicative chaos through the lens of the 2D Gaussian
  free field}.
\newblock {\em ArXiv e-prints}, Sep 2017, \arxiv{1709.04355}.

\bibitem[Bee82]{beer-usc}
G.~Beer.
\newblock Upper semicontinuous functions and the {S}tone approximation theorem.
\newblock {\em J. Approx. Theory}, 34(1):1--11, 1982. \MR{647707}

\bibitem[Bef08]{beffara-dim}
V.~Beffara.
\newblock The dimension of the {SLE} curves.
\newblock {\em Ann. Probab.}, 36(4):1421--1452, 2008, \arxiv{math/0211322}.
  \MR{2435854 (2009e:60026)}

\bibitem[Ber17]{berestycki-gmt-elementary}
N.~Berestycki.
\newblock An elementary approach to {G}aussian multiplicative chaos.
\newblock {\em Electron. Commun. Probab.}, 22:Paper No. 27, 12, 2017,
  \arxiv{1506.09113}. \MR{3652040}

\bibitem[BGRV16]{grv-kpz}
N.~Berestycki, C.~Garban, R.~Rhodes, and V.~Vargas.
\newblock K{PZ} formula derived from {L}iouville heat kernel.
\newblock {\em J. Lond. Math. Soc. (2)}, 94(1):186--208, 2016,
  \arxiv{1406.7280}. \MR{3532169}

\bibitem[BH92]{bh-c-ge1-matrix}
E.~{Br{\'e}zin} and S.~{Hikami}.
\newblock {A naive matrix-model approach to 2D quantum gravity coupled to
  matter of arbitrary central charge}.
\newblock {\em Physics Letters B}, 283:203--208, June 1992,
  \arxiv{hep-th/9204018}.

\bibitem[BJ92]{bj-potts-sim}
C.~F. {Baillie} and D.~A. {Johnston}.
\newblock {A Numerical Test of Kpz Scaling:. Potts Models Coupled to
  Two-Dimensional Quantum Gravity}.
\newblock {\em Modern Physics Letters A}, 7:1519--1533, 1992,
  \arxiv{hep-lat/9204002}.

\bibitem[BJRV13]{bjrv-gmt-duality}
J.~Barral, X.~Jin, R.~Rhodes, and V.~Vargas.
\newblock Gaussian multiplicative chaos and {KPZ} duality.
\newblock {\em Comm. Math. Phys.}, 323(2):451--485, 2013, \arxiv{1202.5296}.
  \MR{3096527}

\bibitem[BS09]{benjamini-schramm-cascades}
I.~Benjamini and O.~Schramm.
\newblock K{PZ} in one dimensional random geometry of multiplicative cascades.
\newblock {\em Comm. Math. Phys.}, 289(2):653--662, 2009, \arxiv{0806.1347}.
  \MR{2506765 (2010c:60151)}

\bibitem[{Cat}88]{cates-branched-polymer}
M.~E. {Cates}.
\newblock {The Liouville field theory of random surfaces: when is the bosonic
  string a branched polymer?}
\newblock {\em EPL (Europhysics Letters)}, 7:719, December 1988.

\bibitem[CKR92]{ckr-c-ge1}
S.~{Catterall}, J.~{Kogut}, and R.~{Renken}.
\newblock {Numerical study of c $>$ 1 matter coupled to quantum gravity}.
\newblock {\em Physics Letters B}, 292:277--282, 1992.

\bibitem[Dav88]{david-conformal-gauge}
F.~David.
\newblock Conformal field theories coupled to {2-D} gravity in the conformal
  gauge.
\newblock {\em {M}od. {P}hys. {L}ett. {A}}, 3(17), 1988.

\bibitem[{Dav}97]{david-c>1-barrier}
F.~{David}.
\newblock {A scenario for the $c > 1$ barrier in non-critical bosonic strings}.
\newblock {\em Nuclear Physics B}, 487:633--649, February 1997,
  \arxiv{hep-th/9610037}.

\bibitem[DDDF19]{dddf-lfpp}
J.~Ding, J.~Dub{\'e}dat, A.~Dunlap, and H.~Falconet.
\newblock {Tightness of Liouville first passage percolation for $\gamma \in
  (0,2)$}.
\newblock {\em ArXiv e-prints}, Apr 2019, \arxiv{1904.08021}.

\bibitem[DF20]{df-lqg-metric}
J.~Dub\'{e}dat and H.~Falconet.
\newblock Liouville metric of star-scale invariant fields: tails and {W}eyl
  scaling.
\newblock {\em Probab. Theory Related Fields}, 176(1-2):293--352, 2020,
  \arxiv{1809.02607}. \MR{4055191}

\bibitem[DFG{\etalchar{+}}20]{lqg-metric-estimates}
J.~Dub\'{e}dat, H.~Falconet, E.~Gwynne, J.~Pfeffer, and X.~Sun.
\newblock Weak {LQG} metrics and {L}iouville first passage percolation.
\newblock {\em Probab. Theory Related Fields}, 178(1-2):369--436, 2020,
  \arxiv{1905.00380}. \MR{4146541}

\bibitem[DFJ84]{dfj-critical-behavior}
B.~Durhuus, J.~Frohlich, and T.~Jonsson.
\newblock {Critical Behavior in a Model of Planar Random Surfaces}.
\newblock {\em Nucl. Phys.}, B240:453, 1984.
\newblock [Phys. Lett.137B,93(1984)].

\bibitem[DG18]{dg-lqg-dim}
J.~{Ding} and E.~{Gwynne}.
\newblock {The fractal dimension of {L}iouville quantum gravity: universality,
  monotonicity, and bounds}.
\newblock {\em {C}ommunications in {M}athematical {P}hysics}, 374:1877--1934,
  2018, \arxiv{1807.01072}.

\bibitem[DG20]{dg-supercritical-lfpp}
J.~{Ding} and E.~{Gwynne}.
\newblock {Tightness of supercritical Liouville first passage percolation}.
\newblock {\em ArXiv e-prints}, May 2020, \arxiv{2005.13576}.




\bibitem[DG21a]{dg-confluence}
J.~{Ding} and E.~{Gwynne}.
\newblock {Regularity and confluence of geodesics for the supercritical
  Liouville quantum gravity metric}.
\newblock {\em ArXiv e-prints}, April 2021, \arxiv{2104.06502}.

\bibitem[DG21b]{dg-uniqueness}
J.~{Ding} and E.~{Gwynne}.
\newblock {Uniqueness of the critical and supercritical Liouville quantum
  gravity metrics}.
\newblock {\em ArXiv e-prints}, September 2021, \arxiv{2110.00177}.

\bibitem[DG21c]{dg-polylog}
J.~{Ding} and E.~{Gwynne}.
\newblock {Up-to-constants comparison of Liouville first passage percolation and Liouville quantum gravity}.
\newblock {\em ArXiv e-prints}, August 2021, \arxiv{2108.12060}.

\bibitem[DG21d]{dg-critical-lqg}
J.~{Ding} and E.~{Gwynne}.
\newblock {The critical Liouville quantum gravity metric induces the Euclidean topology}.
\newblock {\em ArXiv e-prints}, August 2021, \arxiv{2108.12067}.

\bibitem[DK89]{dk-qg}
J.~Distler and H.~Kawai.
\newblock Conformal field theory and {2D} quantum gravity.
\newblock {\em {N}ucl.{P}hys. {B}}, 321(2), 1989.

\bibitem[DMS14]{wedges}
B.~{Duplantier}, J.~{Miller}, and S.~{Sheffield}.
\newblock {Liouville quantum gravity as a mating of trees}.
\newblock {\em {Asterisque}}, 427, 2021, \arxiv{1409.7055}. 
\MR{4340069}

\bibitem[DRSV14]{shef-renormalization}
B.~Duplantier, R.~Rhodes, S.~Sheffield, and V.~Vargas.
\newblock Renormalization of critical {G}aussian multiplicative chaos and {KPZ}
  relation.
\newblock {\em Comm. Math. Phys.}, 330(1):283--330, 2014, \arxiv{1212.0529}.
  \MR{3215583}

\bibitem[DRV16]{drv-torus}
F.~David, R.~Rhodes, and V.~Vargas.
\newblock Liouville quantum gravity on complex tori.
\newblock {\em J. Math. Phys.}, 57(2):022302, 25, 2016, \arxiv{1504.00625}.
  \MR{3450564}

\bibitem[DS11]{shef-kpz}
B.~Duplantier and S.~Sheffield.
\newblock Liouville quantum gravity and {KPZ}.
\newblock {\em Invent. Math.}, 185(2):333--393, 2011, \arxiv{1206.0212}.
  \MR{2819163 (2012f:81251)}

\bibitem[Fal03]{falconer}
K.~Falconer.
\newblock {\em Fractal geometry: mathematical foundations and applications}.
\newblock John Wiley \& Sons, West Sussex, England, second edition, 2003.

\bibitem[GHM20]{ghm-kpz}
E.~Gwynne, N.~Holden, and J.~Miller.
\newblock An almost sure {KPZ} relation for {SLE} and {B}rownian motion.
\newblock {\em Ann. Probab.}, 48(2):527--573, 2020, \arxiv{1512.01223}.
  \MR{4089487}

\bibitem[GHPR20]{ghpr-central-charge}
E.~Gwynne, N.~Holden, J.~Pfeffer, and G.~Remy.
\newblock Liouville quantum gravity with matter central charge in (1, 25): a
  probabilistic approach.
\newblock {\em Comm. Math. Phys.}, 376(2):1573--1625, 2020, \arxiv{1903.09111}.
  \MR{4103975}

\bibitem[GHS19]{ghs-dist-exponent}
E.~{Gwynne}, N.~{Holden}, and X.~{Sun}.
\newblock {A distance exponent for Liouville quantum gravity}.
\newblock {\em {Probability Theory and Related Fields}}, 173(3):931--997, 2019,
  \arxiv{1606.01214}.

\bibitem[GM17]{gwynne-miller-char}
E.~{Gwynne} and J.~{Miller}.
\newblock {Characterizations of SLE$_{\kappa}$ for $\kappa \in (4,8)$ on
  Liouville quantum gravity}.
\newblock {\em {A}st{\'e}risque}, to appear, 2017, \arxiv{1701.05174}.

\bibitem[GM19a]{gm-coord-change}
E.~Gwynne and J.~Miller.
\newblock Conformal covariance of the {L}iouville quantum gravity metric for
  {$\gamma \in (0,2)$}.
\newblock {\em {Annales de l'Institut Henri Poincar{\'e}}}, 57(2):1016--1031, 2021,
  \arxiv{1905.00384}.
\MR{4260493}

\bibitem[GM19b]{local-metrics}
E.~Gwynne and J.~Miller.
\newblock Local metrics of the {G}aussian free field.
\newblock {\em {A}nnales de {l'Institut} {F}ourier}, 70(5):2049--2075, 2020,
  \arxiv{1905.00379}.
  \MR{4245606}

\bibitem[GM20]{gm-confluence}
E.~Gwynne and J.~Miller.
\newblock Confluence of geodesics in {L}iouville quantum gravity for {$\gamma
  \in (0,2)$}.
\newblock {\em Ann. Probab.}, 48(4):1861--1901, 2020, \arxiv{1905.00381}.
  \MR{4124527}

\bibitem[GM21]{gm-uniqueness}
E.~Gwynne and J.~Miller.
\newblock Existence and uniqueness of the {L}iouville quantum gravity metric
  for {$\gamma\in(0,2)$}.
\newblock {\em Invent. Math.}, 223(1):213--333, 2021, \arxiv{1905.00383}.
  \MR{4199443}

\bibitem[GMS19]{gms-harmonic}
E.~{Gwynne}, J.~{Miller}, and S.~{Sheffield}.
\newblock {Harmonic functions on mated-{CRT} maps}.
\newblock {\em Electron. J. Probab.}, 24:no. 58, 55, 2019, \arxiv{1807.07511}.

\bibitem[GP19]{gp-kpz}
E.~{Gwynne} and J.~{Pfeffer}.
\newblock {KPZ formulas for the Liouville quantum gravity metric}.
\newblock {\em {T}ransactions of the {A}merican {M}athematical {S}ociety}, to
  appear, 2019.

\bibitem[GPS20]{lqg-zero-one}
E.~{Gwynne}, J.~{Pfeffer}, and S.~{Sheffield}.
\newblock {Geodesics and metric ball boundaries in Liouville quantum gravity}.
\newblock {\em ArXiv e-prints}, October 2020, \arxiv{2010.07889}.

\bibitem[GRV19]{grv-higher-genus}
C.~Guillarmou, R.~Rhodes, and V.~Vargas.
\newblock Polyakov's formulation of {$2d$} bosonic string theory.
\newblock {\em Publ. Math. Inst. Hautes \'{E}tudes Sci.}, 130:111--185, 2019,
  \arxiv{1607.08467}. \MR{4028515}

\bibitem[Gwy20a]{gwynne-ball-bdy}
E.~Gwynne.
\newblock The {D}imension of the {B}oundary of a {L}iouville {Q}uantum
  {G}ravity {M}etric {B}all.
\newblock {\em Comm. Math. Phys.}, 378(1):625--689, 2020, \arxiv{1909.08588}.
  \MR{4124998}


\bibitem[HMP10]{hmp-thick-pts}
X.~Hu, J.~Miller, and Y.~Peres.
\newblock Thick points of the {G}aussian free field.
\newblock {\em Ann. Probab.}, 38(2):896--926, 2010, \arxiv{0902.3842}.
  \MR{2642894 (2011c:60117)}




\bibitem[Kah85]{kahane}
J.-P. Kahane.
\newblock Sur le chaos multiplicatif.
\newblock {\em Ann. Sci. Math. Qu\'ebec}, 9(2):105--150, 1985. \MR{829798
  (88h:60099a)}
  
  
\bibitem[LR15]{lawler-rezai-nat}
G.~F. Lawler and M.~A. Rezaei.
\newblock Minkowski content and natural parameterization for the
  {S}chramm-{L}oewner evolution.
\newblock {\em Ann. Probab.}, 43(3):1082--1120, 2015, \arxiv{1211.4146}.
  \MR{3342659}
  
  
\bibitem[Rem18]{remy-annulus}
G.~Remy.
\newblock Liouville quantum gravity on the annulus.
\newblock {\em J. Math. Phys.}, 59(8):082303, 26, 2018, \arxiv{1711.06547}.
  \MR{3843631}

\bibitem[RS05]{schramm-sle}
S.~Rohde and O.~Schramm.
\newblock Basic properties of {SLE}.
\newblock {\em Ann. of Math. (2)}, 161(2):883--924, 2005, \arxiv{math/0106036}.
  \MR{2153402 (2006f:60093)}

\bibitem[RV11]{rhodes-vargas-log-kpz}
R.~Rhodes and V.~Vargas.
\newblock K{PZ} formula for log-infinitely divisible multifractal random
  measures.
\newblock {\em ESAIM Probab. Stat.}, 15:358--371, 2011, \arxiv{0807.1036}.
  \MR{2870520}

\bibitem[RV14]{rhodes-vargas-review}
R.~Rhodes and V.~Vargas.
\newblock Gaussian multiplicative chaos and applications: {A} review.
\newblock {\em Probab. Surv.}, 11:315--392, 2014, \arxiv{1305.6221}.
  \MR{3274356}

\bibitem[She07]{shef-gff}
S.~Sheffield.
\newblock Gaussian free fields for mathematicians.
\newblock {\em Probab. Theory Related Fields}, 139(3-4):521--541, 2007,
  \arxiv{math/0312099}. \MR{2322706 (2008d:60120)}

\bibitem[SS13]{ss-contour}
O.~Schramm and S.~Sheffield.
\newblock A contour line of the continuum {G}aussian free field.
\newblock {\em Probab. Theory Related Fields}, 157(1-2):47--80, 2013,
  \arxiv{1008.2447}. \MR{3101840}

\bibitem[WP20]{pw-gff-notes}
W.~{Werner} and E.~{Powell}.
\newblock {Lecture notes on the Gaussian Free Field}.
\newblock {\em ArXiv e-prints}, April 2020, \arxiv{2004.04720}.



\end{thebibliography}
\end{document}